\documentclass[12pt,reqno]{amsart}
\usepackage{hyperref}
\usepackage{amsmath,amssymb,amscd,amsthm,amscd,amsthm,amsfonts,mathrsfs,array}
\usepackage{graphicx,color}
\usepackage[all]{xy}
\usepackage{tikz}
\usepackage{multirow}
\usepackage{here, latexsym, bm, ascmac, mathtools, multicol, tcolorbox, subfig}
\usetikzlibrary{intersections, calc, arrows.meta}
\usetikzlibrary{decorations, decorations.markings}
\usetikzlibrary{intersections,calc,arrows.meta}

\newtheorem{thm}{Theorem}[section]
\newtheorem{lem}{Lemma}[section]
\newtheorem{prop}{Proposition}[section]
\newtheorem{cor}{Corollary}[section]

\theoremstyle{definition}

\newtheorem{defn}{Definition}[section]
\newtheorem{rem}{Remark}[section]

\DeclareMathOperator{\grad}{grad}

\begin{document}

\title[More on Holomorphic Convex Quadrilaterals]{More on explicit Correspondence between Gradient Trees in $\mathbb{R}$ and Holomorphic Convex Quadrilaterals in $T^{*}\mathbb{R}$}
\author{Hidemasa SUZUKI}
\address{Department of Mathematics and Informatics, Graduate School of Science and Engineering, Chiba University,
Yayoicho 1-33, Inage, Chiba, 263-8522 Japan.}
\email{hsuzuki@g.math.s.chiba-u.ac.jp}
\date{}

\begin{abstract}
For given smooth functions $(f_1,\dots,f_n)$ on $M$, Fukaya and Oh showed that the moduli space of pseudoholomorphic disks in $T^*M$ which are bounded by Lagrangian sections $\{L_i^\epsilon=\operatorname{graph}(\epsilon df_i)\}$ is diffeomorphic to the moduli space of gradient trees in $M$ which consist of gradient curves of $\{f_i-f_j\}$. When the image of the pseudoholomorphic disk $w_\epsilon$ is a polygon in $\mathbb{C}\simeq T^*\mathbb{R}$, we can describe $w_\epsilon$ by a Schwarz-Christoffel map. In \cite{S25}, we proved that pseudoholomorphic disks $w_\epsilon$ converge to the gradient tree in the limit $\epsilon\to+0$ when the image of $w_\epsilon$ is a generic convex quadrilateral. In this paper, we show such a convergence for any convex quadrilaterals by studying the non-generic case.
\end{abstract} 

\maketitle

\setcounter{tocdepth}{3}
\tableofcontents

\section{Introduction.}
This paper is a continuation of \cite{S25}, where we studied the correspondence between pseudoholomorphic disks in $T^*\mathbb{R}$ and gradient trees in $\mathbb{R}$ when the image of the pseudoholomorphic disk is a generic convex quadrilateral bounded by affine Lagrangian sections of $T^*\mathbb{R}$. Our goal is to complete this study for any convex quadrilateral.

Fukaya and Oh proved that the moduli space of pseudoholomorphic disks $\mathcal{M}_{J}(T^{*}M;\vec{L^{\epsilon}},\vec{x^{\epsilon}})$ bounded by Lagrangian sections $\vec{L^\epsilon}=(L_{1}^{\epsilon},L_{2}^{\epsilon},\dots,L_{k}^{\epsilon})$ of the cotangent bundle $T^{*}M$ is diffeomorphic to the moduli space of gradient trees $\mathcal{M}_{g}(M;\vec{f},\vec{p})$ constructed by functions $\vec{f}=(f_{1},f_{2},\dots,f_{k})$ on the Riemannian manifold $M$ for sufficiently small $\epsilon>0$ (\cite{FO97}). Here, $L_{i}^{\epsilon}$ is defined by $L_{i}^{\epsilon}\coloneqq\operatorname{graph}(\epsilon df_{i})$ for $i=1,2,\dots,k$. Pseudoholomorphic disks are pseudoholomorphic maps from the unit disk $D^{2}$ to $T^{*}M$. Fukaya and Oh constructed them approximately in $T^{*}M$ first, and show the existence of an exact solution in a neighborhood of the approximate one. Since we assume $M=\mathbb{R}$ in this paper, pseudoholomorphic disks are holomorphic maps from $D^{2}$. Hereafter, we call elements of $\mathcal{M}_{J}(T^{*}\mathbb{R};\vec{L^{\epsilon}},\vec{x^{\epsilon}})$ holomorphic disks. We consider holomorphic disks as maps from the closure of the upper half plane $\mathbb{H}$ instead of $D^{2}$ to make later analysis easier in this paper. When Lagrangian sections of $T^{*}\mathbb{R}\simeq\mathbb{C}$ are affine, holomorphic disks can be described by Schwarz-Christoffel maps. Here, a Schwarz-Christoffel map is a conformal map from $\mathbb{H}$ to a polygonal domain in $\mathbb{C}$. On the other side, gradient trees are continuous maps from trees to $M$, and they map each edges to gradient curves of $f_{i}-f_{j}\,(i\neq j)$. Since each Lagrangian section $L_{i}^{\epsilon}$ of $T^{*}\mathbb{R}$ is affine, each function $f_{i}$ is described by $f_{i}(x)=a_{i}x^{2}+b_{i}x+c_{i}\,(a_{i},b_{i},c_{i}\in\mathbb{R})$ for $i=1,2,\dots,k$.
In this paper, we describe holomorphic disks and gradient trees explicitly, and show the correspondence between gradient trees and pseudoholomorphic disks in the case $M=\mathbb{R}$ and $k=4$. In particular, we prove that the holomorphic disk $w_{\epsilon}$ we consider below converges to the gradient tree.

In the case $M=\mathbb{R},k=4$, we first set $f_{1},f_{2},f_{3},f_{4}$ such that there exists a convex quadrilateral $x_{1}^{\epsilon}x_{2}^{\epsilon}x_{3}^{\epsilon}x_{4}^{\epsilon}$ which has vertices $x_{1}^{\epsilon},x_{2}^{\epsilon},x_{3}^{\epsilon},x_{4}^{\epsilon}$ in counterclockwise order. In this situation, we checked that there exists the gradient tree uniquely in \cite{S25}. Let $w_{\epsilon}$ be the Schwarz-Christoffel map from the upper half plane with four marked points $z_{1},z_{2},z_{3},z_{4,\epsilon}$ to the convex quadrilateral $x_{1}^{\epsilon}x_{2}^{\epsilon}x_{3}^{\epsilon}x_{4}^{\epsilon}$ such that $w_{\epsilon}(z_{i})=x_{i}^{\epsilon}$ for $i=1,2,3$. Here, we fix $z_{1}=1,z_{2}=\infty,z_{3}=0\in\mathbb{R}$, and $z_{4,\epsilon}\in (0,1)$ moves as $\epsilon$ varies.

In \cite{S25}, we studied the case the convex quadrilateral $x_1^\epsilon x_2^\epsilon x_3^\epsilon x_4^\epsilon$ is generic in the sense that $L_{i}\cap L_{j}\neq \emptyset,L_{i}\neq L_{j}\,(i\neq j)$ and $p_{1}\neq p_{3},p_{2}\neq p_{4}$. In this case, the holomorphic disk is described by Appell's hypergeometric series $F_{1}$ in the neighborhoods of $0,z_{4,\epsilon},1,\infty$. We also found that the holomorphic disk $w_\epsilon$ is described by Horn's hypergeometric series $G_{2}$ in the domain $\{z\in\overline{\mathbb{H}}\mid z_{4,\epsilon}<\left\lvert z\right\rvert<1\}$. We used the formula of analytic continuation for $F_{1}$ to proof this. We then divided the upper half plane appropriately by studying the behavior of $z_{4,\epsilon}$. We used conformal moduli which is a conformal invariant of quadrilaterals and the ratio $\left\lvert x_{1}^{\epsilon}-x_{2}^{\epsilon}\right\rvert/\left\lvert x_{3}^{\epsilon}-x_{4}^{\epsilon}\right\rvert$ or $\left\lvert x_{2}^{\epsilon}-x_{3}^{\epsilon}\right\rvert/\left\lvert x_{4}^{\epsilon}-x_{1}^{\epsilon}\right\rvert$ in the convex quadrilateral $x_{1}^{\epsilon}x_{2}^{\epsilon}x_{3}^{\epsilon}x_{4}^{\epsilon}$. We used the inequality which is called Rengel's inequality to evaluate the conformal modulus. This inequality is described by the area of the quadrilateral and the infimum length of Jordan arc between the opposite sides of the quadrilateral. On the other hand, we calculated the ratio of opposite side lengths in two ways. We first used the Euclidean metric of $\mathbb{R}^2$, and next used the integral representation of a Schwarz-Christoffel map. By using the latter way, the ratio is described by gamma functions and hypergeometric functions. 

In this paper, we study the remaining non-generic cases. First, in the case $p_1=p_3$ or $p_2=p_4$, we can not calculate the limit value of the conformal modulus enough by using the same method as in \cite{S25}. For example, when the image of the pseudoholomorphic disk is a parallelogram, the upper bound of the conformal modulus obtained from Rengel's inequality diverges to positive infinity, and the lower bound converges to zero as $\epsilon\to+0$. In this paper, we prove the existence of the limit value by applying some lemmas which state the monotonicity of the conformal modulus. Then the limit value is determined by using the equation which is derived from two different ways of calculating the ratio of adjacent sides of the parallelogram. Consequently, we obtain that the limit of the conformal modulus is $1$, which implies that $z_{4,\epsilon}$ converges to $1/2$. When the image of the pseudoholomorphic disk is not a parallelogram, we use the monotonicity lemma to bound its conformal modulus from above and below by those of parallelograms. By the squeeze theorem, we again obtain that the limit of the conformal modulus is $1$, which leads to the convergence of $z_{4,\epsilon}$ to $1/2$. From the above discussion, we need to modify the method of dividing the upper half plane in \cite{S25}. We will give the details of the method in subsection 3.1.

Second, in the case where a convex quadrilateral has a pair of parallel opposite sides, there are two main problems. The first problem is that the connection formula for $F_1$ used in \cite{S25} can not applied directly in some cases. This is because the quotient of gamma functions in the connection formula may have a numerator whose argument vanishes regardless of $\epsilon>0$. In this paper, we derive the connection formula for $F_1$ by using the well-known connection formula for Gauss's hypergeometric function $_2F_1$. Then the derived formula involves terms which contain logarithms and the digamma function. We will give the detail of this formula in subsection 3.2. The second problem is that the limit of $z_{4,\epsilon}$ cannot be determined by the method in \cite{S25}. In \cite{S25}, we calculated the ratio of the lengths of opposite sides in two ways and derived an equation. When we use the integral representation of a Schwarz-Christoffel mapping, a real power of $z_{4,\epsilon}$ appears. By solving the equation for this power and taking the limit, we obtained the behavior of $\log z_{4,\epsilon}$ or $\log(1-z_{4,\epsilon})$. However, in non-generic case, the exponent may become zero regardless of $\epsilon$. In this paper, we compute the behavior of $\log z_{4,\epsilon}$ or $\log(1-z_{4,\epsilon})$ by deriving the connection formula for $F_1$ in two different ways and comparing the coefficients to modify the equation in \cite{S25}. We will give the details in subsection 3.3. We thus complete any non-generic cases.

The outline of this paper is the following. In Section 2, we first recall definitions of gradient trees and pseudoholomorphic disks, while these definitions are also written in our previous paper \cite{S25}. This is because we do not show the explicit forms of gradient trees in this paper. In order to describe pseudoholomorphic disks explicitly, we next explain the notion of a Schwarz-Christoffel map. We next prepare hypergeometric functions for expanding a Schwarz-Christoffel map in power series and describing the ratio of sides of quadrilaterals. Schwarz-Christoffel maps are written by an integration of power functions. We also prepare one of the connection formulas of hypergeometric functions for describing pseudoholomorphic disks in the regions which are little far from preimage of vertices of polygons. We next recall the notion of conformal moduli of quadrilaterals which are conformal quantities for the unit disk (or the upper half plane) with marked points. We also discuss an inequation for conformal moduli which we use to calculate the limit value of $z_{4,\epsilon}$. At the end of this section, we recall some results for gradient trees and pseudoholomorphic disks which are proved in \cite{S25}. In Section 3, we first explain how to divide the upper half plane into some regions, and give conformal transformation between stripes and these regions. Thereafter, we give the main theorems in the case $M=\mathbb{R},k=4$ and the quadrilaterals are not generic. In order to show these main theorems, we first study power series representation of pseudoholomorphic disks by using integral representations and connection formulas for hypergeometric functions. Then we study the limit value of $z_{4,\epsilon}$ and the principal part of $z_{4,\epsilon}$ at $\epsilon\to+0$ in the case the convex quadrilateral is not generic to know how to divide the upper half plane into several regions.
\par
\noindent
{\bf Acknowledgments.}
The author is grateful to the advisor, Hiroshige Kajiura, for sharing his insights and for valuable advice. This work was supported by JST SPRING, Grant Number JPMJSP2109.

\section{Preliminaries.}
\subsection{Gradient trees and pseudoholomorphic disks.}
In this subsection, we review two moduli spaces of our concern; one is of gradient trees and the other is of pseudoholomorphic disk (see \cite{FO97}). The moduli space of gradient trees is defined by the moduli space $Gr_{k}$ of metric ribbon trees.
\begin{defn}[\cite{FO97}]
  A \textit{ribbon tree} is a pair $(T, i)$ of a tree $T$ and an embedding $i:T\rightarrow D^{2}\subset \mathbb{C}$ which satisfies the following:\\
  (1) No vertex of $T$ has 2-edges.\\
  (2) If $v\in T$ is a vertex with one edge, then $i(v)\in \partial D^{2}$.\\
  (3) $i(T)\cap\partial D^{2}$ consists of vertices with one edge.
\end{defn}
We identify two pairs $(T, i)$ and $(T', i')$ if $T$ and $T'$ are isometric and $i$ and $i'$ are isotopic. Let $G_{k}$ be the set of all triples $(T, i, v_{1})$, where $(T, i)$ is as above, $v_{1}\in T\cap \partial D^{2}$ and $T\cap \partial D^{2}$ consists of $k$ points. We remark that choosing $v_{1}\in T\cap\partial D^{2}$ is equivalent to choosing an order of $T\cap \partial D^{2}$ which is compatible with the cyclic order of $\partial D^{2}$. 
\begin{defn}[\cite{FO97}]
  We call a vertex an \textit{internal vertex} if it has more than two edges attached to it and call it an \textit{external vertex} otherwise. We call an edge an \textit{internal edge} if both of its vertices are interior and call it an \textit{external edge} otherwise.
\end{defn}
For each $\mathfrak{t}=(T, i, v_{1})\in G_{k}$, let $C_{int}^{1}(T)$ be the set of all internal edges of $T$, and let $Gr(\mathfrak{t})$ be the set of all maps $l: C_{int}^{1}(T)\rightarrow \mathbb{R}^{+}$. We put $Gr_{k}=\bigcup_{\mathfrak{t}\in G_{k}}Gr(\mathfrak{t})$. Let $(T, i, v_{1}, l)\in Gr_{k}$ be a ribbon tree. We identify $T$ with $i(T)$ by the embedding $i$. In this paper, we denote the external edge $e_{i}$ if one of its vertices is $v_{i}$. We define a metric on $T$ such that the exterior edge $e_{i}\,(i=1,\dots,k-1)$ is isometric to $(-\infty , 0]$, exterior edge $e_{k}$ is isometric to $[0,\infty)$ and the interior edge $e$ is isometric to $[0, l(e)]$. We call $l(e)$ the \textit{length} of the interior edge $e$. The unit disk $D^{2}$ is separated into $k$ connected components by $i(T)$. We write one of connected components of $D^{2}\setminus i(T)$ as $D_{i}$ if the closure $\bar{D_{i}}$ contains $v_{i}$ and $v_{i+1}$. Note that, for each edge $e$, there are two subsets $D_{i},D_{j}$ such that its closure contains $e$. We define the integers $lef(e)$ and $rig(e)$ so that the closure of $D_{lef(e)}$ contains $e$ and $D_{lef(e)}$ is on the left side of $e$ with respect to the orientation of $e$ and $\mathbb{R}^{2}$. We define $rig(e)$ in the same way as $lef(e)$. In this paper, we define the orientation of edge $e$ such that $lef(e)=\min\{i,j\},rig(e)=\max\{i,j\}$. In order to define a gradient curve at the external edge $e_{k}$ in a similarly way as at other external edges, we consider the orientation of $e_{k}$ such that $lef(e_{k})=k,rig(e_{k})=1$, and define its metric so that $e_{k}$ is isometric to $(-\infty,0]$.
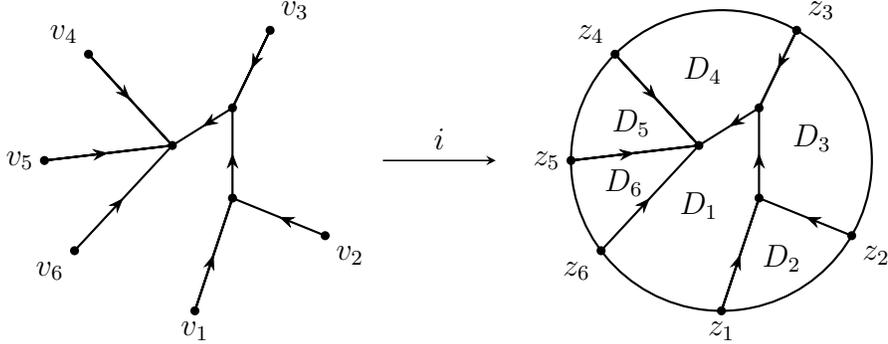
\begin{figure}[tb]
  \centering
  \begin{tikzpicture}
    \fill[black] (-6,{sqrt(3)}) circle (0.06);
    \fill[black] (-{sqrt(2)}-7,{sqrt(2)}) circle (0.06);
    \fill[black] (-9,0) circle (0.06);
    \fill[black] (-63/5+4,-6/5) circle (0.06);
    \fill[black] (-7,-2) circle (0.06);
    \fill[black] ({sqrt(3)-7},-1) circle (0.06);
    \draw[thick] ({sqrt(3)-7},-1)--(0.5-7,-0.5)--(-7,-2)--(-6.5,-0.5)--(-6.5,0.7)--(-6,{sqrt(3)})--(-
    6.5,0.7)--(-7.3,0.2)--(-{sqrt(2)}-7,{sqrt(2)})--(-7.3,0.2)--(-9,0)--(-7.3,0.2)--(-8/5-7,-6/5);
    \draw[->,>={Stealth[scale=1.25]}] ({0-7},-2)--({0.25-7},-1.25);
    \draw[->,>={Stealth[scale=1.25]}] ({sqrt(3)-7},-1)--({(sqrt(3)+0.5)/2-7},{(-1-0.5)/2});
    \draw[->,>={Stealth[scale=1.25]}] ({0.5-7},-0.5)--({0.5-7},0.1);
    \draw[->,>={Stealth[scale=1.25]}] ({1-7},{sqrt(3)})--({(1+0.5)/2-7},{(sqrt(3)+0.7)/2});
    \draw[->,>={Stealth[scale=1.25]}] ({0.5-7},0.7)--({(0.5-0.3)/2-7},{(0.7+0.2)/2});
    \draw[->,>={Stealth[scale=1.25]}] ({-sqrt(2)-7},{sqrt(2)})--({(-sqrt(2)-0.3)/2-7},{(sqrt(2)+0.2)/2});
    \draw[->,>={Stealth[scale=1.25]}] ({-2-7},0)--({(-2-0.3)/2-7},{(0+0.2)/2});
    \draw[->,>={Stealth[scale=1.25]}] ({-8/5-7},-6/5)--({(-8/5-0.3)/2-7},{(-6/5+0.2)/2});
    \fill[black] (-6.5,-0.5) circle (0.06);
    \fill[black] (-6.5,0.7) circle (0.06);
    \fill[black] (-7.3,0.2) circle (0.06);
    \draw[->,>=stealth,semithick] (-4.5,0)--(-3,0);
    \draw (-3.75,0) node[above]{$i$};
    \draw[thick] (0,0) circle (2);
    \fill[black] (1,{sqrt(3)}) circle (0.06);
    \fill[black] (-{sqrt(2)},{sqrt(2)}) circle (0.06);
    \fill[black] (-2,0) circle (0.06);
    \fill[black] (-8/5,-6/5) circle (0.06);
    \fill[black] (0,-2) circle (0.06);
    \fill[black] ({sqrt(3)},-1) circle (0.06);
    \draw[thick] ({sqrt(3)},-1)--(0.5,-0.5)--(0,-2)--(0.5,-0.5)--(0.5,0.7)--(1,{sqrt(3)})--(0.5,0.7)--(-0.3,0.2)--(-{sqrt(2)},{sqrt(2)})--(-0.3,0.2)--(-2,0)--(-0.3,0.2)--(-8/5,-6/5);
    \draw[->,>={Stealth[scale=1.25]}] (0,-2)--(0.25,-1.25);
    \draw[->,>={Stealth[scale=1.25]}] ({sqrt(3)},-1)--({(sqrt(3)+0.5)/2},{(-1-0.5)/2});
    \draw[->,>={Stealth[scale=1.25]}] (0.5,-0.5)--(0.5,0.1);
    \draw[->,>={Stealth[scale=1.25]}] (1,{sqrt(3)})--({(1+0.5)/2},{(sqrt(3)+0.7)/2});
    \draw[->,>={Stealth[scale=1.25]}] (0.5,0.7)--({(0.5-0.3)/2},{(0.7+0.2)/2});
    \draw[->,>={Stealth[scale=1.25]}] (-{sqrt(2)},{sqrt(2)})--({(-sqrt(2)-0.3)/2},{(sqrt(2)+0.2)/2});
    \draw[->,>={Stealth[scale=1.25]}] (-2,0)--({(-2-0.3)/2},{(0+0.2)/2});
    \draw[->,>={Stealth[scale=1.25]}] (-8/5,-6/5)--({(-8/5-0.3)/2},{(-6/5+0.2)/2});
    \fill[black] (0.5,-0.5) circle (0.06);
    \fill[black] (0.5,0.7) circle (0.06);
    \fill[black] (-0.3,0.2) circle (0.06);
    \draw (0,-2)node[below]{$z_{1}$};
    \draw ({sqrt(3)},-1)node[below right]{$z_{2}$};
    \draw (1,{sqrt(3)})node[above right]{$z_{3}$};
    \draw (-8/5,-6/5)node[below left]{$z_{6}$};
    \draw (-{sqrt(2)},{sqrt(2)})node[above left]{$z_{4}$};
    \draw (-2,0)node[left]{$z_{5}$};
    \draw (-7,-2)node[below]{$v_{1}$};
    \draw ({sqrt(3)-7},-1)node[below right]{$v_{2}$};
    \draw (-6,{sqrt(3)})node[above right]{$v_{3}$};
    \draw (-8/5-7,-6/5)node[below left]{$v_{6}$};
    \draw (-{sqrt(2)}-7,{sqrt(2)})node[above left]{$v_{4}$};
    \draw (-9,0)node[left]{$v_{5}$};
    \draw (-0.3,-0.6)node{$D_{1}$};
    \draw (0.8,-1.3)node{$D_{2}$};
    \draw (1.2,0.3)node{$D_{3}$};
    \draw (-0.25,1.2)node{$D_{4}$};
    \draw (-1.2,0.5)node{$D_{5}$};
    \draw (-1.3,-0.3)node{$D_{6}$};
  \end{tikzpicture}
  \caption{The ribbon tree ($k=6$) (Reproduced from \cite{S25}).}
\end{figure}
Now we review the moduli space of gradient trees.
\begin{defn}[\cite{FO97}]
  Let $M$ be a Riemannian manifold, and fix a Riemannian metric $g$ of $M$. Let $f_{1}, \cdots , f_{k}$ be $C^{\infty}$-functions on $M$ such that $f_{i+1}-f_{i}$ is a Morse function for each $i=1, \cdots , k$. Here we put $f_{k+1}=f_{1}$. Let $p_{i}$ be one of critical points of $f_{i+1}-f_{i}$. An element of the moduli space $\mathcal{M}_{g}(M, \vec{f}, \vec{p})$ of gradient trees is a pair $((T, i, v_{1}, l), I)$ of $(T, i, v_{1}, l)\in Gr_{k}$ and a map $I: T\rightarrow M$ which satisfies the following conditions:\\
  (1) $I$ is continuous, $I(v_{i})=p_{i}$.\\
  (2) For each exterior edge $e_{i}$, identify $e_{i}\simeq (-\infty, 0]$, we have
  \[
    \dfrac{dI \left\lvert_{e_{i}} \right.}{dt}=-\grad_{g}(f_{i+1}-f_{i}).
  \]
  (3) For each interior edge $e$, identify $e\simeq [0, l(e)]$, then
  \[
    \dfrac{dI \left\lvert_{e} \right.}{dt}=-\grad_{g}(f_{rig(e)}-f_{lef(e)}).
  \]
\end{defn}
We call an element $((T, i, v_{1}, l), I)$ of the moduli space $\mathcal{M}_{g}(M;\vec{f},\vec{p})$ or the map $I$ the \textit{gradient tree}. The \textit{Morse function} is the function whose critical points are non-degenerate (see \cite{AD14}). Next, we recall the moduli space of pseudoholomorphic disks. This moduli space is defined by the moduli space $\mathfrak{T}_{0,k}$ of disks with $k$ marked points which is defined as follows:
\[
  \mathfrak{T}_{0,k}\coloneqq\left\{(z_{1},\cdots ,z_{k})\in (\partial D^{2})^{k}\left\lvert
      \begin{array}{l}
        z_{i}\neq z_{j}(i\neq j),\\
        z_{1},\cdots ,z_{k}\ \text{respects the cyclic order of}\  \partial D^{2}
      \end{array}
    \right\}\right. \biggl{/}\sim\,.
\]
Here, we use the counterclockwise cyclic ordering for $\partial D^{2}$. Also, we denote $(z_{1},\cdots ,z_{n})\sim (z_{1}',\cdots ,z_{n}')$ if and only if there exists a biholomorphic map $\varphi :D^{2}\rightarrow D^{2}$ such that $\varphi (z_{i})=z_{i}'$. We next prepare definitions of pseudoholomorphic disks.
\begin{defn}[\cite{FO97}]
  Let $(M,g)$ be a Riemannian manifold. Let $\omega$ be the standard symplectic form of the cotangent manifold $T^{*}M$, and let $J$ be the almost complex structure of $T^{*}M$ which is compatible with $\omega$. Let $L_{1},\cdots ,L_{k}$ be Lagrangian submanifolds of $T^{*}M$, and let $x_{i}$ be one of intersection points of $L_{i}$ and $L_{i+1}$. Here, we put $L_{k+1}=L_{1}$. An element of the moduli space $\mathcal{M}_{J}(T^{*}M,\vec{L},\vec{x})$ of pseudoholomorphic disks is a pair $([z_{1},\cdots ,z_{k}],w)$ of elements $[z_{1},\cdots ,z_{k}]\in \mathfrak{T}_{0,k}$ and a smooth map $w: D^{2}\rightarrow T^{*}M$ satisfying the following conditions:\\
  (1) $w(z_{i})=x_{i}$\\
  (2) Denote $\partial_{i}$ the connected component of $\partial D^{2}\setminus\{z_{1},\cdots ,z_{k}\}$ which vertices are $z_{i-1}$ and $z_{i}$. Then, $w(\partial_{i})\subset L_{i}$.\\
  (3) Denote $j$ the canonical complex structure of $D^2$. Then, the differential map of $w$ is compatible with $J$ and $j$, it means $J\circ dw = dw\circ j$.
\end{defn}
We call an element $([z_{1},\cdots ,z_{k}],w)$ of the moduli space $\mathcal{M}_{J}(T^{*}M;\vec{L},\vec{x})$ itself a \textit{pseudoholomorphic disk}. We denote the exact Lagrangian submanifold $L_{i}$ as the graph of $df_{i}$, i.e.,
\[
  L_{i}=\operatorname{graph}(df)\coloneqq\left\{(p,df_{i}(p)):p\in M\right\}.
\]
Here, we denote $L_{i}^{\epsilon}\coloneqq\operatorname{graph}(\epsilon df_{i})$, and denote $x_{i}^{\epsilon}$ one of points of $L_{i}^{\epsilon}\cap L_{i+1}^{\epsilon}$. There exists a correspondence between two moduli spaces. 
\begin{thm}[\cite{FO97}]
  Let $\pi$ be the cotangent bundle $\pi:T^{*}M\rightarrow M$. We set $p_{i}=\pi(x_{i}^{\epsilon})$. Let $J=J_{g}$ be the canonical almost complex structure on $T^{*}M$ associated to the metric $g$ on $M$. For each generic\footnote{The term ``generic'' in this sentence refers to the transversality of the unstable manifolds associated with the gradient flows. Therefore, it is distinct from the term ``generic'' which we refer to the image of pseudoholomorphic disks.} $\vec{f}=(f_{i})$ and for sufficiently small $\epsilon$, we have an oriented diffeomorphism $\mathcal{M}_{g}(M:\vec{f},\vec{p})\simeq \mathcal{M}_{J}(T^{*}M:\vec{L^{\epsilon}},\vec{x^{\epsilon}})$.
\end{thm}
In this paper, we will construct gradient trees and pseudoholomorphic disks in the case of $M=\mathbb{R}$ and that Lagrangian sections are affine. Then, we discuss a correspondence between them. It is known that the unit disk is conformal to the upper half plane. Let $[z_{1},z_{2},\dots,z_{n}]\in\mathfrak{T}_{0,n}$ be a unit disk with $n$ points. Let $\varphi$ be the conformal map from the unit disk to the upper half plane such that $\varphi(z_{1})=1,\varphi(z_{2})=\infty,\varphi(z_{3})=0$. When $n\geq 4$, we set $\xi_{1},\xi_{2},\dots,\xi_{n-3}\in\mathbb{R}$ such that $\varphi(z_{i+3})=\xi_{i}$ for $i=1,2,\dots,n-3$. We thus identify pseudoholomorphic disks $w$ from a unit disk with $n$ marked points $z_{1},z_{2},\dots,z_{n}$ with the holomorphic map $w\circ\varphi^{-1}$ from the upper half plane with marked points $1,\infty, 0, \xi_{1},\xi_{2},\dots,\xi_{n-3}$ in this paper. We hereafter call such map $w$ as a holomorphic disk instead of a pseudoholomorphic disk since pseudoholomorphic disks we treat in this paper are holomorphic.

\subsection{Schwarz-Christoffel maps, hypergeometric function and their properties.}
Since Lagrangian sections we treat are affine, we study holomorphic disks which map the upper half plane to bounded polygonal domains. 
\begin{thm}[\cite{DT02},\cite{Ahl-ca},\cite{Neh75} etc.]
  \label{SCmap}
  Let $P$ be the interior of a polygon $\Gamma$ having vertices $w_{1},w_{2},...,w_{n}$ and interior angles $\alpha_{1}\pi, \alpha_{2}\pi,..., \alpha_{n}\pi$ in counterclockwise order. Let $f$ be a conformal map from the upper half plane $H^{+}$ to $P$ with $f(\infty)=w_{n}$. Then we obtain 
  \[
    f(z)=A+B\int_{}^{z}\prod_{k=1}^{n-1}(\zeta-z_{k})^{\alpha_{k}-1}d\zeta
  \]
  for some complex constants $A$ and $B$, where $w_{k}=f(z_{k})$ for $k=1,...,n-1$. We here take branches of $(\zeta-z_{i})^{\alpha_{i}}$ at each $i=1,2,\dots,n-1$ such that 
  \[
    (\zeta-z_{i})^{\alpha_{i}}=
    \begin{cases*}
      (\zeta-z_{i})^{\alpha_{i}}\,(\zeta\in \mathbb{R},\zeta>z_{i})\\
      \left\lvert \zeta-z_{i}\right\rvert^{\alpha_{i}}e^{\pi\alpha_{i}i}\,(\zeta\in \mathbb{R},\zeta<z_{i})
    \end{cases*}.
  \]
\end{thm}
This map $f$ is called the \textit{Schwarz-Christoffel map} from the upper half plane. As we mention at subsection 3.2, we can locally describe a Schwarz-Christoffel map $f$ by power series. Now, let us study integral representations and connection formulas of the hypergeometric series (and functions). We use Appell's $F_{1}$ function in the case $k=4$. Appell's hypergeometric function $F_{1}$ is defined as follows. 
\begin{defn}[\cite{Erd50},\cite{Mim22} etc.]
  \textit{Appell's hypergeometric function} $F_{1}$ is the analytic continuation of Appell's hypergeometric series
  \[
    F_{1}(a,b_{1},b_{2},c;x,y)=\sum_{m,n\geq 0}\dfrac{(a)_{m+n}(b_{1})_{m}(b_{2})_{n}}{(c)_{m+n}m!n!}x^{m}y^{n}, \left\lvert x\right\rvert<1, \left\lvert y\right\rvert<1 
  \]
  where $a,b_{1},b_{2},c$ are complex numbers, and $c$ is not a negative integer or zero. 
\end{defn}
This hypergeometric function $F_{1}$ has the following integral representation.
\begin{prop}[\cite{Erd50},\cite{Mim22} etc.]
  \label{AppellF1}
  Let $a,b_{1},b_{2},c$ be the complex number such that $0<\operatorname{Re}a<\operatorname{Re}c$. Then, we obtain 
  \[
    F_{1}(a,b_{1},b_{2},c;x,y)=\dfrac{\Gamma(c)}{\Gamma(a)\Gamma(c-a)}\int_{0}^{1}t^{a-1}(1-t)^{c-a-1}(1-xt)^{-b_{1}}(1-yt)^{-b_{2}}\,dt
  \]
  when $\left\lvert x\right\rvert<1,\left\lvert y\right\rvert<1$.
\end{prop}
We need other power series in the case $(x,y)$ does not satisfy $\left\lvert x\right\rvert<1,\left\lvert y\right\rvert<1$. This can be realized by connection formulas of $F_{1}$ and $G_{2}$. We first review Horn's hypergeometric function $G_{2}$.
\begin{defn}[\cite{Ols64},\cite{Mim22} etc.]
  Let $\alpha,\beta,\gamma,\delta$ be complex numbers. \textit{Horn's hypergeometric function} $G_{2}$ is the analytic continuation of the following hypergeometric series
  \[
    G_{2}(\alpha,\beta,\gamma,\delta;x,y)\coloneqq\sum_{m,n\geq 0}(\alpha)_{m}(\beta)_{n}(\gamma)_{n-m}(\delta)_{m-n}\dfrac{x^{m}}{m!}\dfrac{y^{n}}{n!}, \left\lvert x\right\rvert<1, \left\lvert y\right\rvert<1,
  \]
  where $(a)_{m-n}=\Gamma(a+m-n)/\Gamma(a)$.
\end{defn}
The following formula is the connection formula for $F_1$ which is used in \cite{S25}.
\begin{lem}[\cite{Ols64},\cite{Mim22} etc.]
  \label{F1conn1}
  If $\gamma,\beta-\alpha,\beta-\gamma\notin\mathbb{Z}$, then the following holds:
  \begin{align*}
    &F_{1}(\alpha,\beta',\beta,\gamma;y,x)\\
    &=\dfrac{\Gamma(\beta-\alpha)\Gamma(\gamma)}{\Gamma(\beta)\Gamma(\gamma-\alpha)}(-x)^{-\alpha}F_{1}\left(\alpha,1+\alpha-\gamma,\beta',1+\alpha-\beta;\dfrac{1}{x},\dfrac{y}{x}\right)\\
    &+\dfrac{\Gamma(\alpha-\beta)\Gamma(\gamma)}{\Gamma(\alpha)\Gamma(\gamma-\beta)}(-x)^{-\beta}G_{2}\left(\beta,\beta',\alpha-\beta,1+\beta-\gamma;-\dfrac{1}{x},-y\right).
  \end{align*}
  Here the arguments of $-x$ and $-y$ of the factors $(-x)^{*}$ and $(-y)^{*}$ are assigned to be zero on the real region $\infty<x<y<0$ (see \cite{Mim22}).
\end{lem}
In non-generic cases, the quotient of gamma functions in this formula may have a numerator whose argument vanishes. This means that we should discuss the analytic continuation of $F_1(\alpha,\beta,\beta',\gamma;x,y)$ when $\beta-\alpha$ and $\gamma-\beta$ are zero or negative integers. We will derive the modified formula in subsection 3.2 by using the connection formula for Gauss hypergeometric function $_2F_1$. The Gauss hypergeometric function $_{2}F_{1}(a,b,c;x)$ is defined as follows.
\begin{defn}[\cite{AAR99} etc.]
  The \textit{Gauss hypergeometric function} $_{2}F_{1}(a,b,c;x)$ is defined by the series
  \[
    \sum_{n=0}^{\infty}\dfrac{(a)_{n}(b)_{n}}{(c)_{n}n!}x^{n}\ \left((a)_n\coloneqq 
    \begin{cases*}
      a(a+1)\cdots(a+n-1)\ (n\geq1)\\
      1\ (n=0)
    \end{cases*}
    \right)
  \]
  for $\left\lvert x\right\rvert<1$, and by continuation elsewhere. Here $a,b,c$ are complex numbers, and $c$ is not a negative integer or zero. 
\end{defn}
We will use the following formula in subsection 3.2.
\begin{lem}[\cite{EMOT53} etc.]
  \label{2F1conn2}
  Let $a$ be the complex number such that $a\neq0,-1,-2,\dots$, and $l,m,n$ be non-negative integers. Then, we obtain
  \begin{align*}
    &_2F_1(\alpha,\alpha+m,\alpha+m+l+1;z)\cdot \dfrac{\Gamma(\alpha+m)}{\Gamma(\alpha+m+l+1)}\\
    =&(-1)^{m+l+1}(-z)^{-a-m}\sum_{n=l+1}^\infty\dfrac{(\alpha)_{n+m}(n-l-1)!}{(n+m)!n!}z^{-n}+(-z)^{-\alpha}\sum_{n=0}^{m-1}\dfrac{(m-n-1)!(\alpha)_n}{(m+l-n)!n!}z^{-n}\\
    &+\dfrac{(-z)^{-\alpha-m}}{(l+m)!}\sum_{n=0}^{l}\dfrac{(\alpha)_{n+m}(-m-l)_{n+m}}{(n+m)!n!}z^{-n}\left[\log(-z)+h_n'\right]
  \end{align*}
  Here, second term of right hand side is $0$ when $m=0$, and we have
  \[
    h_n'=\psi(1+m+n)+\psi(1+n)-\psi(\alpha+m+n)-\psi(l+1-n),
  \]
  and $\psi(z)$ is  the logarithmic derivative of the gamma function, i.e.
  \[
    \psi(z)\coloneqq \dfrac{d}{dz}\log\Gamma(z)=\dfrac{\Gamma'(z)}{\Gamma(z)}.
  \]
\end{lem}
In the last of this subsection, we review the integral representation of hypergeometric function $_2F_1$.
\begin{prop}[\cite{AAR99} etc.]
  \label{2F1}
  If $\operatorname{Re} c>\operatorname{Re} b>0$, then
  \[
    _{2}F_{1}(a,b,c;x)=\dfrac{\Gamma(c)}{\Gamma(b)\Gamma(c-b)}\int_{0}^{1}t^{b-1}(1-t)^{c-b-1}(1-xt)^{-a}\,dt
  \]
  in the $x$ plane cut along the real axis from $1$ to $\infty$. Here it is understood that $\arg t=\arg(1-t)=0$ and $(1-xt)^{-a}$ has its principal value.
\end{prop}
We will use this formula to calculate the ratio of sides of quadrilaterals in subsection 3.3. We also use the following connection formula for $_2F_1$.
\begin{prop}[\cite{EMOT53} etc.]
  \label{2F1conn1}
  Let $a,b$ be complex numbers such that $a,b\neq 0,-1,-2,\dots$, and $m$ be a positive integer. Then, we obtain
  \begin{align*}
    &_2F_1(a,b,a+b+m;z)/\Gamma(a+b+m)\\
    &=\dfrac{\Gamma(m)}{\Gamma(a+m)\Gamma(b+m)}\sum_{n=0}^{m-1}\dfrac{(a)_n(b)_n}{(1-m)_nn!}(1-z)^n\\
    &+\dfrac{(1-z)^m(-1)^m}{\Gamma(a)\Gamma(b)}\sum_{n=0}^{\infty}\dfrac{(a+m)_n(b+m)_n}{(n+m)!n!}\left[h''_n-\log(1-z)\right](1-z)^n.
  \end{align*}
  Here, the first term of right hand side is $0$ when $m=0$. We have
  \[
    h''_n=\psi(n+1)+\psi(n+m+1)-\psi(a+n+m)-\psi(b+n+m).
  \]
\end{prop}

\subsection{Interior angles in convex polygons.}
In order to discuss the convergence of holomorphic disks $w_{\epsilon}$, we prepare the explicit formula of the angles $\pi\alpha_{i}^{\epsilon}$ at $x_{i}^{\epsilon}$ and their properties. The arguments in this subsection are elementary.
\begin{lem}[\cite{S25}]
  \label{angles}
  Let $k$ be a positive integer. Let $L_{i}^{\epsilon}=\{(x,y)\in\mathbb{R}^{2}\,|\,y=(a_{i}x+b_{i})\epsilon\}$ be straight lines in $\mathbb{R}^{2}$ for $i=1,2,\dots,k$. We assume that $L_{i}^{\epsilon}\cap L_{i+1}^{\epsilon}\neq \emptyset,L_{i}^{\epsilon}\neq L_{i+1}^{\epsilon}, L_{i}^{\epsilon}\cap L_{i+1}^\epsilon\cap L_{i+2}^{\epsilon}=\emptyset$ for $i=1,2,\dots,k$. We denote by $x_{i}^{\epsilon}=(p_{i},q_{i})$ the intersection point of $L_{i}^{\epsilon}$ and $L_{i+1}^{\epsilon}$. Let $P_{\epsilon}$ be the convex polygon having vertices $x_{1}^{\epsilon},x_{2}^{\epsilon},\dots,x_{k}^{\epsilon}$ in counterclockwise order, and let $\pi\alpha_{i}^{\epsilon}$ be the interior angle at $x_{i}^{\epsilon}$. Then we obtain
  \begin{gather*}
    \pi\alpha_{i}^{\epsilon}=\left\lvert \arctan a_{i+1}\epsilon-\arctan a_{i}\epsilon\right\rvert\ \ \ \text{if $(p_{i+1}-p_{i})(p_{i-1}-p_{i})>0$,}\\
    \pi\alpha_{i}^{\epsilon}=\pi-\left\lvert \arctan a_{i+1}\epsilon-\arctan a_{i}\epsilon\right\rvert\ \ \ \text{if $(p_{i+1}-p_{i})(p_{i-1}-p_{i})<0$.}
  \end{gather*}
\end{lem}
\begin{cor}[\cite{S25}]
  \label{angleslim}
  We assume the setting in Lemma \ref{angles}. Then we obtain
  \[
    \lim_{\epsilon\to+0}\alpha_{i}^{\epsilon}=0, \lim_{\epsilon\to+0}\dfrac{\pi\alpha_{i}^{\epsilon}}{\epsilon\left\lvert a_{i+1}-a_{i}\right\rvert}=1,\lim_{\epsilon\to+0}\epsilon\Gamma(\alpha_{i}^{\epsilon})=\dfrac{\pi}{\left\lvert a_{i+1}-a_{i}\right\rvert}
  \]
  in the case $(p_{i+1}-p_{i})(p_{i-1}-p_{i})>0$, and 
  \[
    \lim_{\epsilon\to+0}\alpha_{i}^{\epsilon}=1, \lim_{\epsilon\to+0}\dfrac{\pi(1-\alpha_{i}^{\epsilon})}{\epsilon\left\lvert a_{i+1}-a_{i}\right\rvert}=1,\lim_{\epsilon\to+0}\epsilon\Gamma(1-\alpha_{i}^{\epsilon})=\dfrac{\pi}{\left\lvert a_{i+1}-a_{i}\right\rvert}
  \]
  in the case $(p_{i+1}-p_{i})(p_{i-1}-p_{i})<0$.
\end{cor}
Since $\arctan x-x$ is monotonically decreasing, we obtain the following lemma.
\begin{lem}[\cite{S25}]
  \label{anglesineq}
  If we assume the setting in Lemma \ref{angles}, then we obtain $\pi\alpha_{i}^{\epsilon}-\left\lvert (a_{i+1}-a_{i})\epsilon\right\rvert<0$ in the case $(p_{i+1}-p_{i})(p_{i-1}-p_{i})>0$.
\end{lem}
We use Lemma \ref{anglesineq} to show the uniform convergence of holomorphic disks to gradient trees at the external edges of a tree.
\subsection{Conformal Modulus of Quadrilaterals.}
In this subsection, we review a conformal invariant, called the conformal modulus, for general quadrilaterals. We also discuss some monotonicity properties of the invariant. 
\begin{defn}[\cite{Hen86} etc.]
  A \textit{quadrilateral} is a system $(Q;a_{1}, a_{2}, a_{3}, a_{4})$, where $Q$ is a Jordan region and where the $a_{i}$ are four distinct points on the boundary of $Q$, arranged in the sense of increasing parameter values.
  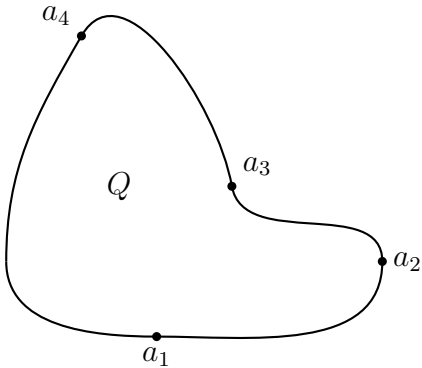
\begin{figure}[ht]
    \centering
    \begin{tikzpicture}
      \draw [thick] (6,3) to [out=-80, in=90] (8,2);
      \draw [thick] (8,2) to [out=-90, in=0] (5,1);
      \draw [thick] (5,1) to [out=180, in=-90] (3,2);
      \draw [thick] (3,2) to [out=90, in=-120] (4,5);
      \draw [thick] (4,5) to [out=60, in=100] (6,3);
      \fill[black] (5,1) circle (0.06) node[below]{$a_{1}$};
      \fill[black] (8,2) circle (0.06) node[right]{$a_{2}$};
      \fill[black] (6,3) circle (0.06) node[above right]{$a_{3}$};
      \fill[black] (4,5) circle (0.06) node[above left]{$a_{4}$};
      \draw (4.5,3) node{$Q$};
    \end{tikzpicture}
    \caption{The quadrilateral (Reproduced from \cite{S25}).}
  \end{figure}
\end{defn}
The word ``quadrilateral'' usually means a polygon which is bounded by four straight lines. Since we consider a quadrilateral $(Q;a_{1}, a_{2}, a_{3}, a_{4})$ as a polygon in Section 3, we use the term ``quadrilateral'' in the sense above only in this subsection. The conformal modulus of the quadrilateral is defined as follows.
\begin{defn}[\cite{Hen86} etc.]
  Let $(Q;a,b,c,d)$ be a quadrilateral. Let $w=f(z)$ be a one-to-one conformal map from the domain $Q$ onto the rectangle $0<u<1,0<v<M\,(w=u+iv)$ which maps $a,b,c,d$ to $0,1,1+iM,iM$, respectively. The number $M$ is called the \textit{conformal modulus} of the quadrilateral $(Q;a,b,c,d)$, and we denote it by $M(Q;a,b,c,d)$.
\end{defn}
Note that $M(Q;a,b,c,d)=1/M(Q;b,c,d,a)$. There are some well-known inequalities of conformal modulus as follows. 
\begin{prop}[\cite{Hen86} etc.]
  \label{mpocm1}
  Let $(Q; a, b, c, d)$ be a quadrilateral, and let $(Q; a', b, c, d)$ be quadrilateral obtained from the first by moving the point $a$ along the boundary of $Q$ in the direction of $d$. Denoting the moduli of the two quadrilaterals by $\mu$ and $\mu'$, respectively, we have $\mu>\mu'$.
\end{prop}
\begin{prop}[\cite{Hen86} etc.]
  \label{mpocm2}
  Let $(Q; a, b, c, d)$ and $(Q; a, b, c, d)$ be quadrilaterals, which have the boundary segments $(b, c)$ and $(d, a)$ in common, but which are such that $Q\subset  Q'$. Denoting the moduli of the two quadrilaterals by $\mu$ and $\mu'$, respectively, we have $\mu<\mu'$.
\end{prop}
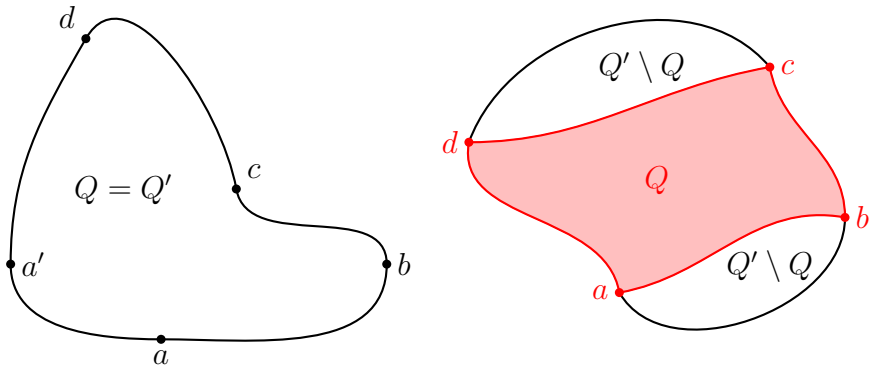
\begin{figure}[ht]
  \centering
  \begin{tikzpicture}
    \draw [thick] (6,3) to [out=-80, in=90] (8,2);
    \draw [thick] (8,2) to [out=-90, in=0] (5,1);
    \draw [thick] (5,1) to [out=180, in=-90] (3,2);
    \draw [thick] (3,2) to [out=90, in=-120] (4,5);
    \draw [thick] (4,5) to [out=60, in=100] (6,3);
    \fill[black] (5,1) circle (0.06) node[below]{$a$};
    \fill[black] (3,2) circle (0.06) node[right]{$a'$};
    \fill[black] (8,2) circle (0.06) node[right]{$b$};
    \fill[black] (6,3) circle (0.06) node[above right]{$c$};
    \fill[black] (4,5) circle (0.06) node[above left]{$d$};
    \draw (4.5,3) node{$Q=Q'$};
  \end{tikzpicture}
  \begin{tikzpicture}
    \path[name path=l1, draw, red] (7,4) to [out=-80, in=90] (8,2);
    \path[name path=l2, draw, red] (8,2) to [out=170, in=10] (5,1);
    \path[name path=l3, draw, red] (5,1) to [out=100, in=-100] (3,3);
    \path[name path=l4, draw, red] (3,3) to [out=0, in=190] (7,4);
    \draw [thick] (8,2) to [out=-90, in=-60] (5,1);
    \draw [thick] (3,3) to [out=70, in=130] (7,4);
    \path[fill, pink] (7,4) to [out=-80, in=90] (8,2) to [out=170, in=10] (5,1) to [out=100, in=-100] (3,3) to [out=0, in=190] (7,4);
    \draw[thick, red] (7,4) to [out=-80, in=90] (8,2);
    \draw[thick, red] (8,2) to [out=170, in=10] (5,1);
    \draw[thick, red] (5,1) to [out=100, in=-100] (3,3);
    \draw[thick, red] (3,3) to [out=0, in=190] (7,4);
    \fill[red] (7,4) circle (0.06) node[right]{$c$};
    \fill[red] (8,2) circle (0.06) node[right]{$b$};
    \fill[red] (5,1) circle (0.06) node[left]{$a$};
    \fill[red] (3,3) circle (0.06) node[left]{$d$};
    \draw[red] (5.5,2.5) node{$Q$};
    \draw (5.3,4) node{$Q'\setminus Q$};
    \draw (7,1.35) node{$Q'\setminus Q$};
  \end{tikzpicture}
  \caption{The situation of Prop.\ref{mpocm1} and Prop.\ref{mpocm2}.}
\end{figure}
By using Prop.\ref{mpocm1} and Prop.\ref{mpocm2}, we also obtain the following property. 
\begin{prop}[\cite{DV09}]
  \label{mpocm3}
  Let $(Q; a, b, c, d)$ be a polygonal quadrilateral i.e. quadrilaterals with (linear) intervals as sides. If the quadrilateral $(Q; a, b, c, d)$ is convex and the point $a', a'\neq a$, lies in the closed convex set, whose boundary consists of the side $(a, b)$ and the linear extensions of the sides $(d, a)$ and $(b, c)$ over the points $a, b$, respectively, up to the point of their intersections (which may be $\infty$), then
  \[
    M(Q;a,b,c,d)<M(Q';a',b,c,d).
  \]
\end{prop}
We use these inequalities to evaluate conformal moduli of polygonal quadrilaterals in this paper. In the last of this section, we show some results in the special case.
\begin{lem}[\cite{Hen86}]
  Every symmetric quadrilateral has modulus 1. Here, the quadrilateral $(Q;x_1,x_2,x_3,x_4)$ is called \textit{symmetric} if there are some $i=1,2,3,4$ which satisfy the following two conditions. 
  \begin{itemize}
    \item The region $Q$ is symmetric with respect to the straight line $\Lambda$ through $x_i$ and $x_{i+2}$.
    \item The points $x_{i+1}$ and $x_{i+3}$ are symmetric with respect to $\Lambda$.
  \end{itemize}
\end{lem}
From this lemma, the unit disk with four marked points $1,i,-1,-i$ has modulus 1. When we set $\varphi$ as the conformal map from the unit disk to the upper half plane such that
\[
  \varphi(-i)=0,\varphi(i)=1,\varphi(-1)=\infty,
\]
we obtain 
\[
  \varphi(z)=\dfrac{z+i}{z+1}\dfrac{1-i}{2}
\]
and $\varphi(1)=1/2$. Therefore, one has the following lemma.
\begin{lem}
  \label{cm1/2}
  The quadrilateral $(Q;a,b,c,d)$ has modulus 1 if and only if the quadrilateral $(Q;a,b,c,d)$ is conformally equivalent to the upper half plane $(\mathbb{H};0,1/2,1,\infty)$ with four marked points $0,1/2,1,\infty$.
\end{lem}
We use this lemma to discuss the conformal moduli in non-generic cases.

\subsection{Gradient trees and holomorphic disks in the case $M=\mathbb{R},k=4$}
In this paper, we consider only the case $M=\mathbb{R},k=4$ and Lagrangian sections are affine. We first recall conditions when four affine Lagrangian sections $L_{1},L_{2},L_{3},L_{4}$ of $T^*\mathbb{R}$ form the convex quadrilateral.
\begin{lem}[\cite{S25}]
  \label{square}
  Let $L_{1},L_{2},L_{3},L_{4}$ be affine Lagrangian sections in $T^{*}\mathbb{R}$ which satisfy $L_{i}\cap L_{i+1}\neq\emptyset,L_{i}\neq L_{i+1}$, for $i=1,2,3,4$. By identifying $T^{*}\mathbb{R}$ with $\mathbb{R}^{2}$, we express each $L_{i}$ as a line $\{(x,y)\mid y=a_{i}x+b_{i}\}$, where $a_{i}$ and $b_{i}$ are real numbers. Let $x_{i}=(p_{i},q_{i})$ be the intersection point between $L_{i}$ and $L_{i+1}$. Then, $x_{1},x_{2},x_{3},x_{4}$ forms a quadrilateral if and only if $p_{1},p_{2},p_{3},p_{4}$ satisfy one of conditions in $(B)$ in Table \ref{tree-critpt2}. Furthermore, the quadrilateral $x_{1}x_{2}x_{3}x_{4}$ is convex and has vertices $x_{1},x_{2},x_{3},x_{4}$ in counterclockwise order if and only if $a_{1},a_{2},a_{3},a_{4}$ satisfy one of conditions in $(A)$ in Table \ref{tree-critpt2}. Here, $(a,b)$ in Table \ref{tree-critpt2} is a open interval, and we set $(a,b)=\emptyset$ if $a\geq b$ holds.
\end{lem}
We can prove Lemma \ref{square} by elementary calculations. Let $f_{1},f_{2},f_{3},f_{4}$ be functions such that $L_{i}=\operatorname{graph}(df_{i})$ for $i=1,2,3,4$. We next study the moduli space $\mathcal{M}_{g}(\mathbb{R};\vec{f},\vec{p})$ in the case when $L_{1},L_{2},L_{3},L_{4}$ satisfy one of conditions in Table \ref{k4trees2}. By the definition of ribbon trees, the ribbon tree $(T,i)$ is the element of $Gr_{4}$ if and only if $T$ is isometric to one of the following trees in Figure \ref{k4trees2}. 
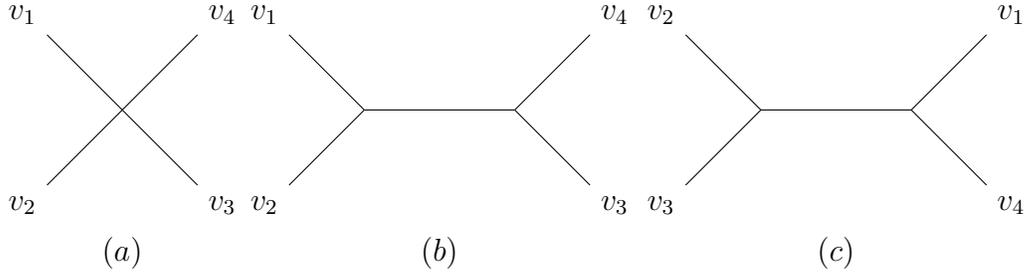
\begin{figure}[tb]
  \center
  \begin{tikzpicture}
    \begin{scope}
      \draw[thin] (1,1)--(0,0);
      \draw[thin] (-1,1)--(0,0);
      \draw[thin] (1,-1)--(0,0);
      \draw[thin] (-1,-1)--(0,0);
      \draw (-1,1) node[above left]{$v_{1}$};
      \draw (-1,-1) node[below left]{$v_{2}$};
      \draw (1,-1) node[below right]{$v_{3}$};
      \draw (1,1) node[above right]{$v_{4}$};
      \draw (0,-1.9) node{$(a)$};
    \end{scope}
    \begin{scope}[xshift=120]
      \draw[thin] (2,1)--(1,0);
      \draw[thin] (2,-1)--(1,0);
      \draw[thin] (-1,0)--(1,0);
      \draw[thin] (-2,1)--(-1,0);
      \draw[thin] (-2,-1)--(-1,0);
      \draw (-2,1) node[above left]{$v_{1}$};
      \draw (-2,-1) node[below left]{$v_{2}$};
      \draw (2,-1) node[below right]{$v_{3}$};
      \draw (2,1) node[above right]{$v_{4}$};
      \draw (0,-1.9) node{$(b)$};
    \end{scope}
    \begin{scope}[xshift=270]
      \draw[thin] (2,1)--(1,0);
      \draw[thin] (2,-1)--(1,0);
      \draw[thin] (-1,0)--(1,0);
      \draw[thin] (-2,1)--(-1,0);
      \draw[thin] (-2,-1)--(-1,0);
      \draw (-2,1) node[above left]{$v_{2}$};
      \draw (-2,-1) node[below left]{$v_{3}$};
      \draw (2,-1) node[below right]{$v_{4}$};
      \draw (2,1) node[above right]{$v_{1}$};
      \draw (0,-1.9) node{$(c)$};
    \end{scope}
  \end{tikzpicture}
  \caption{Candidates of the tree $T$ of the ribbon tree $(T,i)\in Gr_{4}$ (Reproduced from \cite{S25}).}
  \label{k4trees2}
\end{figure}
\begin{lem}[\cite{S25}]
  \label{k4gradtree}
  The moduli space $\mathcal{M}_{g}(\mathbb{R};\vec{f},\vec{p})$ of gradient trees is a one-point set. The gradient tree is constructed by the tree which is isomorphic to the tree in column $(C)$ of Table \ref{tree-critpt2}. Here, $(a),(b),(c)$ correspond to Figure \ref{k4trees2}. 
  \begin{table}[htbp]
    \centering
    \begin{tabular}{c|c|c}
      $(A)$ & $(B)$ & $(C)$ \\ \hline\hline
      \multirow{2}{*}{$a_{2},a_{4}\in(a_{1},a_{3})$} & $p_{4}<p_{1}<p_{2}<p_{3}$ & \multirow{2}{*}{$(c)$} \\
      & $p_{3}<p_{2}<p_{1}<p_{4}$ & \\ \hline
      \multirow{2}{*}{$a_{2},a_{4}\in(a_{3},a_{1})$} & $p_{1}<p_{4}<p_{3}<p_{2}$ & \multirow{2}{*}{$(c)$} \\
      & $p_{2}<p_{3}<p_{4}<p_{1}$ & \\ \hline
      \multirow{2}{*}{$a_{1},a_{3}\in(a_{2},a_{4})$} & $p_{1}<p_{2}<p_{3}<p_{4}$ & \multirow{2}{*}{$(b)$} \\
      & $p_{4}<p_{3}<p_{2}<p_{1}$ & \\ \hline
      \multirow{2}{*}{$a_{1},a_{3}\in(a_{4},a_{2})$} & $p_{3}<p_{4}<p_{1}<p_{2}$ & \multirow{2}{*}{$(b)$} \\
      & $p_{2}<p_{1}<p_{4}<p_{3}$ & \\ \hline
      \multirow{6}{*}{$\max\{a_{1},a_{3}\}<\min\{a_{2},a_{4}\}$} & $p_{4}<p_{1}<p_{3}<p_{2}$ & \multirow{2}{*}{$(c)$} \\
      & $p_{2}<p_{3}<p_{1}<p_{4}$ &\\ \cline{2-3}
      & $p_{4}<p_{3}<p_{1}<p_{2}$ & \multirow{2}{*}{$(b)$} \\
      & $p_{2}<p_{1}<p_{3}<p_{4}$ &\\ \cline{2-3}
      & $p_{4}<p_{1}=p_{3}<p_{2}$ & \multirow{2}{*}{$(a)$} \\
      & $p_{2}<p_{3}=p_{1}<p_{4}$ &\\ \hline
      \multirow{6}{*}{$\max\{a_{2},a_{4}\}<\min\{a_{1},a_{3}\}$} & $p_{1}<p_{4}<p_{2}<p_{3}$ & \multirow{2}{*}{$(c)$} \\
      & $p_{3}<p_{2}<p_{4}<p_{1}$ &\\ \cline{2-3}
      & $p_{3}<p_{4}<p_{2}<p_{1}$ & \multirow{2}{*}{$(b)$} \\
      & $p_{1}<p_{2}<p_{4}<p_{3}$ &\\ \cline{2-3}
      & $p_{1}<p_{4}=p_{2}<p_{3}$ & \multirow{2}{*}{$(a)$} \\
      & $p_{3}<p_{2}=p_{4}<p_{1}$ &\\ \hline
    \end{tabular}
    \caption{All of conditions when we obtain a unique gradient tree (Reproduced from \cite{S25}).}
    \label{tree-critpt2}
  \end{table}
\end{lem}
\section{Main results and their proofs.}
\subsection{Main results.}
Let $f_1,f_2,f_3,f_4$ be functions on $\mathbb{R}$ which satisfy Table \ref{tree-critpt2}. We denote by $p_i$ the critical point of $f_{i+1}-f_i$, and by $x_i^\epsilon$ the intersection point of $L_i^\epsilon$ and $L_{i+1}^\epsilon$. We set $w_\epsilon$ as the Schwarz-Christoffel map from the upper half plane to the quadrilateral $x_1^\epsilon x_2^\epsilon x_3^\epsilon x_4^\epsilon$ such that $w_\epsilon(0)=x_3^\epsilon,w_\epsilon(1)=x_1^\epsilon,w_\epsilon(\infty)=x_2^\epsilon$. We define $z_{4,\epsilon}\in(0,1)$ as $z_{4,\epsilon}=w_\epsilon^{-1}(x_4^\epsilon)$. By Lemma \ref{k4gradtree}, the moduli space $\mathcal{M}_g(\mathbb{R};\vec{f},\vec{p})$ of gradient trees is a one-point set. We denote $(I,(T,i,v_{1},l))$ the unique element of $\mathcal{M}_{g}(\mathbb{R};\vec{f},\vec{p})$. According to the definition of trees, $T\in Gr_{4}$ has at most one internal edge. We therefore consider $l\geq0$ as the length of the internal edge of $T$. In \cite{S25}, we assumed that $L_i^\epsilon\cap L_j^\epsilon\neq \emptyset$ and $L_i^\epsilon\neq L_j^\epsilon\ (i\neq j)$ and $p_1\neq p_3,\ p_2\neq p_4$. We called such a quadrilateral a \textit{generic quadrilateral}. We also say that $(f_1,f_2,f_3,f_4)$ is \textit{generic} when the corresponding quadrilateral is generic. In \cite{S25}, we gave the main statement for the case when the image of pseudoholomorphic disks is generic quadrilaterals in $\mathbb{C}\simeq T^*\mathbb{R}$. In this paper, we will show the complete statement for any case when the image of pseudoholomorphic disks is convex quadrilaterals in $T^*\mathbb{R}$. We here define the neighborhood $D_{p}(\delta)$ of $p\in\mathbb{R}\cup\{\infty\}$ and the transformation $\phi_{p,\delta}$ from a stripe $\Theta_{ext}\coloneqq\{(\tau,\sigma)\in(-\infty,0)\times[0,1]\}$ to $D_{p}(\delta)\setminus\{p\}$ as follows.
\begin{align*}
  &D_{p}(\delta)\coloneqq
  \begin{cases*}
    \{\left\lvert z-p\right\rvert<\delta\mid\text{Im}z\geq0\}\,(p\in\mathbb{R})\\
    \{\left\lvert z\right\rvert^{-1}<\delta\mid\text{Im}z\geq0\}\,(p=\infty)
  \end{cases*}
  \,(\delta>0)\\
  &\phi_{p,\delta}(\tau,\sigma)\coloneqq
  \begin{cases*}
    p+\delta\exp[\pi(\tau+i\sigma)]\,(p\neq\infty)\\
    -\delta^{-1}\exp[-\pi(\tau+i\sigma)]\,(p=\infty)
  \end{cases*}
\end{align*}
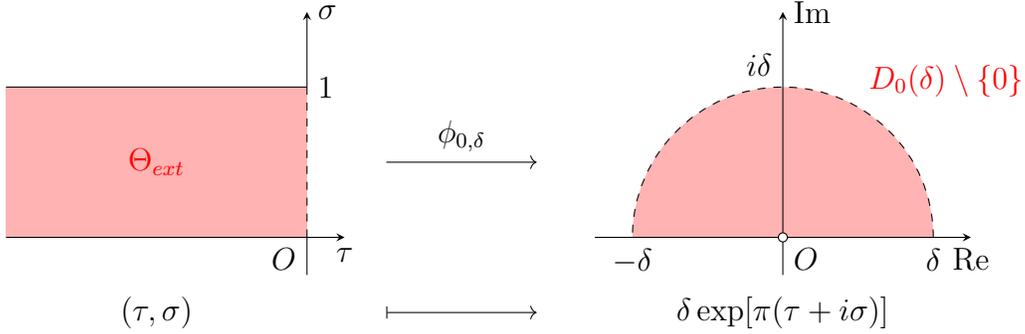
\begin{figure}[bt]
  \centering
  \begin{tikzpicture}
    \draw (0,0) node[below left]{$O$};
    \draw (0,2) node[right]{$1$};
    \fill[red,opacity=0.3] (-4,0)--(0,0)--(0,2)--(-4,2);
    \draw (0,2)--(-4,2);
    \draw[red](-2,1)node{$\Theta_{ext}$};
    \draw[->,,>=stealth] (-4,0)--(0.5,0)node[below]{$\tau$};
    \draw[->,>=stealth] (0,2)--(0,3)node[right]{$\sigma$};
    \draw[dashed] (0,0)--(0,2);
    \draw (0,-0.5)--(0,0);
    \draw (-2,-1)node{$(\tau,\sigma)$};
    \begin{scope}[xshift=30]
      \draw[->] (0,1)--(2,1);
      \draw (1,1)node[above]{$\phi_{0,\delta}$};
      \draw[|->] (0,-1)--(2,-1);
    \end{scope}
    \begin{scope}[xshift=180]
      \fill[red,opacity=0.3] (2,0) arc (0:180:2) --(-2,0)--cycle;
      \draw[dashed] (2,0) arc (0:180:2) --(-2,0);
      \draw[->,>=stealth] (-2.5,0)--(2.5,0)node[below]{Re};
      \draw[->,>=stealth] (0,-0.5)--(0,3)node[right]{Im};
      \draw (0,2) node[above left]{$i\delta$};
      \draw (2,0) node[below]{$\delta$};
      \draw (-2,0) node[below]{$-\delta$};
      \draw (0,0) node[below right]{$O$};
      \draw (0,-1) node{$\delta\exp[\pi(\tau+i\sigma)]$};
      \draw[red] (60:2)node[above right]{$D_{0}(\delta)\setminus\{0\}$};
      \fill[white] (0,0)circle(0.06);
      \draw (0,0) circle (0.06);
    \end{scope}
  \end{tikzpicture}
  \caption{The transformation $\phi_{0,\delta}$ (Reproduced from \cite{S25}).}
  \label{k3trans}
\end{figure}
We sometime omit $\epsilon$ from $z_{4,\epsilon}$ in order to treat $z_{4,\epsilon}$ as $z_{1},z_{2},z_{3}$. We now discuss the correspondence between gradient trees and holomorphic disks. We first consider the case $(p_3-p_1)(p_4-p_2)\neq0$. In this case, we divide the upper half plane into seven region as Figure \ref{fig:k4upperhalf}.
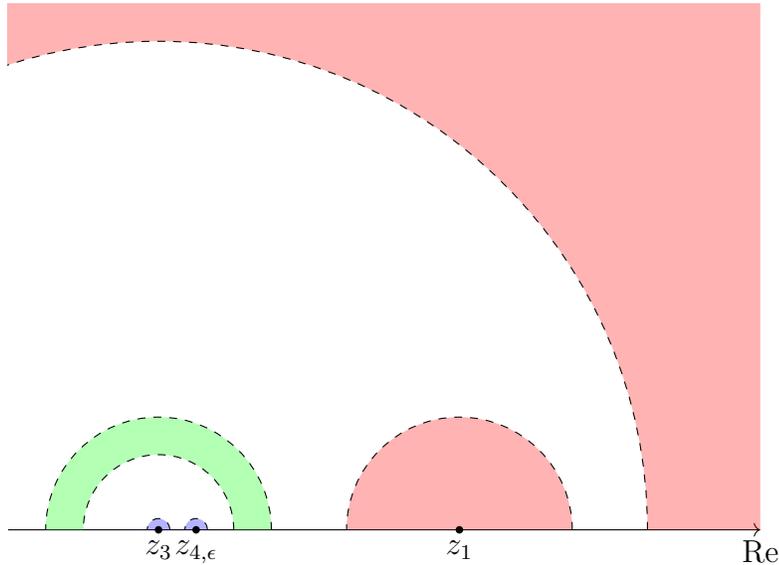
\begin{figure}[tb]
  \centering
  \begin{tikzpicture}
    \clip (-2,-0.75) rectangle (9,7);
    \draw[->]  (-2,0)--(8,0) node [below]{Re};
    \fill[blue,opacity=0.3] (0.15,0) arc (0:180:0.15)--(-0.15,0)--cycle;
    \fill[blue,opacity=0.3] (0.65,0) arc (0:180:0.15)--(-0.35,0)--cycle;
    \fill[red,opacity=0.3] (5.5,0) arc (0:180:1.5)--(4,0)--cycle;
    \fill[red,opacity=0.3] (8,0)--(8,9)--(-6.5,9)--(-6.5,0) arc (180:0:6.5)--(6.5,0)--cycle;
    \fill[green,opacity=0.3] (1.5,0) arc (0:180:1.5)--(-1.5,0)--(-1,0) arc (180:0:1)--(1,0)--cycle;
    \draw[dashed] (0.15,0) arc (0:180:0.15);
    \draw[dashed] (0.65,0) arc (0:180:0.15);
    \fill[black] (0,0) circle (0.05) node[below]{$z_{3}$};
    \fill[black] (0.5,0) circle (0.05) node[below]{$z_{4,\epsilon}$};
    \draw[dashed] (1,0) arc (0:180:1);
    \draw[dashed] (1.5,0) arc (0:180:1.5);
    \draw[dashed] (5.5,0) arc (0:180:1.5);
    \fill[black] (4,0) circle (0.05) node[below]{$z_{1}$};
    \draw[dashed] (6.5,0) arc (0:180:6.5);
  \end{tikzpicture}
  \caption{The figure of the upper half plane divided by seven regions in the case $(p_3-p_1)(p_4-p_2)>0$ (Reproduced from \cite{S25}).}
  \label{fig:k4upperhalf}
\end{figure}
We here define the region $D_{z_i}^\epsilon(\delta)$ and the transformation $\phi_{z_i,\delta}^\epsilon$ from $\Theta_{ext}$ to $D_{z_i}^\epsilon(\delta)$ as in 
\[
  D_{z_{i}}^{\epsilon}(\delta)\coloneqq
  \begin{cases*}
    D_{z_{i}}(\delta)\,(i=1,2)\\
    D_{z_{i}}(z_{4,\epsilon}\delta)\,(i=3,4)
  \end{cases*}
  ,\phi_{z_{i},\delta}^{\epsilon}\coloneqq
  \begin{cases*}
    \phi_{z_{i},\delta}\,(i=1,2)\\
    \psi_{\epsilon}^{-1}\circ\phi_{\psi_{\epsilon}(z_{i}),\delta}\,(i=3,4)
  \end{cases*}
  ,\psi_{\epsilon}(z)\coloneqq z/z_{4,\epsilon},
\]
and the stripe $\Theta_{int}^{\epsilon}(\delta)$ and the transformation $\phi_{int}$ from $\Theta_{int}^{\epsilon}(\delta)$
\begin{align*}
  \phi_{int}(\tau,\sigma)&\coloneqq\exp\left[-\pi\tau+i\pi(1-\sigma)\right]\\
  \Theta_{int}^{\epsilon}(\delta)&\coloneqq\left.\left\{(\tau,\sigma)\,\middle| \,\tau\in\left(-\dfrac{1}{\pi}\log \delta,-\dfrac{1}{\pi}\log z_{4,\epsilon}+\dfrac{1}{\pi}\log \delta\right),\sigma\in[0,1]\right\}\right.\,.
\end{align*}
Images of regions and transformations are in Figure \ref{4markedpts} and Figure \ref{fig-intedge}. Then we obtain the following theorem.
\begin{thm}
  \label{k4corr1}
  Let $\delta>0$ be a real number such that $D_{i}(\delta)\cap D_{j}(\delta)$ for $i\neq j$ and $i,j\in\{0,1,\infty\}$. The point $p_{0}\in \mathbb{R}$ is either $p_{1}$ or $p_{2}$ such that the Morse index of $p_{0}$ is zero, and the point $p_{0}'\in\mathbb{R}$ is either $p_{3}$ or $p_{4}$ satisfying the same condition as $p_{0}$. We set $I_{i}$ as the restriction of $I$ to the external edge $e_{i}$ of $T$, and set $I_{int}$ as the restriction of $I$ to the internal edge of $T$. If $(p_{3}-p_{1})(p_{4}-p_{2})>0$, then we obtain 
  \begin{gather}
    \lim_{\epsilon\to+0}\sup_{z\in D(\delta)}\left\lvert w_{\epsilon}(z)-p_{0}\right\rvert=0,\label{eq:k4p0-1-1}\\
    \lim_{\epsilon\to+0}\sup_{z\in D(\delta)}\left\lvert w_{\epsilon}\circ\psi_{\epsilon}^{-1}(z)-p_{0}'\right\rvert=0,\label{eq:k4p0-1-2}\\
    \lim_{\epsilon\to+0}\sup_{(\tau,\sigma)\in\Theta_{ext}}\left\lvert w_{\epsilon}\circ \phi_{z_{i},\delta}^{\epsilon}(\tau,\sigma)-I_{i}(\epsilon\tau)\right\rvert=0\,(i=1,2,3,4),\label{eq:k4ext-1}\\
    \lim_{\epsilon\to+0}\sup_{(\tau,\sigma)\in\Theta_{int}^{\epsilon}(\delta)}\left\lvert w_{\epsilon}\circ\phi_{int}(\tau,\sigma)-I_{int}(\epsilon\tau)\right\rvert=0.\label{eq:k4int-1}
  \end{gather}
\end{thm}
\begin{rem}
  As we mentioned in \cite{S25}, this convergence in \eqref{eq:k4int-1} is not a uniformly convergence since the domain $\Theta_{int}^{\epsilon}(\delta)$ moves as $\epsilon>0$ varies.
\end{rem}
\begin{figure}[tb]
  \centering
  \begin{tikzpicture}
    \begin{scope}
      \clip (-1,-1) rectangle (6,5);
      \draw[dashed] (0.5,0) arc (0:180:0.5);
      \draw[dashed] (2.5,0) arc (0:180:0.5);
      \draw[dashed] (3.5,0) arc (0:180:3.5);
      \fill[red,opacity=0.3] (0.5,0) arc (0:180:0.5)--(-0.5,0)--cycle;
      \fill[red,opacity=0.3] (2.5,0) arc (0:180:0.5)--(1.5,0)--cycle;
      \draw[red] (2,2)node{$D_{z_{2}}(\delta)$};
      \fill[red,opacity=0.3] (3.5,0)--(5.5,0)--(5.5,6)--(-1,6)--(-3.5,0) arc (180:0:3.5)--cycle;
      \draw[->] (-1,0)--(5.5,0)node[below]{Re};
      \fill[black] (0,0) circle (0.06) node[below]{$0$};
      \fill[black] (2,0) circle (0.06) node[below]{$1$};
      \draw (3.5,0)node[below]{$1/\delta$};
      \draw[red] (0,0.5) node[above]{$D_{z_{3}}(\delta)$};
      \draw[red] (2,0.5) node[above]{$D_{z_{1}}(\delta)$};
    \end{scope}
    \begin{scope}[xshift=200]
      \draw[->] (-1,2.5)--(1,2.5);
      \draw (0,2.5)node[above]{$\psi_{\epsilon}$};
      \draw (0,2.5)node[below]{$\psi_{\epsilon}(z)=z/z_{4,\epsilon}$};
    \end{scope}
    \begin{scope}[xshift=270]
      \clip (-1,-1) rectangle (6,5);
      \draw[dashed] (0.5,0) arc (0:180:0.5);
      \draw[dashed] (2.5,0) arc (0:180:0.5);
      \draw[dashed] (3.5,0) arc (0:180:3.5);
      \draw[dashed] (4.5,0) arc (0:180:4.5);
      \fill[blue,opacity=0.3] (0.5,0) arc (0:180:0.5)--(-0.5,0)--cycle;
      \fill[blue,opacity=0.3] (2.5,0) arc (0:180:0.5)--(1.5,0)--cycle;
      \draw[teal] (4,4)node{$\psi_{\epsilon}(D_{int}^{\epsilon}(\delta))$};
      \fill[green,opacity=0.3] (3.5,0)--(4.5,0) arc (0:180:4.5)--(-4.5,0)--(-3.5,0) arc (180:0:3.5)--cycle;
      \draw[->] (-1,0)--(5.5,0)node[below]{Re};
      \fill[black] (0,0) circle (0.06) node[below]{$0$};
      \fill[black] (2,0) circle (0.06) node[below]{$1$};
      \draw (3.5,0)node[below]{$1/\delta$};
      \draw (4.5,0)node[below]{$\delta/z_{4,\epsilon}$};
      \draw[blue] (0,0.5) node[above]{$D_{0}(\delta)$};
      \draw[blue] (2,0.5) node[above]{$D_{1}(\delta)$};
    \end{scope}
  \end{tikzpicture}
  \caption{The figure of nearby $0$ in the upper half plane in the case $z_{4,\epsilon}\rightarrow0\,(\epsilon\to+0)$ (Reproduced from \cite{S25}).}
  \label{4markedpts}
\end{figure}
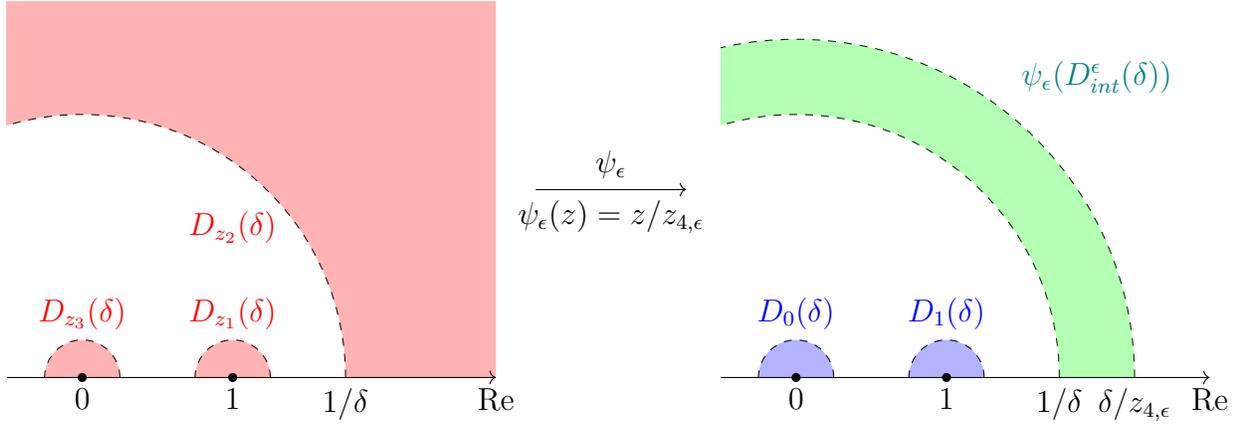
\begin{figure}[tb]
  \centering
  \begin{tikzpicture}
    \begin{scope}
      \clip (-1,-1) rectangle (6,5);
      \fill[green,opacity=0.3] (0.5,0)--(4.5,0)--(4.5,3)--(0.5,3)--cycle;
      \draw[->] (-0.5,0)--(5,0)node[above]{$\tau$};
      \draw[->] (0,-0.5)--(0,4.5)node[left]{$\sigma$};
      \draw (0.5,3)--(4.5,3);
      \draw[dashed] (0.5,0)--(0.5,3);
      \draw[dashed] (4.5,0)--(4.5,3);
      \draw[teal] (2.5,1.5)node{$\Theta_{int}^{\epsilon}(\delta)$};
      \draw (0,0)node[below left]{$O$};
      \draw (0.5,0)node[below]{$\tau_{1,\delta}$};
      \draw (4.5,0)node[below]{$\tau_{2,\delta,\epsilon}$};
      \draw (0,3)node[left]{$1$};
    \end{scope}
    \begin{scope}[xshift=180]
      \draw[->] (-1,2)--(1,2);
      \draw (0,2)node[above]{$\phi_{int}$};
    \end{scope}
    \begin{scope}[xshift=250]
      \clip (-1,-1.25) rectangle (6,6);
      \draw[dashed] (1.5,0) arc (0:180:1.5);
      \draw[dashed] (4,0) arc (0:180:4);
      \draw[teal] (3.7,3.7)node{$\phi_{int}(\Theta_{int}^{\epsilon}(\delta))$};
      \draw[teal] (3.9,3.2)node{$(=D_{int}^{\epsilon}(\delta))$};
      \fill[green,opacity=0.3] (1.5,0)--(4,0) arc (0:180:4)--(-4,0)--(-1.5,0) arc (180:0:1.5)--cycle;
      \draw[->] (-1,0)--(5.5,0)node[below]{Re};
      \draw (1.5,0)node[below]{$\phi_{int}(\tau_{2,\delta,\epsilon},1)$};
      \draw (1.5,-0.5)node[below]{$(=z_{4,\epsilon}/\delta)$};
      \draw (4,0)node[below]{$\phi_{int}(\tau_{1,\delta},1)$};
      \draw (4,-0.5)node[below]{$(=\delta)$};
    \end{scope}
  \end{tikzpicture}
  \caption{The figure of $\phi_{int}$ and $D_{int}^{\epsilon}(\delta)$ in the case $z_{4,\epsilon}\rightarrow+0$ (Here, $\tau_{1,\epsilon}\coloneqq-(\log\delta)/\pi$ and $\tau_{2,\delta,\epsilon}\coloneqq-(\log(z_{4,\epsilon}/\delta))/\pi$) (Reproduced from \cite{S25}).}
  \label{fig-intedge}
\end{figure}
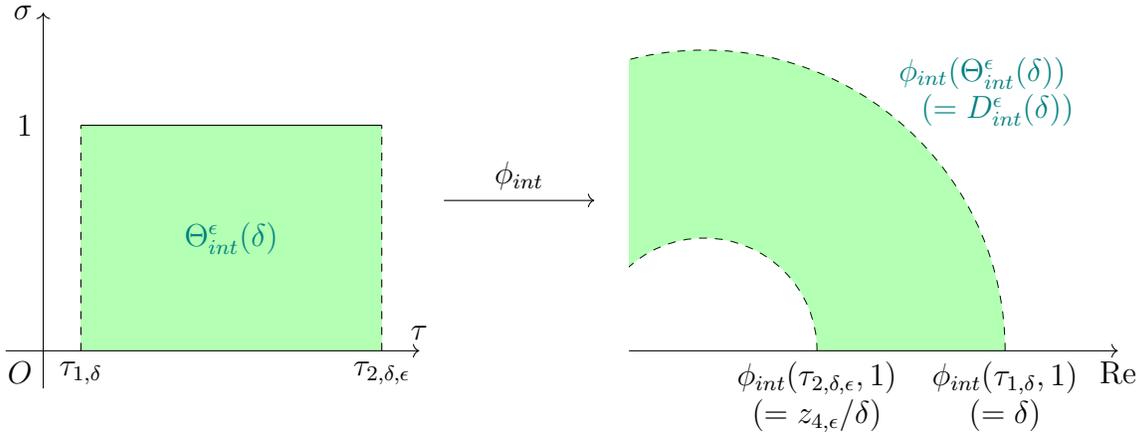
Since we proved this theorem in the generic case in \cite{S25}, we will prove the non-generic case in subsection 3.4, where is quite different from that in \cite{S25}. The significant difference is whether Lemma \ref{F1conn1} can be apply or not. Lemma \ref{F1conn1} is used to compute the analytic continuation of $w_\epsilon$. However, in the non-generic case above, we are forced to substitute negative integers or zero to gamma functions in numerators. Therefore we have to use the other formula for analytic continuation of $F_1$. Therefore, we will induce the other formula by using Proposition \ref{2F1conn1} in the next section.

We next consider the case $(p_3-p_1)(p_4-p_2)=0$. In this case, we get $z_{4,\epsilon}\to 1/2$ as $\epsilon\to+0$ by using Proposition \ref{mpocm1},\ref{mpocm2},\ref{mpocm3} and Lemma \ref{2F1conn2}. We will show it in subsection 4.3. Then we divide the upper half plane into five regions as follows (see Figure \ref{fig:k4upperhalf-2}). 
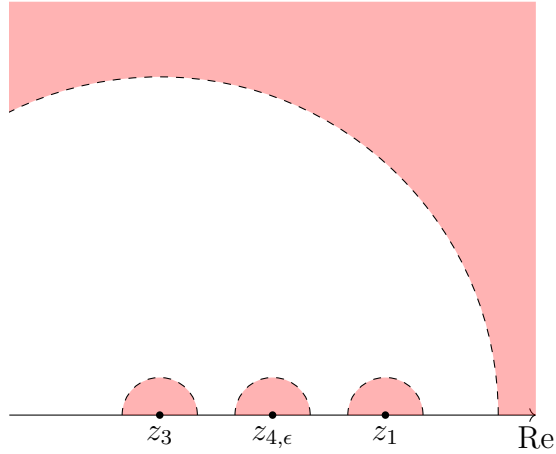
\begin{figure}[tb]
  \centering
  \begin{tikzpicture}
    \clip (-2,-0.75) rectangle (5.75,5.5);
    \draw[->]  (-2,0)--(5,0) node [below]{Re};
    \fill[red,opacity=0.3] (0.5,0) arc (0:180:0.5)--(-0.5,0)--cycle;
    \fill[red,opacity=0.3] (2,0) arc (0:180:0.5)--(1,0)--cycle;
    \fill[red,opacity=0.3] (3.5,0) arc (0:180:0.5)--(2.5,0)--cycle;
    \fill[red,opacity=0.3] (5,0)--(5,5.5)--(-4,5.5)--(-4.5,0) arc (180:0:4.5)--cycle;
    \fill[black] (0,0) circle (0.05) node[below]{$z_{3}$};
    \fill[black] (1.5,0) circle (0.05) node[below]{$z_{4,\epsilon}$};
    \draw[dashed] (0.5,0) arc (0:180:0.5);
    \draw[dashed] (2,0) arc (0:180:0.5);
    \draw[dashed] (3.5,0) arc (0:180:0.5);
    \fill[black] (3,0) circle (0.05) node[below]{$z_{1}$};
    \draw[dashed] (4.5,0) arc (0:180:4.5);
  \end{tikzpicture}
  \caption{The figure of the upper half plane divided by five regions in the case $(p_3-p_1)(p_4-p_2)=0$.}
  \label{fig:k4upperhalf-2}
\end{figure}
\[
  \overline{\mathbb{H}}=D_{z_{3}}(\delta)\cup D_{z_{4,\epsilon}}(\delta)\cup D_{z_{1}}(\delta)\cup D_{z_{2}}(\delta)\cup D_\epsilon(\delta).
\]
We can take a real positive number $\delta>0$ such that $D_{z_{i}}(\delta)\cap D_{z_{j}}(\delta)=\emptyset$ for $i,j\in\{1,2,3,4\},i\neq j$. The set $D_\epsilon(\delta)$ is then given by 
\[
  D_\epsilon(\delta)\coloneqq\overline{\mathbb{H}}\setminus(D_{z_{3}}(\delta)\cup D_{z_{4,\epsilon}}(\delta)\cup D_{z_{1}}(\delta)\cup D_{z_{2}}(\delta)).
\]
Four regions $D_{z_{1}}(\delta),D_{z_{2}}(\delta),D_{z_{3}}(\delta),D_{z_{4,\epsilon}}(\delta)$ are red regions in Figure \ref{fig:k4upperhalf-2}, and these regions correspond to each exterior edges of tree. The white region in Figure \ref{fig:k4upperhalf-2} is $D_\epsilon(\delta)$, and this region corresponds to the interior vertex. We give the transformations between the stripe $\Theta_{ext}$ and the region $D_{z_{i}}(\delta)$ as $\phi_{z_{i},\delta}$. The following theorem is the result in the case $(p_3-p_1)(p_4-p_2)=0$.
\begin{thm}
  \label{k4corr2}
  Let $\delta>0$ be a real number such that $D_{z_{4,\epsilon}}(\delta)\cap D_{z_j}(\delta)=\emptyset$ for sufficiently small $\epsilon>0$ and $j=1,2,3$. The point $p_{0}\in \mathbb{R}$ is one of $p_{1},p_{2},p_3$ or $p_4$ such that the Morse index of $p_{0}$ is zero. We set $I_{i}$ as the restriction of $I$ to the external edge $e_{i}$ of $T$. If $(p_{3}-p_{1})(p_{4}-p_{2})=0$, then we obtain 
  \begin{gather}
    \lim_{\epsilon\to+0}\sup_{z\in D_\epsilon(\delta)}\left\lvert w_{\epsilon}(z)-p_{0}\right\rvert=0,\label{eq:k4p0-2}\\
    \lim_{\epsilon\to+0}\sup_{(\tau,\sigma)\in\Theta_{ext}}\left\lvert w_{\epsilon}\circ \phi_{z_{i},\delta}(\tau,\sigma)-I_{i}(\epsilon\tau)\right\rvert=0\ (i=1,2,3,4).\label{eq:k4ext-2}
  \end{gather}
\end{thm}
\subsection{Power series representations of holomorphic disks.}
We here study the power series representation of a Schwarz-Christoffel map from the upper half plane to a convex quadrilateral. Let $P$ be the interior of the convex quadrilateral $x_{1}x_{2}x_{3}x_{4}$ having vertices $x_{1},x_{2},x_{3},x_{4}$ in counterclockwise order and let $\pi\alpha_{1},\pi\alpha_{2},\pi\alpha_{3},\pi\alpha_{4}$ be the corresponding interior angles. We define $w$ as the Schwarz-Christoffel map from the upper half plane to $P$ such that $w(0)=x_{1},w(1)=x_{3},w(\infty)=x_{4}$. We set $\xi\in(0,1)$ as $\xi=w^{-1}(x_{2})$. In Lemma 4.3 in \cite{S25}, we assumed $\alpha_i+\alpha_{i+1}\neq 1$ for $i=1,2,3,4$. When $\alpha_i+\alpha_{i+1}=1$ holds for some $i=1,2,3,4$, $(\rm{iii})$ does not hold though we still have $(\rm{i}),(\rm{ii}),(\rm{iv})$ in \cite[Lemma 4.3]{S25}. However, only $(\rm{iii})$ does not hold in the case $\alpha_1+\alpha_2=1$. We first rearrange $(\rm{iii})$ in this case.
\begin{lem}
  \label{sqr-rep1}
  When $\alpha_1+\alpha_2=1$ holds, we obtain the following on $\{z\in\overline{\mathbb{H}}\mid\xi<\left\lvert z\right\rvert<1\}$:
  \begin{align*}
    w(z)&=x_{1}+\dfrac{\sin\pi\alpha_1}{\pi}\dfrac{x_{2}-x_{1}}{_{2}F_{1}(\alpha_{1},1-\alpha_{3},1;\xi)}\left[-e^{\pi\alpha_1i}\sum_{\substack{k,l\geq0\\k\neq l}}\dfrac{(\alpha)_k(1-\alpha_3)_l}{k!l!}\dfrac{1}{k-l}\left(\dfrac{z}{\xi}\right)^{-k}z^l\right.\\
    &\left.\hspace{1.5cm}+e^{\pi\alpha_1i}\sum_{n\geq0}\left[\psi(n+1)-\psi(\alpha_1+n)+\log\left(-\dfrac{z}{\xi}\right)\right]\dfrac{(\alpha_1)_n(1-\alpha_3)_n}{n!n!}\xi^n\right].
  \end{align*}
\end{lem}
\begin{proof}
  From Lemma 4.3 $(\rm{i})$ \cite{S25}, we have
  \begin{align}
    \label{z0rep}
    w(z)&=x_{1}+\dfrac{1}{\Gamma(\alpha_{1}+1)\Gamma(1-\alpha_1)}\dfrac{x_{2}-x_{1}}{_{2}F_{1}(\alpha_{1},1,1-\alpha_{3};\xi)}\left(\dfrac{z}{\xi}\right)^{\alpha_{1}}\notag\\
    &\hspace{5cm}\cdot F_{1}\left(\alpha_{1},\alpha_1,1-\alpha_{3},\alpha_{1}+1;\dfrac{z}{\xi},z\right).
  \end{align}
  In this proof, we consider the analytic continuation of $F_1$ in (\ref{z0rep}). We first have 
  \begin{align*}
    &F_{1}\left(\alpha_{1},1-\alpha_{2},1-\alpha_{3},\alpha_{1}+1;\dfrac{z}{\xi},z\right)\\
    &=\sum_{m,n\geq0}\dfrac{(\alpha_1)_{m+n}(1-\alpha_2)_m(1-\alpha_3)_n}{(\alpha_1+1)_{m+n}m!n!}\left(\dfrac{z}{\xi}\right)^mz^n\\
    &=\sum_{n\geq0}\dfrac{(\alpha_1)_n(1-\alpha_3)_n}{(\alpha_1+1)_nn!}\cdot\left(\sum_{m\geq0}\dfrac{(\alpha_1+n)_m(1-\alpha_2)_m}{(\alpha_1+1+n)_mm!}\left(\dfrac{z}{\xi}\right)^m\right)z^n\\
    &=\sum_{n\geq0}\ _2F_1\left(\alpha_1+n,1-\alpha_2,\alpha_1+1+n;\dfrac{z}{\xi}\right)\cdot\dfrac{(\alpha_1)_n(1-\alpha_3)_n}{(\alpha_1+1)_nn!}z^n.
  \end{align*}
  We here set $\alpha_2=1-\alpha_1$. From Lemma \ref{2F1conn2}, we have 
  \begin{align*}
    &_2F_1\left(\alpha_1+n,\alpha_1,\alpha_1+1+n;\dfrac{z}{\xi}\right)\cdot\dfrac{\Gamma(\alpha_1+n)}{\Gamma(\alpha_1+n+1)}\\
    =&(-1)^{n+1}\left(-\dfrac{z}{\xi}\right)^{-\alpha_1-n}\sum_{j=1}^\infty\dfrac{(\alpha_1)_{j+n}(j-1)!}{(j+n)!j!}\left(\dfrac{z}{\xi}\right)^{-j}+\left(-\dfrac{z}{\xi}\right)^{-\alpha_1}\sum_{j=0}^{n-1}\dfrac{(n-j-1)!(\alpha_1)_j}{(n-j)!j!}\left(\dfrac{z}{\xi}\right)^{-j}\\
    &+\left(-\dfrac{z}{\xi}\right)^{-\alpha_1-n}\dfrac{1}{n!}\dfrac{(\alpha_1)_{n}(-n)_{n}}{n!}\left[\log\left(-\dfrac{z}{\xi}\right)+h_0'\right]\\
    =&(-1)^{n+1}\left(-\dfrac{z}{\xi}\right)^{-\alpha_1-n}\sum_{j=1}^\infty\dfrac{(\alpha_1)_{j+n}(j-1)!}{(j+n)!j!}\left(\dfrac{z}{\xi}\right)^{-j}+\left(-\dfrac{z}{\xi}\right)^{-\alpha_1}\sum_{j=0}^{n-1}\dfrac{(n-j-1)!(\alpha_1)_j}{(n-j)!j!}\left(\dfrac{z}{\xi}\right)^{-j}\\
    &+(-1)^n\left(-\dfrac{z}{\xi}\right)^{-\alpha_1-n}\dfrac{(\alpha_1)_n}{n!}\left[\log\left(-\dfrac{z}{\xi}\right)+\psi(n+1)-\psi(\alpha_1+n)\right]
  \end{align*}
  for $\left\lvert z\right\rvert<1<\left\lvert z/\xi\right\rvert$, i.e. $\xi<\left\lvert z\right\rvert<1$, and $\left\lvert \arg(-z/\xi)\right\rvert<\pi$. Hence, we obtain
  \begin{align*}
    &F_{1}\left(\alpha_{1},\alpha_1,1-\alpha_{3},\alpha_{1}+1;\dfrac{z}{\xi},z\right)\\
    &=\sum_{n\geq0}\ _2F_1\left(\alpha_1+n,\alpha_1,\alpha_1+1+n;\dfrac{z}{\xi}\right)\cdot\dfrac{(\alpha_1)_n(1-\alpha_3)_n}{(\alpha_1+1)_nn!}z^n\\
    &=\sum_{n\geq0}\dfrac{\Gamma(\alpha_1+n+1)}{\Gamma(\alpha_1+n)}\left[(-1)^{n+1}\left(-\dfrac{z}{\xi}\right)^{-\alpha_1-n}\sum_{j=1}^\infty\dfrac{(\alpha_1)_{j+n}(j-1)!}{(j+n)!j!}\left(\dfrac{z}{\xi}\right)^{-j}\right.\\
    &+\left(-\dfrac{z}{\xi}\right)^{-\alpha_1}\sum_{j=0}^{n-1}\dfrac{(n-j-1)!(\alpha_1)_j}{(n-j)!j!}\left(\dfrac{z}{\xi}\right)^{-j}\\
    &\left.+(-1)^n\left(-\dfrac{z}{\xi}\right)^{-\alpha_1-n}\dfrac{(\alpha_1)_n}{n!}\left[\log\left(-\dfrac{z}{\xi}\right)+\psi(n+1)-\psi(\alpha_1+n)\right]\right]\cdot\dfrac{(\alpha_1)_n(1-\alpha_3)_n}{(\alpha_1+1)_nn!}z^n.
  \end{align*}
  We have
  \begin{align*}
    &\sum_{n\geq0}\sum_{j\geq1}\dfrac{\Gamma(\alpha_1+n+1)}{\Gamma(\alpha_1+n)}(-1)^{n+1}\left(-\dfrac{z}{\xi}\right)^{-\alpha_1-n}\dfrac{(\alpha_1)_{j+n}(j-1)!}{(j+n)!j!}\left(\dfrac{z}{\xi}\right)^{-j}\dfrac{(\alpha_1)_n(1-\alpha_3)_n}{(\alpha_1+1)_nn!}z^n\\
    &=-\left(-\dfrac{z}{\xi}\right)^{-\alpha_1}\sum_{n\geq0}\sum_{j\geq1}\dfrac{\Gamma(\alpha_1+n+1)}{\Gamma(\alpha_1+n)}\dfrac{(\alpha_1)_{j+n}(j-1)!}{(j+n)!j!}\dfrac{(\alpha_1)_n(1-\alpha_3)_n}{(\alpha_1+1)_nn!}\left(\dfrac{z}{\xi}\right)^{-j-n}z^n\\
    &=-\alpha_1\left(-\dfrac{z}{\xi}\right)^{-\alpha_1}\sum_{n\geq0}\sum_{j\geq1}\dfrac{(\alpha_1)_{j+n}(1-\alpha_3)_n}{(j+n)!n!}\dfrac{1}{j}\left(\dfrac{z}{\xi}\right)^{-j-n}z^n\\
    &=-\alpha_1\left(-\dfrac{z}{\xi}\right)^{-\alpha_1}\sum_{l\geq0}\sum_{k\geq l+1}\dfrac{(\alpha_1)_{k}(1-\alpha_3)_l}{k!l!}\dfrac{1}{k-l}\left(\dfrac{z}{\xi}\right)^{-k}z^l
  \end{align*}
  and
  \begin{align*}
    &\sum_{n\geq0}\dfrac{\Gamma(\alpha_1+n+1)}{\Gamma(\alpha_1+n)}\left(-\dfrac{z}{\xi}\right)^{-\alpha_1}\sum_{j=0}^{n-1}\dfrac{(n-j-1)!(\alpha_1)_j}{(n-j)!j!}\left(\dfrac{z}{\xi}\right)^{-j}\dfrac{(\alpha_1)_n(1-\alpha_3)_n}{(\alpha_1+1)_nn!}z^n\\
    &=\alpha_1\left(-\dfrac{z}{\xi}\right)^{-\alpha_1}\sum_{n\geq0}\sum_{j=0}^{n-1}\dfrac{(\alpha_1)_j(1-\alpha_3)_n}{j!n!}\dfrac{1}{n-j}\left(\dfrac{z}{\xi}\right)^{-j}z^n
  \end{align*}
  and 
  \begin{align*}
    &\sum_{n\geq0}\dfrac{\Gamma(\alpha_1+n+1)}{\Gamma(\alpha_1+n)}(-1)^n\left(-\dfrac{z}{\xi}\right)^{-\alpha_1-n}\dfrac{(\alpha_1)_n}{n!}\\
    &\hspace{3cm}\cdot\left[\log\left(-\dfrac{z}{\xi}\right)+\psi(n+1)-\psi(\alpha_1+n)\right]\dfrac{(\alpha_1)_n(1-\alpha_3)_n}{(\alpha_1+1)_nn!}z^n\\
    &=\alpha_1\left(-\dfrac{z}{\xi}\right)^{-\alpha_1}\sum_{n\geq0}\left[\log\left(-\dfrac{z}{\xi}\right)+\psi(n+1)-\psi(\alpha_1+n)\right]\dfrac{(\alpha_1)_n(1-\alpha_3)_n}{n!n!}\xi^n.
  \end{align*}
  These series converge absolutely in the same time when $\xi<\left\lvert z\right\rvert<1$ holds. Then we have
  \begin{align*}
    &F_{1}\left(\alpha_{1},\alpha_1,1-\alpha_{3},\alpha_{1}+1;\dfrac{z}{\xi},z\right)\\
    =&-\alpha_1\left(-\dfrac{z}{\xi}\right)^{-\alpha_1}\sum_{\substack{k,l\geq0\\ k\neq l}}\dfrac{(\alpha_1)_{k}(1-\alpha_3)_l}{k!l!}\dfrac{1}{k-l}\left(\dfrac{z}{\xi}\right)^{-k}z^l\\
    &+\alpha_1\left(-\dfrac{z}{\xi}\right)^{-\alpha_1}\sum_{n\geq0}\left[\log\left(-\dfrac{z}{\xi}\right)+\psi(n+1)-\psi(\alpha_1+n)\right]\dfrac{(\alpha_1)_n(1-\alpha_3)_n}{n!n!}\xi^n.
  \end{align*}
  Here, the argument of $-z/\xi$ of the factors $(-z/\xi)^*$ is assigned to be zero if $z$ is a real number which satisfies $-\infty<z/\xi<z<0$. Since we have $\arg z=\arg(z/\xi)=0$ when $z$ is positive real number and $0<z<z/\xi<\infty$ holds, we set $-z/\xi=\exp[-i\pi]\cdot z/\xi$. We finally obtain
  \begin{align*}
    w(z)&=x_{1}+\dfrac{1}{\Gamma(\alpha_{1}+1)\Gamma(1-\alpha_1)}\dfrac{x_{2}-x_{1}}{_{2}F_{1}(\alpha_{1},1,1-\alpha_{3};\xi)}\left(\dfrac{z}{\xi}\right)^{\alpha_{1}}\\
    &\cdot\left[-\alpha_1\left(\dfrac{z}{\xi}e^{-\pi i}\right)^{-\alpha_1}\sum_{\substack{k,l\geq0\\k\neq l}}\dfrac{(\alpha)_k(1-\alpha_3)_l}{k!l!}\dfrac{1}{k-l}\left(\dfrac{z}{\xi}\right)^{-k}z^l\right.\\
    &\left.\hspace{1cm}+\alpha_1\left(\dfrac{z}{\xi}e^{-\pi i}\right)^{-\alpha_1}\sum_{n\geq0}\left[\psi(n+1)-\psi(\alpha_1+n)+\log\left(\dfrac{z}{\xi}e^{-\pi i}\right)\right]\dfrac{(\alpha_1)_n(1-\alpha_3)_n}{n!n!}\xi^n\right]\\
    &=x_{1}+\dfrac{\sin\pi\alpha_1}{\pi}\dfrac{x_{2}-x_{1}}{_{2}F_{1}(\alpha_{1},1,1-\alpha_{3};\xi)}\\
    &\cdot\left[-e^{\pi\alpha_1i}\sum_{\substack{k,l\geq0\\k\neq l}}\dfrac{(\alpha)_k(1-\alpha_3)_l}{k!l!}\dfrac{1}{k-l}\left(\dfrac{z}{\xi}\right)^{-k}z^l\right.\\
    &\left.\hspace{1cm}+e^{\pi\alpha_1i}\sum_{n\geq0}\left[\psi(n+1)-\psi(\alpha_1+n)+\log z-\log\xi-\pi i\right]\dfrac{(\alpha_1)_n(1-\alpha_3)_n}{n!n!}\xi^n\right].
  \end{align*}
\end{proof}
If we set $w'(z)\coloneqq w(z/\xi)$, then we obtain the Schwarz-Christoffel map $w'$ which maps $0,1,1/\xi,\infty$ to $x_{1},x_{2},x_{3},x_{4}$ respectively. 
\begin{lem}
  \label{sqr-rep2}
  Let $w$ be a Schwarz-Christoffel map from the upper half plane to the same convex quadrilateral $x_{1}x_{2}x_{3}x_{4}$ as in Lemma \ref{sqr-rep1} such that $w$ maps $0,1,\xi,\infty$ to $x_{1},x_{2},x_{3},x_{4}$, respectively. Here, we set $\xi=w^{-1}(x_{3})\in(1,\infty)$. If $\alpha_1+\alpha_2=1$ holds, then $w$ can be expressed by the following power series on $\{z\in\overline{\mathbb{H}}\mid1<\left\lvert z\right\rvert<\xi\}$:
  \begin{align*}
    w(z)&=x_{1}+\dfrac{\sin\pi\alpha_1}{\pi}\dfrac{x_{2}-x_{1}}{_{2}F_{1}(\alpha_{1},1-\alpha_{3},1;1/\xi)}\left[-e^{\pi\alpha_1i}\sum_{\substack{k,l\geq0\\k\neq l}}\dfrac{(\alpha)_k(1-\alpha_3)_l}{k!l!}\dfrac{1}{k-l}z^{-k}\left(\dfrac{z}{\xi}\right)^l\right.\\
    &\left.\hspace{1.5cm}+e^{\pi\alpha_1i}\sum_{n\geq0}\left[\psi(n+1)-\psi(\alpha_1+n)+\log(-z)\right]\dfrac{(\alpha_1)_n(1-\alpha_3)_n}{n!n!}\xi^{-n}\right].
  \end{align*}
\end{lem}
We use the power series representation of $w$ on the neighborhood of $z=\xi$ in subsection 3.5.
\begin{lem}
  \label{sqr-rep3}
  On $\{z\in\overline{\mathbb{H}}\mid\left\lvert \xi-z\right\rvert<\min\{\xi,1-\xi\}\}$, we obtain
  \begin{align*}
    w(z)&=x_2+\dfrac{\Gamma(\alpha_2+\alpha_3)}{\Gamma(\alpha_2+1)\Gamma(\alpha_3)}\dfrac{x_3-x_2}{_2F_1(\alpha_3,1-\alpha_1,\alpha_2+\alpha_3;1-\xi)}\xi^{\alpha_1-1}\\
    &\hspace{3cm}\cdot\left(\dfrac{z-\xi}{1-\xi}\right)^{\alpha_2}F_1\left(\alpha_2,1-\alpha_3,1-\alpha_1,\alpha_2+1;\dfrac{z-\xi}{1-\xi},-\dfrac{z-\xi}{\xi}\right).
  \end{align*}
\end{lem}
We can prove this lemma in the same way as Lemma 4.3 in \cite{S25}.
\subsection{The behavior of $z_{4,\epsilon}$ and trees.}
We first study the limit value of $z_{4,\epsilon}$. In \cite{S25}, we calculated this value in the case $(p_3-p_1)(p_4-p_2)\neq0$.
\begin{lem}[\cite{S25}]
  \label{square-cm1}
  If $(p_3-p_1)(p_4-p_2)>0$ (resp. $(p_3-p_1)(p_4-p_2)<0$) holds, then we have
  \[
    \lim_{\epsilon\to+0}z_{4,\epsilon}=z_3\ (\text{resp.}\ z_1).
  \]
\end{lem}
In this section, we calculate the limit value of $z_{4,\epsilon}$ in the case $(p_3-p_1)(p_4-p_2)=0$.
\begin{lem}
  \label{square-cm2}
  If $(p_3-p_1)(p_4-p_2)=0$ holds, then we have
  \[
    \lim_{\epsilon\to+0}z_{4,\epsilon}=1/2.
  \]
\end{lem}
We first prove the following lemma to prove Lemma \ref{square-cm2}.
\begin{lem}
  \label{parallel}
  Let $L_{1}^{\epsilon},L_{2}^{\epsilon},L_{3}^{\epsilon},L_{4}^{\epsilon}$ be straight lines in $\mathbb{R}^2$ which satisfy $L_i^\epsilon\cap L_{i+1}^\epsilon\neq\emptyset$ and $L_i^\epsilon\neq L_{i+1}^\epsilon$ for $i=1,2,3,4$. We set $L_i^\epsilon=\{(x,y)\in\mathbb{R}\mid y=\epsilon(a_i x+b_i)\}\ (a_i,b_i\in\mathbb{R})$ for $i=1,2,3,4$. Let $x_i^\epsilon=(p_i,q_j^\epsilon)$ be the intersection point of $L_i^\epsilon$ and $L_{i+1}^\epsilon$. We assume that $L_1^\epsilon,\dots,L_4^\epsilon$ forms a parallelogram $(Q_{\epsilon};x_{1}^{\epsilon},x_{2}^{\epsilon},x_{3}^{\epsilon},x_{4}^{\epsilon})$. If $(p_3-p_1)(p_4-p_2)=0$ holds, then we have 
  \[
    \lim_{\epsilon\to+0}M(Q_{\epsilon};x_{1}^{\epsilon},x_{2}^{\epsilon},x_{3}^{\epsilon},x_{4}^{\epsilon})=1.
  \]
\end{lem}
\begin{proof}
  For any $\epsilon>0$, there exists $\xi_\epsilon\in(0,1)$ uniquely such that $f_\epsilon(\xi_\epsilon)=x_2^\epsilon$, when $f_\epsilon$ is the Schwarz-Christoffel map from the upper half plane to the parallelogram $Q_\epsilon$ such that
  \[
    f_\epsilon(0)=x_1^\epsilon,\ f_\epsilon(1)=x_3^\epsilon,\ f_\epsilon(\infty)=x_4^\epsilon.
  \]
  We have $M(Q;x_1,x_2,x_3,x_4)=1$ if and only if the quadrilateral $(Q;x_1,x_2,x_3,x_4)$ is conformally equivalent to the square. Therefore, it is enough to prove $\xi_\epsilon\to1/2$ as $\epsilon\to+0$. Since we have
  \[
    M(Q_{\epsilon};x_{1}^{\epsilon},x_{2}^{\epsilon},x_{3}^{\epsilon},x_{4}^{\epsilon})=1/M(Q_{\epsilon};x_{2}^{\epsilon},x_{3}^{\epsilon},x_{4}^{\epsilon},x_{1}^{\epsilon}),
  \]
  without loss of generality, we set $L_1^\epsilon,\dots,L_4^\epsilon$ as
  \begin{align*}
    &L_{1}^{\epsilon}=\{(x,\epsilon a_{1}x):x\in\mathbb{R}\},\ L_{2}^{\epsilon}=\{(x,\epsilon a_{2}x):x\in\mathbb{R}\},\\
    &L_{3}^{\epsilon}=\{(x,\epsilon (a_{1}x+b)):x\in\mathbb{R}\},\ L_{4}^{\epsilon}=\{(x,\epsilon (a_{2}x+b)):x\in\mathbb{R}\},
  \end{align*} 
  and assume $p_4<0<p_2$ and $b>0$. We now consider the case $0\leq a_1<a_2$. We first show that the limit of $\xi_\epsilon$ exists. We here set $\epsilon_{1}>\epsilon_{2}>0$, and we compare $M(Q_{\epsilon_{1}};x_{1}^{\epsilon_{1}},x_{2}^{\epsilon_{1}},x_{3}^{\epsilon_{1}},x_{4}^{\epsilon_{1}})$ and $M(Q_{\epsilon_{2}};x_{1}^{\epsilon_{2}},x_{2}^{\epsilon_{2}},x_{3}^{\epsilon_{2}},x_{4}^{\epsilon_{2}})$. We take the parallelogram $(Q_{\epsilon_{2}}';x_{1}^{\epsilon_{2}},y_{2}^{\epsilon_{2}},x_{3}^{\epsilon_{2}},y_{4}^{\epsilon_{2}})$ which is similar to the parallelogram $(Q_{\epsilon_{2}};x_{1}^{\epsilon_{2}},x_{2}^{\epsilon_{2}},x_{3}^{\epsilon_{2}},x_{4}^{\epsilon_{2}})$.
  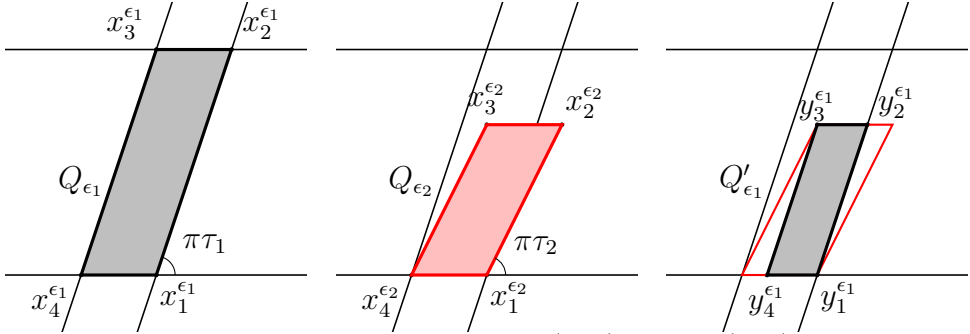
\begin{figure}[ht]
    \center
    \begin{tikzpicture}
      \begin{scope}[scale=0.5]
        \clip (-4,-1.5) rectangle (4,7.25);
        \path[name path=l1, draw, semithick] plot ({\x}, {0});
        \path[name path=l2, draw, semithick] plot ({\x}, {3*\x});
        \path[name path=l3, draw, semithick] plot ({\x}, {6});
        \path[name path=l4, draw, semithick] plot ({\x}, {3*\x+6});
        \path[name intersections={of=l1 and l2,by=x1},draw];
        \path[name intersections={of=l2 and l3,by=x2},draw];
        \path[name intersections={of=l3 and l4,by=x3},draw];
        \path[name intersections={of=l4 and l1,by=x4},draw];
        \fill[black] (x1) circle (0.06) node[below=0.3cm,right=-0.1cm]{$x_{1}^{\epsilon_{1}}$};
        \fill[black] (x2) circle (0.06) node[above right]{$x_{2}^{\epsilon_{1}}$};
        \fill[black] (x3) circle (0.06) node[above left]{$x_{3}^{\epsilon_{1}}$};
        \fill[black] (x4) circle (0.06) node[below left]{$x_{4}^{\epsilon_{1}}$};
        \fill[lightgray](x1)--(x2)--(x3)--(x4)--cycle;
        \draw[very thick](x1)--(x2)--(x3)--(x4)--cycle;
        \draw (-2,2.5) node{$Q_{\epsilon_{1}}$};
        \draw[thin] (0.5,0) arc (0:{atan(3)}:0.5);
        \draw ({atan(3)/2}:0.5) node[above right]{$\pi\tau_{1}$};
      \end{scope}
      \begin{scope}[scale=0.5,xshift=250]
        \clip (-4,-1.5) rectangle (4,7.25);
        \path[name path=l1, draw, semithick] plot ({\x}, {0});
        \path[name path=l2, draw, semithick] plot ({\x}, {3*\x});
        \path[name path=l3, draw, semithick] plot ({\x}, {6});
        \path[name path=l4, draw, semithick] plot ({\x}, {3*\x+6});
        \path[name intersections={of=l1 and l2,by=x1},draw];
        \path[name intersections={of=l2 and l3,by=x2},draw];
        \path[name intersections={of=l3 and l4,by=x3},draw];
        \path[name intersections={of=l4 and l1,by=x4},draw];
        \fill[black] (0,0) circle (0.06) node[below=0.3cm,right=-0.1cm]{$x_{1}^{\epsilon_{2}}$};
        \fill[black] (2,4) circle (0.06) node[above=0.3cm,right=-0.1cm]{$x_{2}^{\epsilon_{2}}$};
        \fill[black] (0,4) circle (0.06) node[above]{$x_{3}^{\epsilon_{2}}$};
        \fill[black] (-2,0) circle (0.06) node[below left]{$x_{4}^{\epsilon_{2}}$};
        \fill[pink] (0,0)--(2,4)--(0,4)--(-2,0)--cycle;
        \draw[very thick,red] (0,0)--(2,4)--(0,4)--(-2,0)--cycle;
        \draw (-2,2.5) node{$Q_{\epsilon_{2}}$};
        \draw[thin] (0.5,0) arc (0:{atan(2)}:0.5);
        \draw ({atan(2)/2}:0.5) node[above right]{$\pi\tau_{2}$};
      \end{scope}
      \begin{scope}[scale=0.5,xshift=500]
        \clip (-4,-1.5) rectangle (4,7.25);
        \path[name path=l1, draw, semithick] plot ({\x}, {0});
        \path[name path=l2, draw, semithick] plot ({\x}, {3*\x});
        \path[name path=l3, draw, semithick] plot ({\x}, {6});
        \path[name path=l4, draw, semithick] plot ({\x}, {3*\x+6});
        \path[name intersections={of=l1 and l2,by=x1},draw];
        \path[name intersections={of=l2 and l3,by=x2},draw];
        \path[name intersections={of=l3 and l4,by=x3},draw];
        \path[name intersections={of=l4 and l1,by=x4},draw];
        \fill[black] (0,0) circle (0.06) node[below=0.3cm,right=-0.1cm]{$y_{1}^{\epsilon_{1}}$};
        \fill[black] ({4/3},4) circle (0.06) node[above=0.3cm,right]{$y_{2}^{\epsilon_{1}}$};
        \fill[black] (0,4) circle (0.06) node[above=-0.1cm]{$y_{3}^{\epsilon_{1}}$};
        \fill[black] ({-4/3},0) circle (0.06) node[below ]{$y_{4}^{\epsilon_{1}}$};
        \draw[thick,red] (0,0)--(2,4)--(0,4)--(-2,0)--cycle;
        \fill[lightgray] (0,0)--({4/3},4)--(0,4)--({-4/3},0)--cycle;
        \draw[very thick,black] (0,0)--({4/3},4)--(0,4)--({-4/3},0)--cycle;
        \draw (-2,2.5) node{$Q_{\epsilon_{1}}'$};
      \end{scope}
    \end{tikzpicture}
    \caption{Comparing $M(Q_{\epsilon_{1}})$ and $M(Q_{\epsilon_{2}})$.}
    \label{fig:compare-para-1}
  \end{figure}
  When we consider $T^*\mathbb{R}$ as the complex plane $\mathbb{C}$, we have
  \[
    \operatorname{Re}y_2^{\epsilon_2}>\operatorname{Re}x_2^{\epsilon_2},\operatorname{Im}y_2^{\epsilon_2}<\operatorname{Im}x_2^{\epsilon_2},\operatorname{Re}y_4^{\epsilon_2}<\operatorname{Re}x_4^{\epsilon_2},\operatorname{Im}y_4^{\epsilon_2}>\operatorname{Im}x_4^{\epsilon_2}.
  \]
  From Proposition \ref{mpocm3}, we obtain
  \[
    M(Q_{\epsilon_{1}};x_{1}^{\epsilon_{1}},x_{2}^{\epsilon_{1}},x_{3}^{\epsilon_{1}},x_{4}^{\epsilon_{1}})<M(Q_{\epsilon_{2}}';x_{1}^{\epsilon_{2}},y_{2}^{\epsilon_{2}},x_{3}^{\epsilon_{2}},y_{4}^{\epsilon_{2}})=M(Q_{\epsilon_{2}};x_{1}^{\epsilon_{2}},x_{2}^{\epsilon_{2}},x_{3}^{\epsilon_{2}},x_{4}^{\epsilon_{2}}).
  \]
  When we denote the conformal modulus of $(Q_{\epsilon};x_{4}^{\epsilon},x_{1}^{\epsilon},x_{2}^{\epsilon},x_{3}^{\epsilon})$ as $M(Q_{\epsilon})$, $M(Q_{\epsilon})$ increases monotonically as $\epsilon$ decreases. This is equivalent to that $\xi_{\epsilon}$ decreases monotonically as $\epsilon$ decreases. Also $\xi_{\epsilon}$ is bounded above, then the limit value of $\xi_\epsilon$ exists. We here set $\xi_0\coloneqq \lim_{\epsilon\to+0}\xi_\epsilon$. On the other side, for any $\epsilon>0$, there exists the rhombus $M(Q_{\epsilon}'';x_{4}^{\epsilon\prime},x_{1}^{\epsilon},x_{2}^{\epsilon\prime},x_{3}^{\epsilon})$ such that
  \[
    M(Q_{\epsilon};x_{4}^{\epsilon},x_{1}^{\epsilon},x_{2}^{\epsilon},x_{3}^{\epsilon})<M(Q_{\epsilon}'';x_{4}^{\epsilon\prime},x_{1}^{\epsilon},x_{2}^{\epsilon\prime},x_{3}^{\epsilon})=1.
  \]
  \begin{figure}[tb]
  \center
    \begin{tikzpicture}
      \begin{scope}[scale=0.5]
        \clip (-4,-1.25) rectangle (4,7.25);
        \path[name path=l1, draw, semithick] plot ({\x}, {0});
        \path[name path=l2, draw, semithick] plot ({\x}, {3*\x});
        \path[name path=l3, draw, semithick] plot ({\x}, {6});
        \path[name path=l4, draw, semithick] plot ({\x}, {3*\x+6});
        \path[name intersections={of=l1 and l2,by=x1},draw];
        \path[name intersections={of=l2 and l3,by=x2},draw];
        \path[name intersections={of=l3 and l4,by=x3},draw];
        \path[name intersections={of=l4 and l1,by=x4},draw];
        \fill[black] (x1) circle (0.06) node[below=0.3cm,right=-0.1cm]{$x_{1}^{\epsilon}$};
        \fill[black] (x2) circle (0.06) node[above right]{$x_{2}^{\epsilon}$};
        \fill[black] (x3) circle (0.06) node[above left]{$x_{3}^{\epsilon}$};
        \fill[black] (x4) circle (0.06) node[below left]{$x_{4}^{\epsilon}$};
        \fill[lightgray](x1)--(x2)--(x3)--(x4)--cycle;
        \draw[very thick](x1)--(x2)--(x3)--(x4)--cycle;
        \draw (-2,2.5) node{$Q_{\epsilon}$};
        \draw[thin] (0.5,0) arc (0:{atan(3)}:0.5);
        \draw ({atan(3)/2}:0.5) node[above right]{$\pi\tau$};
      \end{scope}
      \begin{scope}[scale=0.5,xshift=250]
        \clip (-4,-1.25) rectangle (4,7.25);
        \path[name path=l1, draw, semithick] plot ({\x}, {0});
        \path[name path=l2, draw, semithick] plot ({\x}, {3*\x});
        \path[name path=l3, draw, semithick] plot ({\x}, {6});
        \path[name path=l4, draw, semithick] plot ({\x}, {3*\x+6});
        \path[name intersections={of=l1 and l2,by=x1},draw];
        \path[name intersections={of=l2 and l3,by=x2},draw];
        \path[name intersections={of=l3 and l4,by=x3},draw];
        \path[name intersections={of=l4 and l1,by=x4},draw];
        \fill[black] (0,0) circle (0.06) node[below=0.3cm,right=-0.1cm]{$x_{1}^{\epsilon}$};
        \fill[black] (1,3) circle (0.06) node[right]{$x_{2}^{\epsilon\prime}$};
        \fill[black] (0,6) circle (0.06) node[above left]{$x_{3}^{\epsilon}$};
        \fill[black] (-1,3) circle (0.06) node[left]{$x_{4}^{\epsilon\prime}$};
        \fill[cyan] (0,0)--(1,3)--(0,6)--(-1,3)--cycle;
        \draw[very thick,blue] (0,0)--(1,3)--(0,6)--(-1,3)--cycle;
        \draw (0,3) node{$Q_{\epsilon}''$};
      \end{scope}
    \end{tikzpicture}
    \caption{How to make the rhombus $Q_\epsilon''$ in the parallelogram $Q_\epsilon$.}
  \end{figure}
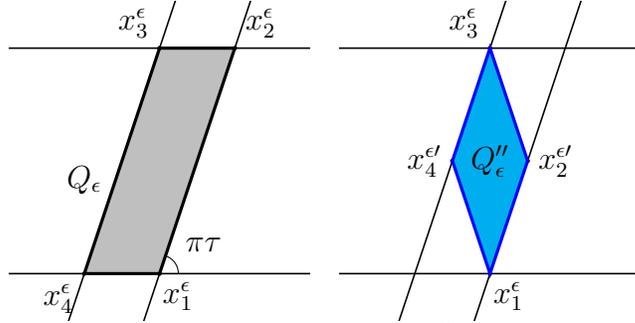
  Then we have $\lim_{\epsilon\to+0}M(Q_\epsilon)\leq1$ and $\xi_0\in[1/2,1)$. We next calculate $\xi_0$. From the integral representation of a Schwarz-Christoffel map, we have 
  \begin{align*}
    \dfrac{\left\lvert x_4^\epsilon-x_1^\epsilon\right\rvert}{\left\lvert x_2^\epsilon-x_1^\epsilon\right\rvert}=\left.\left\lvert \int_{-\infty}^{0}\zeta^{-\theta^\epsilon}(\zeta-\xi_\epsilon)^{\theta^\epsilon-1}(\zeta-1)^{-\theta^\epsilon}\,d\zeta\right\rvert \middle/\left\lvert \int_{0}^{\xi_\epsilon}\zeta^{-\theta^\epsilon}(\zeta-\xi_\epsilon)^{\theta^\epsilon-1}(\zeta-1)^{-\theta^\epsilon}\,d\zeta\right\rvert\right..
  \end{align*}
  We here set 
  \[
    \theta^\epsilon\coloneqq \theta_2^\epsilon-\theta_1^\epsilon,\theta_1^\epsilon=\arctan a_1\epsilon,\theta_2^\epsilon=\arctan a_2\epsilon.
  \]
  We also have 
  \begin{align*}
    \left\lvert \int_{-\infty}^{0}\zeta^{-\theta^\epsilon}(\zeta-\xi_\epsilon)^{\theta^\epsilon-1}(\zeta-1)^{-\theta^\epsilon}\,d\zeta\right\rvert&=\int_{0}^{1}s^{\theta^\epsilon-1}(1-s)^{-\theta^\epsilon}(1-(1-\xi_\epsilon)s)^{\theta^\epsilon-1}\,ds\\
    &=B(\theta^\epsilon,1-\theta^\epsilon)\ _2F_1(\theta^\epsilon,1-\theta^\epsilon,1;1-\xi_\epsilon)
  \end{align*}
  and 
  \begin{align*}
    \left\lvert \int_{0}^{\xi_\epsilon}\zeta^{-\theta^\epsilon}(\zeta-\xi_\epsilon)^{\theta^\epsilon-1}(\zeta-1)^{-\theta^\epsilon}\,d\zeta\right\rvert&=\int_{0}^{1}s^{-\theta^\epsilon}(1-s)^{\theta^\epsilon-1}(1-\xi_\epsilon s)^{-\theta^\epsilon}\,ds\\
    &=B(\theta^\epsilon,1-\theta^\epsilon)\ _2F_1(1-\theta^\epsilon,\theta^\epsilon,1;\xi_\epsilon).
  \end{align*}
  Since $x_4^\epsilon,x_1^\epsilon\in L_1^\epsilon$ and $x_1^\epsilon,x_2^\epsilon\in L_2^\epsilon$, one has
  \[
    \dfrac{\left\lvert x_4^\epsilon-x_1^\epsilon\right\rvert}{\left\lvert x_2^\epsilon-x_1^\epsilon\right\rvert}=\left.\dfrac{\left\lvert p_4-p_1\right\rvert}{\cos\theta_1^\epsilon} \middle/ \dfrac{\left\lvert p_1-p_2\right\rvert}{\cos\theta_2^\epsilon}\right.=\dfrac{\cos\theta_2^\epsilon}{\cos\theta_1^\epsilon}.
  \]
  Then we obtain
  \[
    \cos\theta_2^\epsilon\ _2F_1(1-\theta^\epsilon,\theta^\epsilon,1;\xi_\epsilon)-\cos\theta_1^\epsilon\ _2F_1(\theta^\epsilon,1-\theta^\epsilon,1;1-\xi_\epsilon)=0.
  \]
  By Lemma \ref{2F1conn1}, we have
  \begin{align*}
    &_2F_1(\theta^\epsilon,1-\theta^\epsilon,1;\xi_\epsilon)\\
    &=\dfrac{1}{\Gamma(\theta^\epsilon)\Gamma(1-\theta^\epsilon)}\sum_{n=0}^{\infty}\dfrac{(\theta^\epsilon)_n(1-\theta^\epsilon)_n}{n!n!}(1-\xi_\epsilon)^n\\
    &\hspace{5cm}\cdot\{2\psi(n+1)-\psi(\theta^\epsilon+n)-\psi(1-\theta^\epsilon+n)-\log(1-\xi_\epsilon)\}\\
    &=-\dfrac{\log(1-\xi_\epsilon)}{\Gamma(\theta^\epsilon)\Gamma(1-\theta^\epsilon)}\ _2F_1(\theta^\epsilon,1-\theta^\epsilon,1;1-\xi_\epsilon)+\dfrac{1}{\Gamma(\theta^\epsilon)\Gamma(1-\theta^\epsilon)}\sum_{n=0}^{\infty}\dfrac{(\theta^\epsilon)_n(1-\theta^\epsilon)_n}{n!n!}(1-\xi_\epsilon)^n\\
    &\hspace{9cm}\cdot\{2\psi(n+1)-\psi(\theta^\epsilon+n)-\psi(1-\theta^\epsilon+n)\}
  \end{align*}
  Then we obtain
  \begin{align*}
    &\dfrac{\cos\theta_2^\epsilon}{\Gamma(\theta^\epsilon)\Gamma(1-\theta^\epsilon)}\sum_{n=0}^{\infty}\dfrac{(\theta^\epsilon)_n(1-\theta^\epsilon)_n}{n!n!}(2\psi(n+1)-\psi(\theta^\epsilon+n)-\psi(1-\theta^\epsilon+n))(1-\xi_\epsilon)^n\\
    &\hspace{2cm}-\dfrac{\cos\theta_2^\epsilon\log(1-\xi_\epsilon)}{\Gamma(\theta^\epsilon)\Gamma(1-\theta^\epsilon)}\ _2F_1(\theta^\epsilon,1-\theta^\epsilon,1;1-\xi_\epsilon)-\cos\theta_1^\epsilon\ _2F_1(\theta^\epsilon,1-\theta^\epsilon,1;1-\xi_\epsilon)=0.
  \end{align*}
  This leads to 
  \begin{align*}
    \log(1-\xi_\epsilon)=&\dfrac{1}{_2F_1(\theta^\epsilon,1-\theta^\epsilon,1;1-\xi_\epsilon)}\sum_{n=0}^{\infty}\dfrac{(\theta^\epsilon)_n(1-\theta^\epsilon)_n}{n!n!}(2\psi(n+1)-\psi(1-\theta^\epsilon+n))(1-\xi_\epsilon)^n\\
    &-\dfrac{1}{_2F_1(\theta^\epsilon,1-\theta^\epsilon,1;1-\xi_\epsilon)}\sum_{n=1}^{\infty}\dfrac{(\theta^\epsilon)_n(1-\theta^\epsilon)_n}{n!n!}\psi(\theta^\epsilon+n)(1-\xi_\epsilon)^n\\
    &+\dfrac{1}{\theta^\epsilon\ _2F_1(\theta^\epsilon,1-\theta^\epsilon,1;1-\xi_\epsilon)}-\dfrac{\cos\theta_1^\epsilon}{\theta^\epsilon\cos\theta_2^\epsilon}\Gamma(\theta^\epsilon+1)\Gamma(1-\theta^\epsilon)\\
    &-\dfrac{\psi(\theta^\epsilon+1)}{_2F_1(\theta^\epsilon,1-\theta^\epsilon,1;1-\xi_\epsilon)}.
  \end{align*}
  Since we have $\theta_1^\epsilon,\theta_2^\epsilon,\theta^\epsilon\to0\ (\epsilon\to+0)$, the first three terms of the most-right hand side tends to zero as $\epsilon\to+0$. Therefore, we here calculate the limit value of
  \begin{align*}
    &\dfrac{1}{\theta^\epsilon\ _2F_1(\theta^\epsilon,1-\theta^\epsilon,1;1-\xi_\epsilon)}-\dfrac{\cos\theta_1^\epsilon}{\theta^\epsilon\cos\theta_2^\epsilon}\Gamma(\theta^\epsilon+1)\Gamma(1-\theta^\epsilon)\\
    &=\dfrac{1}{\cos\theta_2^\epsilon\ _2F_1(\theta^\epsilon,1-\theta^\epsilon,1;1-\xi_\epsilon)}\\
    &\hspace{2cm}\cdot\dfrac{\cos\theta_2^\epsilon-\cos\theta_1^\epsilon\Gamma(\theta^\epsilon+1)\Gamma(1-\theta^\epsilon)\ _2F_1(\theta^\epsilon,1-\theta^\epsilon,1;1-\xi_\epsilon)}{\theta^\epsilon}
  \end{align*}
  We first obtain
  \begin{align}
    \label{cm-para-1}
    \lim_{\epsilon\to+0}\dfrac{1}{\cos\theta_2^\epsilon\ _2F_1(\theta^\epsilon,1-\theta^\epsilon,1;1-\xi_\epsilon)}=1.
  \end{align}
  One has
  \begin{align}
    &\dfrac{\cos\theta_2^\epsilon-\cos\theta_1^\epsilon\Gamma(\theta^\epsilon+1)\Gamma(1-\theta^\epsilon)\ _2F_1(\theta^\epsilon,1-\theta^\epsilon,1;1-\xi_\epsilon)}{\theta^\epsilon}\notag\\
    &=\dfrac{\cos\theta_2^\epsilon-\cos\theta_1^\epsilon\Gamma(\theta^\epsilon+1)\Gamma(1-\theta^\epsilon)}{\theta^\epsilon}-\cos\theta_1^\epsilon\Gamma(\theta^\epsilon+1)\Gamma(1-\theta^\epsilon)\sum_{n=1}^{\infty}\dfrac{(\theta^\epsilon+1)_{n-1}(1-\theta^\epsilon)_n}{n!n!}(1-\xi_\epsilon)^n.
    \label{cm-para-2}
  \end{align}
  We obtain
  \begin{align}
    &\lim_{\epsilon\to+0}\cos\theta_1^\epsilon\Gamma(\theta^\epsilon+1)\Gamma(1-\theta^\epsilon)\sum_{n=1}^{\infty}\dfrac{(\theta^\epsilon+1)_{n-1}(1-\theta^\epsilon)_n}{n!n!}(1-\xi_\epsilon)^n\notag\\
    &\hspace{7cm}=\sum_{n=1}^{\infty}\dfrac{1}{n}(1-\xi_0)^n=-\log\xi_0.\label{cm-para-3}
  \end{align}
  One also has
  \begin{align}
    &\dfrac{\cos\theta_2^\epsilon-\cos\theta_1^\epsilon\Gamma(\theta^\epsilon+1)\Gamma(1-\theta^\epsilon)}{\theta^\epsilon}\notag\\
    &=\dfrac{\cos^2\theta_2^\epsilon-\cos^2\theta_1^\epsilon(\Gamma(\theta^\epsilon+1)\Gamma(1-\theta^\epsilon))^2}{\theta^\epsilon}\dfrac{1}{\cos\theta_2^\epsilon+\cos\theta_1^\epsilon\Gamma(\theta^\epsilon+1)\Gamma(1-\theta^\epsilon)}\notag\\
    &=\left\{\dfrac{1-\Gamma(\theta^\epsilon+1)^2\Gamma(1-\theta^\epsilon)^2}{\theta^\epsilon}-\left(\dfrac{\sin\theta_2^\epsilon}{\theta_2^\epsilon}\right)^2\left(\dfrac{\theta_2^\epsilon}{\epsilon}\right)^2\dfrac{\epsilon^2}{\theta^\epsilon}+\left(\dfrac{\sin\theta_1^\epsilon}{\theta_1^\epsilon}\right)^2\left(\dfrac{\theta_1^\epsilon}{\epsilon}\right)^2\dfrac{\epsilon^2}{\theta^\epsilon}\right\}\notag\\
    &\hspace{3cm}\cdot\dfrac{1}{\cos\theta_2^\epsilon+\cos\theta_1^\epsilon\Gamma(\theta^\epsilon+1)\Gamma(1-\theta^\epsilon)}\label{cm-para-4}
  \end{align}
  The function $g(x)=1-\Gamma(x+1)^2\Gamma(1-x)^2\ (x\in\mathbb{R})$ is smooth in the neighborhood of $x=0$ and satisfies $g(0)=0$. Then we obtain 
  \begin{align}
    &\lim_{\epsilon\to+0}\dfrac{1-\Gamma(\theta^\epsilon+1)^2\Gamma(1-\theta^\epsilon)^2}{\theta^\epsilon}=\dfrac{dg}{dx}(0)\notag\\
    &=\left.(-2\Gamma(x+1)\Gamma(1-x)^2+2\Gamma(1-x)\Gamma(x+1)^2)\right|_{x=0}=0.\label{cm-para-5}
  \end{align}
  From \eqref{cm-para-1}-\eqref{cm-para-5}, we finally obtain
  \[
    \lim_{\epsilon\to+0}\left(\dfrac{1}{\theta^\epsilon\ _2F_1(\theta^\epsilon,1-\theta^\epsilon,1;1-\xi_\epsilon)}-\dfrac{\cos\theta_1^\epsilon}{\theta^\epsilon\cos\theta_2^\epsilon}\Gamma(\theta^\epsilon+1)\Gamma(1-\theta^\epsilon)\right)=\log\xi_0.
  \]
  We therefore obtain $\log(1-\xi_0)=\log\xi_0$, and this leads to $\xi_0=1/2$. We can also consider the case $a_{1}<a_{2}\leq0$ in a similar way. When $a_{1}<0<a_{2}$ holds, we can take two parallelograms $(Q_{\epsilon}';x_{1}^{\epsilon},x_{2}^{\epsilon\prime},x_{3}^{\epsilon},x_{4}^{\epsilon\prime})$ and $(Q_{\epsilon}'';x_{1}^{\epsilon},x_{2}^{\epsilon\prime\prime},x_{3}^{\epsilon},x_{4}^{\epsilon\prime\prime})$ which edges are parallel to $x$-axis such that
  \[
    M(Q_{\epsilon}';x_{1}^{\epsilon},x_{2}^{\epsilon\prime},x_{3}^{\epsilon},x_{4}^{\epsilon\prime})<M(Q_{\epsilon};x_{1}^{\epsilon},x_{2}^{\epsilon},x_{3}^{\epsilon},x_{4}^{\epsilon})<M(Q_{\epsilon}'';x_{1}^{\epsilon},x_{2}^{\epsilon\prime\prime},x_{3}^{\epsilon},x_{4}^{\epsilon\prime\prime})
  \]
  by using Proposition \ref{mpocm3}.
  \begin{figure}[tb]
    \center
    \begin{tikzpicture}
      \begin{scope}[scale=0.5]
        \clip (-6,-1.5) rectangle (2,6);
        \path[name path=l1, draw, semithick] plot ({\x}, {\x+2});
        \path[name path=l2, draw, semithick] plot ({\x}, {(-5/3)*\x+5/3});
        \path[name path=l3, draw, semithick] plot ({\x}, {\x+7});
        \path[name path=l4, draw, semithick] plot ({\x}, {(-3/2)*\x-3});
        \draw plot ({\x}, {0});
        \path[name intersections={of=l1 and l2,by=x1},draw];
        \path[name intersections={of=l2 and l3,by=x2},draw];
        \path[name intersections={of=l3 and l4,by=x3},draw];
        \path[name intersections={of=l4 and l1,by=x4},draw];
        \fill[black] (x1) circle (0.06) node[below]{$x_{2}^{\epsilon}$};
        \fill[black] (x2) circle (0.06) node[right]{$x_{3}^{\epsilon}$};
        \fill[black] (x3) circle (0.06) node[left]{$x_{4}^{\epsilon}$};
        \fill[black] (x4) circle (0.06) node[below]{$x_{1}^{\epsilon}$};
        \fill[lightgray](x1)--(x2)--(x3)--(x4)--cycle;
        \draw[very thick](x1)--(x2)--(x3)--(x4)--cycle;
        \draw (-2,2.5) node{$Q_{\epsilon}$};
        \draw plot ({\x}, {(-5/3)*\x+5/3});
      \end{scope}
      \begin{scope}[scale=0.5,xshift=230]
        \clip (-6,-1.5) rectangle (2,6);
        \path[name path=l1, draw, semithick] plot ({\x}, {\x+2});
        \path[name path=l2, draw, semithick] plot ({\x}, {(-5/3)*\x+5/3});
        \path[name path=l3, draw, semithick] plot ({\x}, {\x+7});
        \path[name path=l4, draw, semithick] plot ({\x}, {(-3/2)*\x-3});
        \draw plot ({\x}, {0});
        \path[name intersections={of=l1 and l2,by=x1},draw];
        \path[name intersections={of=l2 and l3,by=x2},draw];
        \path[name intersections={of=l3 and l4,by=x3},draw];
        \path[name intersections={of=l4 and l1,by=x4},draw];
        \fill[black] (x1) circle (0.06);
        \fill[black] (x2) circle (0.06) node[left]{$x_{3}^{\epsilon}$};
        \fill[black] (x3) circle (0.06);
        \fill[black] (x4) circle (0.06) node[below]{$x_{1}^{\epsilon}$};
        \fill[cyan](0,5)--(x2)--(-4,0)--(x4)--cycle;
        \draw[very thick,blue](0,5)--(x2)--(-4,0)--(x4)--cycle;
        \draw (-2,2.5) node{$Q_{\epsilon}'$};
        \draw (0,5) node[right]{$x_{2}^{\epsilon\prime}$};
        \draw (-4,0) node[below]{$x_{4}^{\epsilon\prime}$};
        \draw plot ({\x}, {(-3/2)*\x-3});
        \draw plot ({\x}, {(-5/3)*\x+5/3});
      \end{scope}
      \begin{scope}[scale=0.5,xshift=460]
        \clip (-6,-1.5) rectangle (2,6);
        \path[name path=l1, draw, semithick] plot ({\x}, {\x+2});
        \path[name path=l2, draw, semithick] plot ({\x}, {(-5/3)*\x+5/3});
        \path[name path=l3, draw, semithick] plot ({\x}, {\x+7});
        \path[name path=l4, draw, semithick] plot ({\x}, {(-3/2)*\x-3});
        \draw plot ({\x}, {0});
        \path[name intersections={of=l1 and l2,by=x1},draw];
        \path[name intersections={of=l2 and l3,by=x2},draw];
        \path[name intersections={of=l3 and l4,by=x3},draw];
        \path[name intersections={of=l4 and l1,by=x4},draw];
        \fill[black] (x1) circle (0.06);
        \fill[black] (x2) circle (0.06) node[right]{$x_{3}^{\epsilon}$};
        \fill[black] (x3) circle (0.06);
        \fill[black] (x4) circle (0.06) node[below]{$x_{1}^{\epsilon}$};
        \fill[pink](0,0)--(x2)--(-4,5)--(x4)--cycle;
        \draw[very thick,red](0,0)--(x2)--(-4,5)--(x4)--cycle;
        \draw (0,0) node[below]{$x_{2}^{\epsilon\prime\prime}$};
        \draw (-4,5) node[left]{$x_{4}^{\epsilon\prime\prime}$};
        \draw (-2,2.5) node{$Q_{\epsilon}''$};
        \draw plot({\x}, {\x+7});
        \draw plot ({\x}, {\x+2});
      \end{scope}
    \end{tikzpicture}
    \caption{How to make $Q_{\epsilon}',Q_{\epsilon}''$. (Horizontal line is $x$-axis, and other lines are $L_{1}^{\epsilon},\dots,L_{4}^{\epsilon}$.)}
    \label{1case}
  \end{figure}
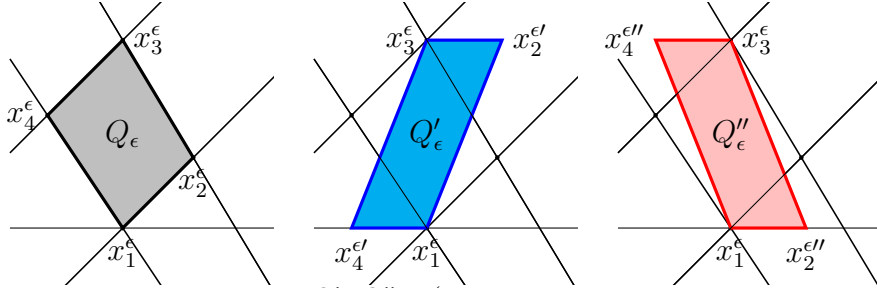
  Limit values of conformal moduli of parallelograms $Q_\epsilon'$ and $Q_\epsilon''$ are $1$, which we discussed in the case $a_{1}<a_{2}\leq0$ or $0\leq a_{1}<a_{2}$. Therefore we obtain $M(Q_{\epsilon};x_{1}^{\epsilon},x_{2}^{\epsilon},x_{3}^{\epsilon},x_{4}^{\epsilon})\rightarrow1$ as $\epsilon\to+0$.
\end{proof}

By using Lemma \ref{parallel}, we can prove Lemma \ref{square-cm2} in the case $(p_3-p_1)(p_4-p_2)=0$. We here assume $p_1=p_3$. We can take two parallelogram $(Q_{\epsilon}';x_{1}^{\epsilon},x_{2}^{\epsilon\prime},x_{3}^{\epsilon},x_{4}^{\epsilon\prime})$ and $(Q_{\epsilon}'';x_{1}^{\epsilon},x_{2}^{\epsilon\prime\prime},x_{3}^{\epsilon},x_{4}^{\epsilon\prime\prime})$ such that 
\[
  M(Q_{\epsilon}';x_{1}^{\epsilon},x_{2}^{\epsilon\prime},x_{3}^{\epsilon},x_{4}^{\epsilon\prime})<M(Q_{\epsilon};x_{1}^{\epsilon},x_{2}^{\epsilon},x_{3}^{\epsilon},x_{4}^{\epsilon})<M(Q_{\epsilon}'';x_{1}^{\epsilon},x_{2}^{\epsilon\prime\prime},x_{3}^{\epsilon},x_{4}^{\epsilon\prime\prime})
\]
By Lemma \ref{parallel}, we obtain 
\[
  \lim_{\epsilon\to+0}M(Q_{\epsilon};x_{1}^{\epsilon},x_{2}^{\epsilon},x_{3}^{\epsilon},x_{4}^{\epsilon})=1,
\]
which leads to Lemma \ref{square-cm2}. We can also discuss in a similar way in the case $p_2=p_4$.\\
We here set $z_4=1/2$, and identify the upper half plane with the unit disk. If $z_{4,\epsilon}$ tends to $z_{3}$ or $z_{1}$, then the boundary collision happens, and the flat bubble arises from the unit disk $[z_{1},z_{2},z_{3},z_{4,\epsilon}]$ with four marked points. It is known that these boundary collisions correspond to trees as Figure \ref{disktree}. On the other side, the unit disk $[z_1,z_2,z_3,z_4]$ corresponds to the tree as Figure \ref{disktree} (see \cite{Dev12}). 
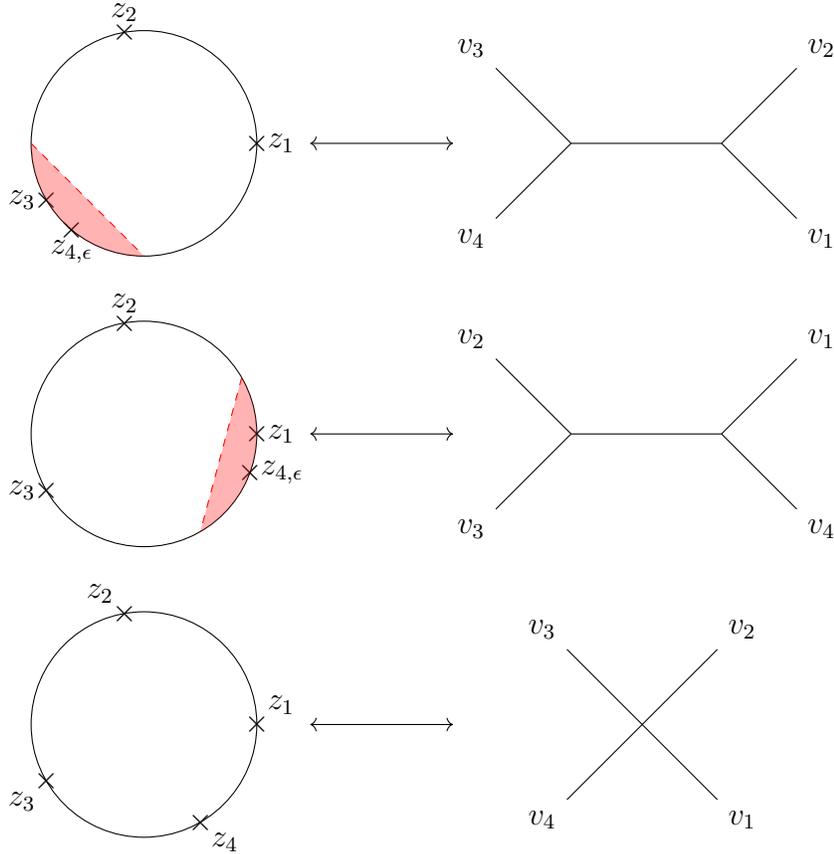
\begin{figure}[tb]
  \centering
  \begin{tikzpicture}
    \draw (0,0) circle [radius=1.5];
    \draw (0:1.5) node{$\times$};
    \draw (100:1.5) node{$\times$};
    \draw (210:1.5) node{$\times$};
    \draw (230:1.5) node{$\times$};
    \draw (0:1.5) node[right]{$z_{1}$};
    \draw (100:1.5) node[above]{$z_{2}$};
    \draw (210:1.5) node[left]{$z_{3}$};
    \draw (230:1.5) node[below]{$z_{4,\epsilon}$};
    \draw[red,dashed] (180:1.5)--(270:1.5);
    \fill[red,opacity=0.3] (180:1.5) arc (180:270:1.5)--(270:1.5)--cycle;
    \begin{scope}[xshift=60]
      \draw[<->] (0.1,0)--(2,0);
    \end{scope}
    \begin{scope}[xshift=190]
      \draw[thin] (2,1)--(1,0);
      \draw[thin] (2,-1)--(1,0);
      \draw[thin] (-1,0)--(1,0);
      \draw[thin] (-2,1)--(-1,0);
      \draw[thin] (-2,-1)--(-1,0);
      \draw (-2,1) node[above left]{$v_{3}$};
      \draw (-2,-1) node[below left]{$v_{4}$};
      \draw (2,-1) node[below right]{$v_{1}$};
      \draw (2,1) node[above right]{$v_{2}$};
    \end{scope}
    \begin{scope}[yshift=-110]
      \draw (0,0) circle [radius=1.5];
      \draw (0:1.5) node{$\times$};
      \draw (100:1.5) node{$\times$};
      \draw (210:1.5) node{$\times$};
      \draw (-20:1.5) node{$\times$};
      \draw (0:1.5) node[right]{$z_{1}$};
      \draw (100:1.5) node[above]{$z_{2}$};
      \draw (210:1.5) node[left]{$z_{3}$};
      \draw (-20:1.5) node[right]{$z_{4,\epsilon}$};
      \draw[red,dashed] (30:1.5)--(-60:1.5);
      \fill[red,opacity=0.3] (-60:1.5) arc (-60:30:1.5)--(30:1.5)--cycle;
    \end{scope}
    \begin{scope}[xshift=60,yshift=-110]
      \draw[<->] (0.1,0)--(2,0);
    \end{scope}
    \begin{scope}[xshift=190,yshift=-110]
      \draw[thin] (2,1)--(1,0);
      \draw[thin] (2,-1)--(1,0);
      \draw[thin] (-1,0)--(1,0);
      \draw[thin] (-2,1)--(-1,0);
      \draw[thin] (-2,-1)--(-1,0);
      \draw (-2,1) node[above left]{$v_{2}$};
      \draw (-2,-1) node[below left]{$v_{3}$};
      \draw (2,-1) node[below right]{$v_{4}$};
      \draw (2,1) node[above right]{$v_{1}$};
    \end{scope}
    \begin{scope}[yshift=-220]
      \draw (0,0) circle [radius=1.5];
      \draw (0:1.5) node{$\times$};
      \draw (100:1.5) node{$\times$};
      \draw (210:1.5) node{$\times$};
      \draw (-60:1.5) node{$\times$};
      \draw (0:1.5) node[above right]{$z_{1}$};
      \draw (100:1.5) node[above left]{$z_{2}$};
      \draw (210:1.5) node[below left]{$z_{3}$};
      \draw (-60:1.5) node[below right]{$z_{4}$};
    \end{scope}
    \begin{scope}[xshift=60,yshift=-220]
      \draw[<->] (0.1,0)--(2,0);
    \end{scope}
    \begin{scope}[xshift=160,yshift=-220]
      \draw[thin] (2,1)--(1,0);
      \draw[thin] (2,-1)--(1,0);
      \draw[thin] (0,1)--(1,0);
      \draw[thin] (0,-1)--(1,0);
      \draw (0,1) node[above left]{$v_{3}$};
      \draw (0,-1) node[below left]{$v_{4}$};
      \draw (2,-1) node[below right]{$v_{1}$};
      \draw (2,1) node[above right]{$v_{2}$};
    \end{scope}
  \end{tikzpicture}
  \caption{Correspondence between boundary collisions and trees (The first line and the second line are reproduced from \cite{S25}).}
  \label{disktree}
\end{figure}
By comparing Table \ref{tree-critpt2} and the results in Lemma \ref{square-cm1} and \ref{square-cm2}, we obtain the following lemma.
\begin{lem}
  \label{4thpt-tree}
  Let $f_1,f_2,f_3,f_4$ be functions such that $(f_1,f_2,f_3,f_4)$ is not generic. Then the tree which corresponds to the deformation of the upper half plane with four points $z_{1},z_{2},z_{3},z_{4,\epsilon}$ by $\epsilon\to+0$ is isomorphic to the tree which is used to construct the gradient tree $I\in\mathcal{M}_{g}(\mathbb{R};\vec{f},\vec{p})$.
\end{lem}
We next study the length $l>0$ of the internal edge in the case $(p_3-p_1)(p_4-p_2)\neq0$. In \cite{S25}, we assumed $(f_1,f_2,f_3,f_4)$ is generic, and we proved that the length $l$ was induced by considering the behavior of $z_{4,\epsilon}$ in Lemma 4.7. In this section, we prove that Lemma 4.7 in \cite{S25} still holds even if the convex quadrilateral $x_1^\epsilon x_2^\epsilon x_3^\epsilon x_4^\epsilon$ has a pair of parallel opposite sides.
\begin{lem}
  \label{length-4thpt}
  We assume that $(f_1,f_2,f_3,f_4)$ is not generic. If $(p_{3}-p_{1})(p_{4}-p_{2})>0$ (resp. $(p_{3}-p_{1})(p_{4}-p_{2})<0$) holds, then we have
  \[
    \log z_{4,\epsilon}\sim -\dfrac{\pi l}{\epsilon}\ \left(\text{resp.}\ \log(1-z_{4,\epsilon})\sim -\dfrac{\pi l}{\epsilon}\right)\ (\epsilon\to+0).
  \]
\end{lem}
\begin{proof}
  We only prove in the case $(p_3-p_1)(p_4-p_2)>0$.
  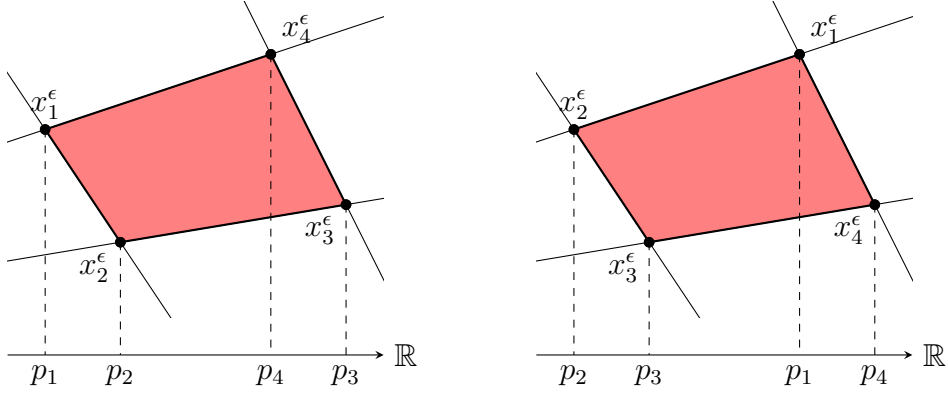
\begin{figure}[tb]
    \centering
    \begin{tikzpicture}
    \begin{scope}
      \clip (-0.5,0.5) rectangle (4.5,4.7);
      \draw [name path=L1] plot(\x,\x/3+3);
      \draw [name path=L2] plot(\x,{\x*(-3/2)+3});
      \draw [name path=L3] plot(\x,\x/6+4/3);
      \draw [name path=L4] plot(\x,-2*\x+10);
      \fill [black, name intersections={of=L1 and L2, by={x1}}] (x1) circle (2pt) node [above] at (x1) {$x_{1}^{\epsilon}$};
      \fill [black, name intersections={of=L2 and L3, by={x2}}] (x2) circle (2pt) node [below left] at (x2) {$x_{2}^{\epsilon}$};
      \fill [black, name intersections={of=L3 and L4, by={x3}}] (x3) circle (2pt) node [below left] at (x3) {$x_{3}^{\epsilon}$};
      \fill [black, name intersections={of=L4 and L1, by={x4}}] (x4) circle (2pt) node [above right] at (x4) {$x_{4}^{\epsilon}$};
      \fill[red,opacity=0.5] (x1)--(x2)--(x3)--(x4)--cycle;
      \draw[thick] (x1)--(x2)--(x3)--(x4)--cycle;
      \fill[black] (x1) circle (2pt);
      \fill[black] (x2) circle (2pt);
      \fill[black] (x3) circle (2pt);
      \fill[black] (x4) circle (2pt);
    \end{scope}
    \draw[->,>=stealth] (-0.5,0)--(4.5,0) node[right]{$\mathbb{R}$};
    \draw[dashed] (0,3)--(0,0) node[below]{$p_{1}$};
    \draw[dashed] (1,1.5)--(1,0) node[below]{$p_{2}$};
    \draw[dashed] (4,2)--(4,0) node[below]{$p_{3}$};
    \draw[dashed] (3,4)--(3,0) node[below]{$p_{4}$};
    \begin{scope}[xshift=200]
      \begin{scope}
        \clip (-0.5,0.5) rectangle (4.5,4.7);
        \draw [name path=L1] plot(\x,\x/3+3);
        \draw [name path=L2] plot(\x,{\x*(-3/2)+3});
        \draw [name path=L3] plot(\x,\x/6+4/3);
        \draw [name path=L4] plot(\x,-2*\x+10);
        \fill [black, name intersections={of=L1 and L2, by={x1}}] (x1) circle (2pt) node [above] at (x1) {$x_{2}^{\epsilon}$};
        \fill [black, name intersections={of=L2 and L3, by={x2}}] (x2) circle (2pt) node [below left] at (x2) {$x_{3}^{\epsilon}$};
        \fill [black, name intersections={of=L3 and L4, by={x3}}] (x3) circle (2pt) node [below left] at (x3) {$x_{4}^{\epsilon}$};
        \fill [black, name intersections={of=L4 and L1, by={x4}}] (x4) circle (2pt) node [above right] at (x4) {$x_{1}^{\epsilon}$};
        \fill[red,opacity=0.5] (x1)--(x2)--(x3)--(x4)--cycle;
        \draw[thick] (x1)--(x2)--(x3)--(x4)--cycle;
        \fill[black] (x1) circle (2pt);
        \fill[black] (x2) circle (2pt);
        \fill[black] (x3) circle (2pt);
        \fill[black] (x4) circle (2pt);
      \end{scope}
      \draw[->,>=stealth] (-0.5,0)--(4.5,0) node[right]{$\mathbb{R}$};
      \draw[dashed] (0,3)--(0,0) node[below]{$p_{2}$};
      \draw[dashed] (1,1.5)--(1,0) node[below]{$p_{3}$};
      \draw[dashed] (4,2)--(4,0) node[below]{$p_{4}$};
      \draw[dashed] (3,4)--(3,0) node[below]{$p_{1}$};
    \end{scope}
    \end{tikzpicture}
    \caption{Figure of quadrilaterals $x_{1}^{\epsilon}x_{2}^{\epsilon}x_{3}^{\epsilon}x_{4}^{\epsilon}$ corresponding to the case $(p_{3}-p_{1})(p_{4}-p_{2})>0$ (the left hand side) and $(p_{3}-p_{1})(p_{4}-p_{2})<0$ (the right hand side). (Reproduced from \cite{S25}.)}
    \label{const-1}
  \end{figure}
  When $\alpha_1^\epsilon+\alpha_2^\epsilon\neq1$ holds, i.e., two sides $x_1^\epsilon x_4^\epsilon$ and $x_2^\epsilon x_3^\epsilon$ of the quadrilateral $x_1^\epsilon x_2^\epsilon x_3^\epsilon x_4^\epsilon$ (the left hand side of Figure \ref{const-1}) are parallel, we can prove Lemma \ref{length-4thpt} in the same way as in \cite{S25} even if $x_1^\epsilon x_2^\epsilon$ and $x_4^\epsilon x_3^\epsilon$ are parallel. Thus we concentrate on the case $\alpha_1^\epsilon+\alpha_2^\epsilon=1$ only. By Lemma 3.1, we have
  \begin{equation}
    \label{SC-ac-1}
    \begin{aligned}
      w(z)&=x_{1}+\dfrac{\sin\pi\alpha_1}{\pi}\dfrac{x_{2}-x_{1}}{_{2}F_{1}(\alpha_{1},1,1-\alpha_{3};\xi)}\\
      &\hspace{0.5cm}\cdot\left[-e^{\pi\alpha_1i}\sum_{\substack{k,l\geq0\\k\neq l}}\dfrac{(\alpha)_k(1-\alpha_3)_l}{k!l!}\dfrac{1}{k-l}\left(\dfrac{z}{\xi}\right)^{-k}z^l\right.\\
      &\hspace{0.5cm}\left.+e^{\pi\alpha_1i}\sum_{n\geq0}\left[\psi(n+1)-\psi(\alpha_1+n)+\log z-\log\xi-\pi i\right]\dfrac{(\alpha_1)_n(1-\alpha_3)_n}{n!n!}\xi^n\right].
    \end{aligned}
  \end{equation}
  for $\xi<\left\lvert z\right\rvert<1$ and $\left\lvert \arg(-z/\xi)\right\rvert<\pi$. By combining  Lemma \ref{2F1conn2} and Lemma 4.3 (\cite{S25}), we also have 
  \begin{align*}
    w(z)&=x_{4}+e^{\pi(\alpha_3-1)i}\dfrac{\sin\pi\alpha_3}{\pi}\dfrac{x_{3}-x_{4}}{_{2}F_{1}(\alpha_{1},1-\alpha_{3},1;\xi)}\left[-\sum_{\substack{k,l\geq0\\k\neq l}}\dfrac{(1-\alpha_3)_k(\alpha_1)_l}{k!l!}\dfrac{1}{k-l}\left(\dfrac{1}{z}\right)^{-k}\left(\dfrac{\xi}{z}\right)^l\right.\\
    &\left.\hspace{1.5cm}+\sum_{n\geq0}\left[\psi(n+1)-\psi(1-\alpha_3+n)+\log\left(-\dfrac{1}{z}\right)\right]\dfrac{(\alpha_1)_n(1-\alpha_3)_n}{n!n!}\xi^n\right].
  \end{align*}
  for $\xi<\left\lvert z\right\rvert<1$ and $\left\lvert \arg(-1/z)\right\rvert<\pi$. Here, $\arg(-1/z)=0$ holds when $-\infty<1/z<\xi/z<0$, and also $\arg(1/z)=0$ holds when $0<\xi/z<1/z<\infty$. Then, we have $-1/z=1/z\cdot \exp[\pi i]$. Therefore, we obtain
  \begin{align*}
    w(z)&=x_{4}+e^{\pi(\alpha_3-1)i}\dfrac{\sin\pi\alpha_3}{\pi}\dfrac{x_{3}-x_{4}}{_{2}F_{1}(\alpha_{1},1-\alpha_{3},1;\xi)}\\
    &\hspace{0.75cm}\cdot\left[-\sum_{\substack{k,l\geq0\\k\neq l}}\dfrac{(1-\alpha_3)_k(\alpha_1)_l}{k!l!}\dfrac{1}{k-l}\left(\dfrac{1}{z}\right)^{-k}\left(\dfrac{\xi}{z}\right)^l\right.\\
    &\left.\hspace{1.5cm}+\sum_{n\geq0}\left[\psi(n+1)-\psi(1-\alpha_3+n)-\log{z}+\pi i\right]\dfrac{(\alpha_1)_n(1-\alpha_3)_n}{n!n!}\xi^n\right].
  \end{align*}
  We also have 
  \begin{align*}
    (x_2-x_1)e^{\pi\alpha_1i}\sin\pi\alpha_1&=\left\lvert x_2-x_1\right\rvert\sin\pi\alpha_1 \dfrac{x_4-x_1}{\left\lvert x_4-x_1\right\rvert}\\
    &=\left\lvert x_4-x_3\right\rvert\sin\pi\alpha_3\dfrac{x_4-x_1}{\left\lvert x_4-x_1\right\rvert}\\
    &=(x_4-x_3)e^{\pi(\alpha_3-1)i}\sin\pi\alpha_3e^{-\pi\alpha_1i}.
  \end{align*}
  Then, one obtains the following:
  \begin{equation}
    \label{SC-ac-2}
    \begin{aligned}
      w(z)&=x_{4}-e^{\pi\alpha_1i}\dfrac{\sin\pi\alpha_1}{\pi}\dfrac{x_{2}-x_{1}}{_{2}F_{1}(\alpha_{1},1-\alpha_{3},1;\xi)}\\
      &\hspace{0.75cm}\cdot\left[-\sum_{\substack{k,l\geq0\\k\neq l}}\dfrac{(1-\alpha_3)_k(\alpha_1)_l}{k!l!}\dfrac{1}{k-l}\left(\dfrac{1}{z}\right)^{-k}\left(\dfrac{\xi}{z}\right)^l\right.\\
      &\left.\hspace{1.5cm}+\sum_{n\geq0}\left[\psi(n+1)-\psi(1-\alpha_3+n)-\log{z}+\pi i\right]\dfrac{(\alpha_1)_n(1-\alpha_3)_n}{n!n!}\xi^n\right].
    \end{aligned}
  \end{equation}
  The Schwarz-Christoffel map $w$ is defined on the upper half plane i.e. $0\le\arg z\le\pi$. Also, \eqref{SC-ac-1} and \eqref{SC-ac-2} hold when $0<\arg z<\pi$. Then, we obtain
  \begin{align*}
    &x_{1}+\dfrac{\sin\pi\alpha_1}{\pi}\dfrac{x_{2}-x_{1}}{_{2}F_{1}(\alpha_{1},1,1-\alpha_{3};\xi)}e^{\pi\alpha_1i}\sum_{n\geq0}\left[\psi(n+1)-\psi(\alpha_1+n)-\log\xi\right]\dfrac{(\alpha_1)_n(1-\alpha_3)_n}{n!n!}\xi^n\\
    &=x_{4}-e^{\pi\alpha_1i}\dfrac{\sin\pi\alpha_1}{\pi}\dfrac{x_{2}-x_{1}}{_{2}F_{1}(\alpha_{1},1-\alpha_{3},1;\xi)}\sum_{n\geq0}\left[\psi(n+1)-\psi(1-\alpha_3+n)\right]\dfrac{(\alpha_1)_n(1-\alpha_3)_n}{n!n!}\xi^n.
  \end{align*}
  This leads to the following:
  \begin{align*}
    -\dfrac{\epsilon}{\pi}\log\xi=&\dfrac{x_4-x_1}{x_2-x_1}e^{-\pi\alpha_1i}\dfrac{\epsilon}{\sin\pi\alpha_1}\\
    &-\dfrac{\epsilon}{\pi}\sum_{n\geq0}\dfrac{2\psi(n+1)-\psi(\alpha_1+n)-\psi(1-\alpha_3+n)}{_2F_1(\alpha_1,1-\alpha_3,1;\xi)}\dfrac{(\alpha_1)_n(1-\alpha_3)_n}{n!n!}\xi^n.
  \end{align*}
  We here assume $\alpha_1\to0,\alpha_3\to1$. Only two cases 
  \begin{align*}
    p_1<p_2<p_3<p_4,\ a_2<a_1=a_3<a_4\\
    p_1>p_2>p_3>p_4,\ a_2<a_1=a_3<a_4
  \end{align*}
  satisfy these conditions. By Lemma \ref{angleslim}, we have
  \[
    \lim_{\epsilon\to+0}\dfrac{\epsilon}{\sin\pi\alpha_1}=\lim_{\epsilon\to+0}\dfrac{\pi\alpha_1}{\sin\pi\alpha_1}\dfrac{\epsilon}{\pi\alpha_1}=\dfrac{1}{a_1-a_2}.
  \]
  By the property of $\psi$ function
  \[
    \lim_{z\to0}z\psi(z)=\lim_{z\to0}(z\psi(z+1)-1)=-1,
  \]
  we also have
  \[
    \lim_{\epsilon\to+0}\epsilon\psi(\alpha_1)=\lim_{\epsilon\to+0}\left(\dfrac{\epsilon}{\alpha_1}\cdot\alpha_1\psi(\alpha_1)\right)=-\dfrac{\pi}{a_1-a_2},\ \lim_{\epsilon\to+0}\epsilon\psi(1-\alpha_3)=-\dfrac{\pi}{ a_4-a_1}.
  \]
  Therefore, we obtain
  \begin{align*}
    &\lim_{\epsilon\to+0}-\dfrac{\epsilon}{\pi}\log\xi=\dfrac{p_4-p_1}{p_2-p_1}\dfrac{1}{a_1-a_2}-\dfrac{1}{a_1-a_2}-\dfrac{1}{a_4-a_1}\\
    &=\left.(p_4-p_1)\middle/\left\{\left(-\dfrac{b_3-b_2}{a_1-a_2}+\dfrac{b_2-b_1}{a_2-a_1}\right)\cdot(a_1-a_2)\right\}\right.-\dfrac{1}{a_1-a_2}-\dfrac{1}{a_4-a_1}\\
    &=-\dfrac{p_4-p_1}{b_3-b_1}-\dfrac{1}{a_1-a_2}-\dfrac{1}{a_4-a_1}
  \end{align*}
  On the other side, we have 
  \[
    I_{int}(t)=-(b_3-b_1)t+p_2,
  \]
  where $I_{int}$ is the restriction of the gradient tree $I$ to the internal edge. There uniquely exists the positive real number $l>0$ such that $I_{int}(l)=p_3$. Then, we obtain
  \begin{align*}
    l&=-\dfrac{p_3-p_2}{b_3-b_1}=-\left(-\dfrac{b_4-b_3}{a_4-a_1}+\dfrac{b_3-b_2}{a_1-a_2}\right)\cdot\dfrac{1}{b_3-b_1}\\
    &=-\left(p_4-\dfrac{b_1-b_3}{a_4-a_1}-p_1+\dfrac{b_3-b_1}{a_1-a_2}\right)\cdot\dfrac{1}{b_3-b_1}\\
    &=-\dfrac{p_4-p_1}{b_3-b_1}-\dfrac{1}{a_4-a_1}-\dfrac{1}{a_1-a_2}=\lim_{\epsilon\to+0}-\dfrac{\epsilon}{\pi}\log \xi.
  \end{align*}
  The proof of other cases is analogous to the above.
\end{proof}
\subsection{Proof of Theorem \ref{k4corr1}.}
We here prove Theorem \ref{k4corr1} in the case $a_{2}<a_{1}=a_{3}<a_{4},\,p_{1}<p_{2}<p_{3}<p_{4}$ only since other cases can be proved in a similar way. We can proof \eqref{eq:k4ext-1} in the same way in \cite{S25}. We first prove \eqref{eq:k4p0-1-1}. We have 
  \begin{align*}
    w_{\epsilon}(z)=x_{3}^{\epsilon}+e^{\pi(2-\alpha_{4}^{\epsilon}-\alpha_{1}^{\epsilon})}\dfrac{x_{4}-x_{3}}{_{2}F_{1}(\alpha_{3}^{\epsilon},1-\alpha_{1}^{\epsilon},1;z_{4,\epsilon})}\dfrac{1}{\Gamma(\alpha_{3}^{\epsilon})\Gamma(\alpha_{4}^{\epsilon})}\\
    \cdot\int_{0}^{z}\zeta^{\alpha_{3}^{\epsilon}-1}(\zeta-z_{4,\epsilon})^{\alpha_{4}^{\epsilon}-1}(\zeta-1)^{\alpha_{1}^{\epsilon}-1}\,d\zeta
  \end{align*}
  because of 
  \begin{align*}
    &\lim_{z\to z_{4,\epsilon}}\int_{0}^{z}\zeta^{\alpha_{3}^{\epsilon}-1}(\zeta-z_{4,\epsilon})^{\alpha_{4}^{\epsilon}-1}(\zeta-1)^{\alpha_{1}^{\epsilon}-1}\,d\zeta\\
    &\hspace{3cm}=e^{\alpha_{4}^{\epsilon}+\alpha_{1}^{\epsilon}-2}\Gamma(\alpha_{3}^{\epsilon})\Gamma(\alpha_{4}^{\epsilon})\,_{2}F_{1}(\alpha_{3}^{\epsilon},1-\alpha_{1}^{\epsilon},1;z_{4,\epsilon}).
  \end{align*}
  When $z\in\overline{\mathbb{H}}$ satisfies $z\in D(\delta)$ and $\left\lvert z\right\rvert\leq 1$, we obtain
  \begin{equation}
    \label{eq4.1-1}
    \begin{aligned}
      &\left\lvert w_{\epsilon}(z)-p_{2}\right\rvert\\
      &<\left\lvert x_{3}^{\epsilon}-p_{2}+e^{\pi(1+\alpha_3^\epsilon-\alpha_{1}^{\epsilon})}\dfrac{x_{4}-x_{3}}{_{2}F_{1}(\alpha_{3}^{\epsilon},1-\alpha_{1}^{\epsilon},1;z_{4,\epsilon})}\dfrac{1}{\Gamma(\alpha_{3}^{\epsilon})\Gamma(1-\alpha_{3}^{\epsilon})}\right.\\
      &\hspace{7cm}\left.\cdot\int_{0}^{\delta z/\left\lvert z\right\rvert }\zeta^{\alpha_{3}^{\epsilon}-1}(\zeta-z_{4,\epsilon})^{-\alpha_{3}^{\epsilon}}(\zeta-1)^{\alpha_{1}^{\epsilon}-1}\,d\zeta\right\rvert\\
      &+\left\lvert \dfrac{x_{4}-x_{3}}{_{2}F_{1}(\alpha_{3}^{\epsilon},1-\alpha_{1}^{\epsilon},1;z_{4,\epsilon})}\dfrac{1}{\Gamma(\alpha_{3}^{\epsilon})\Gamma(1-\alpha_{3}^{\epsilon})}\right.\\
      &\hspace{6cm}\left.\cdot\int_{\delta z/\left\lvert z\right\rvert }^{z}\zeta^{\alpha_{3}^{\epsilon}-1}(\zeta-z_{4,\epsilon})^{-\alpha_{3}^{\epsilon}}(\zeta-1)^{\alpha_{1}^{\epsilon}-1}\,d\zeta\right\rvert.
    \end{aligned}
  \end{equation}
  By Lemma \ref{sqr-rep1}, we have
  \begin{align*}
    &x_{3}^{\epsilon}+e^{\pi(1+\alpha_{3}^{\epsilon}-\alpha_{1}^{\epsilon})}\dfrac{x_{4}-x_{3}}{_{2}F_{1}(\alpha_{3}^{\epsilon},1-\alpha_{1}^{\epsilon},1;z_{4,\epsilon})}\dfrac{1}{\Gamma(\alpha_{3}^{\epsilon})\Gamma(1-\alpha_{3}^{\epsilon})}\\
    &\hspace{7cm}\cdot\int_{0}^{\delta z/\left\lvert z\right\rvert }\zeta^{\alpha_{3}^{\epsilon}-1}(\zeta-z_{4,\epsilon})^{-\alpha_{3}^{\epsilon}}(\zeta-1)^{\alpha_{1}^{\epsilon}-1}\,d\zeta\\
    =&x_{3}+\dfrac{\sin\pi\alpha_3}{\pi}\dfrac{x_{4}-x_{3}}{_{2}F_{1}(\alpha_{3},1-\alpha_{1},1;\xi)}\left[-e^{\pi\alpha_3i}\sum_{\substack{k,l\geq0\\k\neq l}}\dfrac{(\alpha)_k(1-\alpha_1)_l}{k!l!}\dfrac{1}{k-l}\left(\dfrac{\delta z}{\xi\left\lvert z\right\rvert}\right)^{-k}\left(\delta\dfrac{z}{\left\lvert z\right\rvert}\right)^l\right.\\
    &\left.\hspace{1.5cm}+e^{\pi\alpha_3i}\sum_{n\geq0}\left[\psi(n+1)-\psi(\alpha_3+n)+\log\left(-\dfrac{\delta z}{\xi\left\lvert z\right\rvert }\right)\right]\dfrac{(\alpha_3)_n(1-\alpha_1)_n}{n!n!}\xi^n\right].
  \end{align*}
  Hence, we obtain
  \begin{equation}
    \label{eq4.1-2}
    \begin{aligned}
      &\left\lvert x_{3}^{\epsilon}-p_{2}+e^{\pi(1+\alpha_{3}^{\epsilon}-\alpha_{1}^{\epsilon})}\dfrac{x_{4}-x_{3}}{_{2}F_{1}(\alpha_{3}^{\epsilon},1-\alpha_{1}^{\epsilon},1;z_{4,\epsilon})}\dfrac{1}{\Gamma(\alpha_{3}^{\epsilon})\Gamma(1-\alpha_{3}^{\epsilon})}\right.\\
      &\hspace{7cm}\left.\cdot\int_{0}^{\delta z/\left\lvert z\right\rvert }\zeta^{\alpha_{3}^{\epsilon}-1}(\zeta-z_{4,\epsilon})^{-\alpha_{3}^{\epsilon}}(\zeta-1)^{\alpha_{1}^{\epsilon}-1}\,d\zeta\right\rvert\\
      \leq&\left\lvert x_3^\epsilon-p_2+\dfrac{\sin\pi\alpha_3}{\pi}\dfrac{x_{4}-x_{3}}{_{2}F_{1}(\alpha_{3},1-\alpha_{1},1;\xi)}e^{\pi\alpha_3i}\right.\\
      &\left.\hspace{2cm}\cdot\sum_{n\geq0}\left[\psi(n+1)-\psi(\alpha_3+n)-\log\xi\right]\dfrac{(\alpha_3)_n(1-\alpha_1)_n}{n!n!}\xi^n\right\rvert\\
      &\hspace{1cm}+\left\lvert \dfrac{\sin\pi\alpha_3}{\pi}\dfrac{x_{4}-x_{3}}{_{2}F_{1}(\alpha_{3},1-\alpha_{1},1;\xi)}\sum_{\substack{k,l\geq0\\k\neq l}}\dfrac{(\alpha_3)_k(1-\alpha_1)_l}{k!l!}\dfrac{1}{k-l}\left(\dfrac{\delta z}{\xi\left\lvert z\right\rvert}\right)^{-k}\left(\delta\dfrac{z}{\left\lvert z\right\rvert}\right)^l\right\rvert\\
      &\hspace{2cm}+\left\lvert \dfrac{\sin\pi\alpha_3}{\pi}(x_{4}-x_{3})\log\left(-\dfrac{\delta z}{\xi\left\lvert z\right\rvert }\right)\right\rvert
    \end{aligned}
  \end{equation}
  We here denote 
  \begin{align*}
    M_{0,\epsilon,1}\coloneqq&\left\lvert x_3^\epsilon-p_2+\dfrac{\sin\pi\alpha_3}{\pi}\dfrac{x_{4}-x_{3}}{_{2}F_{1}(\alpha_{3},1-\alpha_{1},1;\xi)}e^{\pi\alpha_3i}\right.\\
    &\left.\hspace{2cm}\cdot\sum_{n\geq0}\left[\psi(n+1)-\psi(\alpha_3+n)-\log\xi\right]\dfrac{(\alpha_3)_n(1-\alpha_1)_n}{n!n!}\xi^n\right\rvert.
  \end{align*}
  One has
  \begin{equation}
    \label{eq4.1-3}
    \begin{aligned}
      &\left\lvert \sum_{\substack{k,l\geq0\\k\neq l}}\dfrac{(\alpha_3)_k(1-\alpha_1)_l}{k!l!}\dfrac{1}{k-l}\left(\dfrac{\delta z}{\xi\left\lvert z\right\rvert}\right)^{-k}\left(\delta\dfrac{z}{\left\lvert z\right\rvert}\right)^l\right\rvert\leq\sum_{\substack{k,l\geq0\\k\neq l}}\dfrac{(\alpha_3)_k(1-\alpha_1)_l}{k!l!}\left(\dfrac{\delta}{\xi}\right)^{-k}\delta^l\\
      &\hspace{4cm}=F_1\left(\alpha_3,1-\alpha_1,1,1;\dfrac{\xi}{\delta},\delta\right)-\ _2F_1(\alpha_3,1-\alpha_1,1;\xi).
    \end{aligned}
  \end{equation}
  We here put 
  \[
    M_{0,\epsilon,2}\coloneqq\left\lvert \dfrac{\sin\pi\alpha_3}{\pi}\dfrac{x_{4}-x_{3}}{_{2}F_{1}(\alpha_{3},1-\alpha_{1},1;\xi)}\left[F_1\left(\alpha_3,1-\alpha_1,1,1;\dfrac{\xi}{\delta},\delta\right)-\ _2F_1(\alpha_3,1-\alpha_1,1;\xi)\right]\right\rvert.
  \]
  We also have
  \begin{equation}
    \label{eq4.1-4}
    \begin{aligned}
      \left\lvert \dfrac{\sin\pi\alpha_3}{\pi}(x_{4}-x_{3})\log\left(-\dfrac{\delta z}{\left\lvert z\right\rvert }\right)\right\rvert&=\left\lvert \dfrac{\sin\pi\alpha_3}{\pi}(x_{4}-x_{3})\left(\log\delta+\log\left(-\dfrac{z}{\left\lvert z\right\rvert }\right)\right)\right\rvert\\
      &\leq\left\lvert \dfrac{\sin\pi\alpha_3}{\pi}(x_{4}-x_{3})\right\rvert\cdot\left\lvert \log\delta-\pi i\right\rvert.
    \end{aligned}
  \end{equation}
  We here put
  \[
    M_{0,\epsilon,3}\coloneqq\left\lvert \dfrac{\sin\pi\alpha_3}{\pi}(x_{4}-x_{3})\right\rvert\cdot\left\lvert \log\delta-\log\xi-\pi i\right\rvert.
  \]
  We also have 
  \begin{equation}
    \label{eq4.1-5}
    \begin{aligned}
      \int_{\delta z/\left\lvert z\right\rvert }^{z}\zeta^{\alpha_{3}^{\epsilon}-1}(\zeta-z_{4,\epsilon})^{-\alpha_{3}^{\epsilon}}(\zeta-1)^{\alpha_{1}^{\epsilon}-1}\,d\zeta<\,&\int_{\delta}^{\left\lvert z\right\rvert}s^{\alpha_{3}^{\epsilon}-1}\left\lvert \dfrac{z}{\left\lvert z\right\rvert}s-z_{4,\epsilon}\right\rvert^{-\alpha_{3}^{\epsilon}}\left\lvert \dfrac{z}{\left\lvert z\right\rvert}s-1\right\rvert^{\alpha_{1}^{\epsilon}-1}\,ds\\
      <\,&\delta^{\alpha_{3}^{\epsilon}-1}(\delta-z_{4,\epsilon})^{-\alpha_{3}^{\epsilon}}\delta^{\alpha_{1}^{\epsilon}-1}(\left\lvert z\right\rvert-\delta)\\
      <\,&\delta^{\alpha_{3}^{\epsilon}-1}(\delta-z_{4,\epsilon})^{-\alpha_{3}^{\epsilon}}\delta^{\alpha_{1}^{\epsilon}-1}(1-\delta).
    \end{aligned}
  \end{equation}
  By \eqref{eq4.1-1}-\eqref{eq4.1-5}, we obtain
  \begin{equation}
    \label{eq4.1-6}
    \begin{aligned}
      \left\lvert w_{\epsilon}(z)-p_{1}\right\rvert&\,\leq M_{0,\epsilon,1}+M_{0,\epsilon,2}+M_{0,\epsilon,3}\\
      &\hspace{2cm}+M_{0,\epsilon,4}\delta^{\alpha_{3}^{\epsilon}-1}(\delta-z_{4,\epsilon})^{\alpha_{4}^{\epsilon}-1}\delta^{\alpha_{1}^{\epsilon}-1}(1-\delta),
    \end{aligned}
  \end{equation}
  where
  \[
    M_{0,\epsilon,4}\coloneqq\left\lvert \dfrac{x_{4}-x_{3}}{_{2}F_{1}(\alpha_{3}^{\epsilon},1-\alpha_{1}^{\epsilon},\alpha_{3}^{\epsilon}+\alpha_{4}^{\epsilon};z_{4,\epsilon})}\dfrac{\sin\pi\alpha_3}{\pi}\right\rvert.
  \]
  By Corollary \ref{angleslim}, we have $M_{0,\epsilon,j}\to0$ as $\epsilon\to+0$ for $j=2,3,4$. One has
  \begin{align*}
    M_{0,\epsilon,1}=&\left\lvert x_3^\epsilon-p_2-\log\xi\cdot\dfrac{\sin\pi\alpha_3}{\pi}(x_{4}-x_{3})e^{\pi\alpha_3i}\right.\\
    &\left.+\dfrac{\sin\pi\alpha_3}{\pi}\dfrac{x_{4}-x_{3}}{_{2}F_{1}(\alpha_{3},1-\alpha_{1},1;\xi)}e^{\pi\alpha_3i}\sum_{n\geq0}\left[\psi(n+1)-\psi(\alpha_3+n)\right]\dfrac{(\alpha_3)_n(1-\alpha_1)_n}{n!n!}\xi^n\right\rvert.
  \end{align*}
  By Lemma \ref{anglesineq}, we obtain 
  \[
    \lim_{\epsilon\to+0}\dfrac{\sin\pi\alpha_3}{\pi}\dfrac{x_{4}-x_{3}}{_{2}F_{1}(\alpha_{3},1-\alpha_{1},1;\xi)}e^{\pi\alpha_3i}\sum_{n\geq0}\left[\psi(n+1)-\psi(\alpha_3+n)\right]\dfrac{(\alpha_3)_n(1-\alpha_1)_n}{n!n!}\xi^n=0
  \]
  and 
  \[
    \lim_{\epsilon\to+0}\dfrac{\sin\pi\alpha_3}{\epsilon}=\lim_{\epsilon\to+0}\dfrac{\sin\pi(1-\alpha_3)}{\pi(1-\alpha_3)}\dfrac{\pi(1-\alpha_3)}{\epsilon}=a_4-a_1.
  \]
  By Lemma \ref{length-4thpt}, we also have
  \[
    \lim_{\epsilon\to+0}\left(-\dfrac{\epsilon}{\pi}\log\xi\right)= \dfrac{p_3-p_2}{b_1-b_3}.
  \]
  This leads to the following:
  \[
    \lim_{\epsilon\to+0}\left(-\log\xi\cdot\dfrac{\sin\pi\alpha_3}{\pi}(x_{4}-x_{3})\right)=\dfrac{p_3-p_2}{b_1-b_3}\cdot(a_4-a_1)(p_4-p_3)=p_3-p_2.
  \]
  Therefore we obtain $M_{0,\epsilon,1}\to0\ (\epsilon\to+0)$, and we get \eqref{eq:k4p0-1-1}. On the other hand, we have 
  \begin{align*}
    w_{\epsilon}(z)&=x_{2}^{\epsilon}-\dfrac{x_{1}^{\epsilon}-x_{2}^{\epsilon}}{_{2}F_{1}(1-\alpha_{1}^{\epsilon},\alpha_{3}^{\epsilon},1;z_{4,\epsilon})}\dfrac{1}{\Gamma(\alpha_{1}^{\epsilon})\Gamma(1-\alpha_{1}^{\epsilon})}\\
    &\hspace{4cm}\cdot\int_{\infty}^{z}\zeta^{\alpha_{3}^{\epsilon}-1}(\zeta-z_{4,\epsilon})^{-\alpha_{3}^{\epsilon}}(\zeta-1)^{\alpha_{1}^{\epsilon}-1}\,d\zeta
  \end{align*}
  because of 
  \begin{align*}
    \lim_{z\to1}\int_{\infty}^{z}\zeta^{\alpha_{3}^{\epsilon}-1}(\zeta-z_{4,\epsilon})^{-\alpha_{3}^{\epsilon}}(\zeta-1)^{\alpha_{1}^{\epsilon}-1}\,d\zeta=-\Gamma(\alpha_{1}^{\epsilon})\Gamma(1-\alpha_{1}^{\epsilon})\,_{2}F_{1}(1-\alpha_{1}^{\epsilon},\alpha_{3}^{\epsilon},1;z_{4,\epsilon}).
  \end{align*}
  We next discuss the case $z\in D(\delta)$ and $\left\lvert z\right\rvert\geq1$. In this case, we can prove \eqref{eq:k4p0-1-1} as same as the proof in \cite{S25}. Then we obtain \eqref{eq:k4p0-1-1}. We can also prove \eqref{eq:k4p0-1-2} in a similar way by using Lemma \ref{sqr-rep2}. We next prove \eqref{eq:k4ext-1}, but we can prove it in the same way as in \cite{S25}. We then prove \eqref{eq:k4int-1}. We first have
  \begin{align*}
    w_{\epsilon}(z)&=x_{3}+\dfrac{\sin\pi\alpha_3}{\pi}\dfrac{x_{4}-x_{3}}{_{2}F_{1}(\alpha_{3},1-\alpha_{1},1;\xi)}\left[-e^{\pi\alpha_3i}\sum_{\substack{k,l\geq0\\k\neq l}}\dfrac{(\alpha)_k(1-\alpha_1)_l}{k!l!}\dfrac{1}{k-l}\left(\dfrac{\delta z}{\xi\left\lvert z\right\rvert}\right)^{-k}\left(\delta\dfrac{z}{\left\lvert z\right\rvert}\right)^l\right.\\
    &\left.\hspace{1.5cm}+e^{\pi\alpha_3i}\sum_{n\geq0}\left[\psi(n+1)-\psi(\alpha_3+n)+\log\left(-\dfrac{\delta z}{\xi\left\lvert z\right\rvert }\right)\right]\dfrac{(\alpha_3)_n(1-\alpha_1)_n}{n!n!}\xi^n\right].
  \end{align*}
  when $z_{4,\epsilon}/\delta<\left\lvert z\right\rvert<\delta$ and $\left\lvert \arg(-z/\xi)\right\rvert<\pi$ holds. By the definition of gradient trees, we also have $I_{int}(\tau)=p_{2}-(b_{3}-b_{1})\tau$. We thus obtain
  \begin{equation}
    \label{eq4.4-1}
    \begin{aligned}
      &\left\lvert w_{\epsilon}(z)-I_{int}(\epsilon\tau)\right\rvert\\
      &\leq\left\lvert x_3^\epsilon-p_2+\dfrac{\sin\pi\alpha_3}{\pi}\dfrac{x_{4}-x_{3}}{_{2}F_{1}(\alpha_{3},1-\alpha_{1},1;\xi)}e^{\pi\alpha_3i}\right.\\
      &\hspace{2cm}\left.\cdot\sum_{n\geq0}\left[\psi(n+1)-\psi(\alpha_3+n)-\log\xi\right]\dfrac{(\alpha_3)_n(1-\alpha_1)_n}{n!n!}\xi^n\right\rvert\\
      &\hspace{1.5cm}+\left\lvert \dfrac{\sin\pi\alpha_3}{\pi}\dfrac{x_{4}-x_{3}}{_{2}F_{1}(\alpha_{3},1-\alpha_{1},1;\xi)}\sum_{\substack{k,l\geq0\\k\neq l}}\dfrac{(\alpha_3)_k(1-\alpha_1)_l}{k!l!}\dfrac{1}{k-l}\left(\dfrac{\delta z}{\xi\left\lvert z\right\rvert}\right)^{-k}\left(\delta\dfrac{z}{\left\lvert z\right\rvert}\right)^l\right\rvert\\
      &\hspace{3cm}+\left\lvert \dfrac{\sin\pi\alpha_3}{\pi}(x_{4}-x_{3})e^{\pi\alpha_3i}\log\left(-\dfrac{\delta z}{\left\lvert z\right\rvert }\right)+(b_{3}-b_{1})\epsilon\tau\right\rvert
    \end{aligned}
  \end{equation}
  This is equivalent to
  \begin{align*}
    \left\lvert w_{\epsilon}(z)-I_{int}(\epsilon\tau)\right\rvert\leq M_{0,\epsilon,1}+M_{0,\epsilon,2}+\left\lvert \dfrac{\sin\pi\alpha_3}{\pi}(x_{4}-x_{3})e^{\pi\alpha_3i}\log\left(-\dfrac{\delta z}{\left\lvert z\right\rvert }\right)+(b_{3}-b_{1})\epsilon\tau\right\rvert.
  \end{align*}
  When we put $z=\exp[-\pi\tau+i\pi(1-\sigma)]$, we obtain
  \begin{align*}
    &\left\lvert \frac{\sin\pi\alpha_3}{\pi}(x_{4}-x_{3})e^{\pi\alpha_3i}\log\left(-\dfrac{\delta z}{\left\lvert z\right\rvert }\right)+(b_{3}-b_{1})\epsilon\tau\right\rvert \\
    =&\left\lvert \frac{\sin\pi\alpha_3}{\pi}(x_{4}-x_{3})e^{\pi\alpha_3i}\left(\log\delta-\pi\tau-i\pi\sigma \right)+(b_{3}-b_{1})\epsilon\tau\right\rvert\\
    \leq&\left\lvert \frac{\sin\pi\alpha_3}{\pi}(x_{4}-x_{3})e^{\pi\alpha_3i}\right\rvert\left\lvert \log\delta-i\pi\sigma\right\rvert+\left\lvert -\dfrac{\sin\pi\alpha_3}{\epsilon}(x_{4}-x_{3})e^{\pi\alpha_3i}+(b_{3}-b_{1})\right\rvert\left\lvert \epsilon\tau\right\rvert\\
    \leq&\left\lvert \frac{\sin\pi\alpha_3}{\pi}(x_{4}-x_{3})e^{\pi\alpha_3i}\right\rvert\left\lvert \log\delta-i\pi\right\rvert\\
    &\hspace{4cm}+\left\lvert -\frac{\sin\pi\alpha_3}{\pi}(x_{4}-x_{3})e^{\pi\alpha_3i}+(b_{3}-b_{1})\right\rvert\left\lvert -\dfrac{\epsilon}{\pi}\log\xi+\dfrac{\epsilon}{\pi}\log\delta\right\rvert.
  \end{align*}
  By Corollary \ref{angleslim}, we have
  \begin{align*}
    \lim_{\epsilon\to+0}\left(-\frac{\sin\pi\alpha_3}{\pi}(x_{4}-x_{3})e^{\pi\alpha_3i}\right)&=(a_4-a_1)(p_4-p_3)\\
    &=(a_4-a_1)\left(-\dfrac{b_1-b_4}{a_1-a_4}+\dfrac{b_4-b_3}{a_4-a_1}\right)\\
    &=b_1-b_3.
  \end{align*}
  Since $M_{0,\epsilon,1},M_{0,\epsilon,2}\to0\ (\epsilon\to+0)$ hold, we obtain \eqref{eq:k4int-1}.
\subsection{Proof of Theorem \ref{k4corr2}.}
We can prove \eqref{eq:k4ext-2} in a similar way as Theorem \ref{k4corr1} \eqref{eq:k4ext-1}. We here only prove \eqref{eq:k4p0-2}. We here assume $p_4<p_1=p_3<p_2$ and $\max\{a_1,a_3\}<\max\{a_2,a_4\}$. From Lemma \ref{square-cm2}, we can take $\delta>0$ such that the segment between $\delta\cdot z/\left\lvert z\right\rvert $ and $z$ is included in $D_\epsilon(\delta)$ for
\[
  z\in D_\epsilon(\delta)\setminus(D_0(z_{4,\epsilon})\cup D_{z_{4,\epsilon}}(\min\{z_{4,\epsilon},1-z_{4,\epsilon}\})\cup D_1(1-z_{4,\epsilon})\cup D_\infty(1))(\eqqcolon D_{\epsilon}(\delta)')
\]
and the sufficiently small $\epsilon>0$. This fact can be proved by using elementary geometry. By applying a similar discussion in the proof of Theorem \ref{k4corr1} and using Lemma \ref{sqr-rep3}, there exists positive real numbers $M_{\epsilon,1},M_{\epsilon,2},M_{\epsilon,3},M_{\epsilon,4}$ which depend on $\epsilon$ and satisfy 
\begin{equation}
  \label{eq5.1-1}
  \begin{aligned}
    M_{\epsilon,1}\geq\max_{z\in D_\epsilon(\delta)\cap D_0(\delta)}\left\lvert w_\epsilon(z)-p_1\right\rvert,\ M_{\epsilon,2}\geq\max_{z\in D_\epsilon(\delta)\cap D_{z_{4,\epsilon}}(\min\{z_{4,\epsilon},1-z_{4,\epsilon}\})}\left\lvert w_\epsilon(z)-p_1\right\rvert,\\
    M_{\epsilon,3}\geq\max_{z\in D_\epsilon(\delta)\cap D_1(1-z_{4,\epsilon})}\left\lvert w_\epsilon(z)-p_1\right\rvert,\ M_{\epsilon,4}\geq\max_{z\in D_\epsilon(\delta)\cap D_\infty(1)}\left\lvert w_\epsilon(z)-p_1\right\rvert
  \end{aligned}
\end{equation}
and $M_{1,\epsilon},M_{2,\epsilon},M_{3,\epsilon},M_{4,\epsilon}\to0$ as $\epsilon\to+0$. When 
\[
  z\in D_\epsilon(\delta)\setminus(D_0(z_{4,\epsilon})\cup D_{z_{4,\epsilon}}(\min\{z_{4,\epsilon},1-z_{4,\epsilon}\})\cup D_1(1-z_{4,\epsilon})\cup D_\infty(1))(\eqqcolon D_{\epsilon}(\delta)')
\]
holds, the segment between $z$ and $\delta\cdot z/\left\lvert z\right\rvert$ are included in $\overline{\mathbb{H}}\setminus (D_0(\delta)\cup D_{z_{4,\epsilon}}(\delta)\cup D_1(\delta)\cup D_\infty(\delta))$. Then we have
\begin{equation}
  \begin{aligned}
    \left\lvert w_\epsilon(z)-p_1\right\rvert&\leq\left\lvert x_3^\epsilon-p_1+e^{\pi(2-\alpha_{4}^{\epsilon}-\alpha_{1}^{\epsilon})}\dfrac{x_{4}-x_{3}}{_{2}F_{1}(\alpha_{3}^{\epsilon},1-\alpha_{1}^{\epsilon},1;z_{4,\epsilon})}\dfrac{\Gamma(\alpha_3^\epsilon+\alpha_4^\epsilon)}{\Gamma(\alpha_{3}^{\epsilon})\Gamma(\alpha_{4}^{\epsilon})}\right.\\
    &\hspace{4cm}\left.\cdot\int_{0}^{\delta z/\left\lvert z\right\rvert}\zeta^{\alpha_3^\epsilon-1}(\zeta-z_{4,\epsilon})^{\alpha_4^{\epsilon}-1}(\zeta-1)^{\alpha_1^{\epsilon}-1}d\zeta\right\rvert\\
    &\hspace{1cm}+\left\lvert \dfrac{x_{4}-x_{3}}{_{2}F_{1}(\alpha_{3}^{\epsilon},1-\alpha_{1}^{\epsilon},1;z_{4,\epsilon})}\dfrac{\Gamma(\alpha_3^\epsilon+\alpha_4^\epsilon)}{\Gamma(\alpha_{3}^{\epsilon})\Gamma(\alpha_{4}^{\epsilon})}\right\rvert\\
    &\hspace{4cm}\cdot\left\lvert \int_{\delta z/\left\lvert z\right\rvert }^{z}\zeta^{\alpha_3^\epsilon-1}(\zeta-z_{4,\epsilon})^{\alpha_4^{\epsilon}-1}(\zeta-1)^{\alpha_1^{\epsilon}-1}d\zeta\right\rvert.
  \end{aligned}
\end{equation}
By applying a similar discussion in the proof of Theorem \ref{k4corr1}, we also have the positive real number $M_{5,\epsilon}$ which depends on $\epsilon$ such that
\begin{align*}
    M_{5,\epsilon}&\geq\left\lvert x_3^\epsilon-p_1+e^{\pi(2-\alpha_{4}^{\epsilon}-\alpha_{1}^{\epsilon})}\dfrac{x_{4}-x_{3}}{_{2}F_{1}(\alpha_{3}^{\epsilon},1-\alpha_{1}^{\epsilon},1;z_{4,\epsilon})}\dfrac{\Gamma(\alpha_3^\epsilon+\alpha_4^\epsilon)}{\Gamma(\alpha_{3}^{\epsilon})\Gamma(\alpha_{4}^{\epsilon})}\right.\\
    &\hspace{4cm}\left.\cdot\int_{0}^{\delta z/\left\lvert z\right\rvert}\zeta^{\alpha_3^\epsilon-1}(\zeta-z_{4,\epsilon})^{\alpha_4^{\epsilon}-1}(\zeta-1)^{\alpha_1^{\epsilon}-1}d\zeta\right\rvert
\end{align*}
and $M_{5,\epsilon}\to0$ as $\epsilon\to+0$. We also have 
\begin{align*}
  &\left\lvert \int_{\delta z/\left\lvert z\right\rvert }^{z}\zeta^{\alpha_3^\epsilon-1}(\zeta-z_{4,\epsilon})^{\alpha_4^{\epsilon}-1}(\zeta-1)^{\alpha_1^{\epsilon}-1}d\zeta\right\rvert\\
  &=\left\lvert \int_{\delta}^{\left\lvert z\right\rvert}\left(\dfrac{z}{\left\lvert z\right\rvert}t\right)^{\alpha_3^\epsilon-1}\left(\dfrac{z}{\left\lvert z\right\rvert}t-z_{4,\epsilon}\right)^{\alpha_4^{\epsilon}-1}\left(\dfrac{z}{\left\lvert z\right\rvert}t-1\right)^{\alpha_1^{\epsilon}-1}\dfrac{z}{\left\lvert z\right\rvert}\,dt\right\rvert\\
  &\leq\delta^{\alpha_3^\epsilon-1}\delta^{\alpha_4^\epsilon-1}\delta^{\alpha_1^\epsilon-1}(\left\lvert z\right\rvert-\delta)\leq\delta^{\alpha_3^\epsilon-1}\delta^{\alpha_4^\epsilon-1}\delta^{\alpha_1^\epsilon-1}(1-\delta).
\end{align*}
Since 
\[
  \lim_{\epsilon\to+0}\alpha_3^\epsilon=1,\ \lim_{\epsilon\to+0}\alpha_4^\epsilon=0
\]
holds by Corollary \ref{angleslim}, we obtain
\begin{equation}
  \label{eq5.1-2}
  M_{6,\epsilon}\geq\max_{z\in D_\epsilon(\delta)'}\left\lvert w_\epsilon(z)-p_1\right\rvert,\ \lim_{\epsilon\to+0}M_{6,\epsilon}=0,
\end{equation}
where we denote
\[
  M_{6,\epsilon}\coloneqq M_{5,\epsilon}+\delta^{\alpha_3^\epsilon+\alpha_4^\epsilon+\alpha_1^\epsilon-3}(1-\delta)\left\lvert \dfrac{x_{4}-x_{3}}{_{2}F_{1}(\alpha_{3}^{\epsilon},1-\alpha_{1}^{\epsilon},1;z_{4,\epsilon})}\dfrac{\Gamma(\alpha_3^\epsilon+\alpha_4^\epsilon)}{\Gamma(\alpha_{3}^{\epsilon})\Gamma(\alpha_{4}^{\epsilon})}\right\rvert.
\]
From \eqref{eq5.1-1} and \eqref{eq5.1-2}, we get
\[
  \max_{z\in D_\epsilon(\delta)}\left\lvert w_\epsilon(z)-p_1\right\rvert\leq\max\{M_{1,\epsilon},M_{2,\epsilon},M_{3,\epsilon},M_{4,\epsilon},M_{6,\epsilon}\}
\]
and \eqref{eq:k4p0-2} holds.

\bibliographystyle{plain} 

\end{document}